%% file: main_arxiv.tex
\documentclass[11pt]{amsart}
\usepackage[utf8]{inputenc}
\usepackage{amssymb,amsmath}
\usepackage{graphicx}
\usepackage{color}
\usepackage{tikz}
\usepackage{pgfplots}
\usepackage{enumitem}
\usepackage{algorithm}
\usepackage{setspace}
\usepackage{algpseudocode}
\usepackage{mathtools}
\usepackage{xfrac}

\usepackage{accents}
\usepackage{multicol} 
\usepackage{soul}
\usepackage{subcaption}
\usepackage{booktabs}
\usepackage{array}
\usepackage{comment}
\newcolumntype{L}[1]{>{\raggedright\let\newline\\\arraybackslash\hspace{0pt}}m{#1}}
\newcolumntype{C}[1]{>{\centering\let\newline\\\arraybackslash\hspace{0pt}}m{#1}}
\newcolumntype{R}[1]{>{\raggedleft\let\newline\\\arraybackslash\hspace{0pt}}m{#1}}
\usepackage{siunitx}
\usepackage{bbold}

\usepackage[hidelinks]{hyperref}
\usepackage[nameinlink,noabbrev]{cleveref}
\newtheorem{theorem}{Theorem}

\theoremstyle{definition}

\newtheorem{example}[theorem]{Example}
\theoremstyle{lemma}
\newtheorem{lemma}[theorem]{Lemma}

\theoremstyle{remark}
\newtheorem{remark}[theorem]{Remark}

\newtheorem{assumption}[theorem]{Assumption}
\Crefname{assumption}{Assumption}{Assumptions}
\numberwithin{theorem}{section}
\numberwithin{equation}{section}
\numberwithin{table}{section}
\numberwithin{figure}{section}
\input{def}
\allowdisplaybreaks
\begin{document}
\title[Semi-explicit Time Discretization for Linear Thermo--poroelasticity]{Semi-explicit Time Discretization for\\ Linear Thermo--poroelasticity}
\author[]{R.~Altmann$^\dagger$, R.~Maier$^\ddagger$, J.~Schmeck$^\dagger$}
\address{${}^{\dagger}$ Institute of Analysis and Numerics, Otto von Guericke University Magdeburg, Universit\"atsplatz 2, 39106 Magdeburg, Germany}
\email{\{robert.altmann,jochewed.schmeck\}@ovgu.de}
\address{${}^{\ddagger}$ Institute for Applied and Numerical Mathematics, Karlsruhe Institute of Technology, Englerstr.~2, 76131 Karlsruhe, Germany}
\email{roland.maier@kit.edu}
%
\date{\today}
%
%
\begin{abstract}
Within this paper, we introduce partially and fully decoupled time stepping schemes for linear thermo--poroelasticity. This means that the mechanics, heat, and flow equations can be solved sequentially. We provide sufficient conditions on the material parameters, which can be checked a priori, guaranteeing first-order convergence of the introduced schemes. Hence, the proposed methods have the same order as the implicit Euler scheme but are computationally more efficient due to the decoupling of the system equations. Numerical examples validate the proven convergence results and analyze the sharpness of the mentioned parameter condition. Further, we compare the schemes with other decoupling schemes from the literature.
\end{abstract}
%
%
\maketitle
\setcounter{tocdepth}{3}
%
%
%
%
\section{Introduction}
This paper deals with the numerical treatment of thermo--poroelasticity, i.e., poroelastic applications where the influence of the temperature is taken into account such as in human geological activities. Here, the temperature has a great impact on the physical phenomena, which are the basis for CO$_2$ sequestration and geothermal energy production~\cite{AntBB23}. In geothermal engineering, the heat energy that is stored inside the earth is used to produce energy~\cite{Kim18}.  

In this paper, we consider the linear thermo--poroelastic model derived in~\cite{BruBNR18}. Starting from a micro-level model that is based on fluid--structure interaction as well as conservation principles for the solid and fluid phases, a two-scale expansion is used to derive a macroscopic homogenized model, involving also a nonlinear convective term. However, for small P\'eclet numbers (which is satisfied for high thermal conductivity or low flow velocity), this term can be neglected, resulting in a system of linear partial differential equations (PDEs) for the displacement of the solid, the pressure of the fluid, and the temperature as unknown variables. A theoretical study of the well-posedness for a corresponding six-field formulation, including the linear as well as the nonlinear case, is given in~\cite{BruANR19}. 
%
A similar homogenization idea as in~\cite{BruBNR18} has also been used in~\cite{DuiMWW19} to derive a thermo--poroelastic macro-model. Moreover, other approaches to obtain suitable macro-equations are presented, e.g., in~\cite{LeeM97} (which, however, relies on restrictive assumptions on the allowed deformation) and~\cite{BriBPR16} (where the matrix is assumed to be rigid). Generally, a complete study of the constitutive laws for the total stress tensor, the rate of change of fluid mass, and the rate of change of the energy is presented in~\cite[Ch.~4]{Cou04}. For an extension of the model to poroelastoplasticity with thermal effects, we refer to~\cite{Cou89}, whereas the fully dynamic model is considered in~\cite{BonBMA23}.  

In this work, we focus on the time discretization of the linear model. While classical (implicit) time discretizations lead to a fully coupled system of equations, our goal is to decouple the equations with appropriate semi-explicit schemes under suitable coupling conditions. These conditions are only based on the physical parameters and do not present a CFL-type time step restriction. More precisely, we aim at extending the semi-explicit scheme introduced in~\cite{AltMU21} (see also~\cite{AltM22,AltMU24}) for poroelasticity problems to thermo--poroelasticity. The main idea to obtain decoupled equations is to introduce a certain time-lagging related to the chosen time step in parts of the equations, resulting in a system of delay differential equations on the analytical level. The analysis of the semi-explicit schemes is based on the closeness of the delay equations to the original system of PDEs in combination with the analysis of a classical implicit time discretization scheme for the upcoming delay equations. Based on a splitting as in~\cite{KolVV14}, 
a similar decoupling approach is presented in~\cite{KolV17} but without any convergence analysis. 

Besides our approach, there exist multiple other strategies to decouple the equations of linear thermo--poroelasticity, mostly motivated by similar ideas used for poroelasticity, see~\cite{ArmS92,SimM92,KimTJ11a,KimTJ11b,MikW13}. These approaches have in common that -- in contrast to our strategy -- they are based on an additional inner iteration. Although no assumptions on the physical parameters are required, sufficiently large stabilization parameters are necessary. Moreover, convergence is only given up to a certain tolerance, which is fixed by the number of inner iteration steps. Actual examples are given in \cite{Kim18,BruABNR20,BalLY24} as well as~\cite{CaiLLL25} for the corresponding four-field formulation (similar to the three-field formulation in poroelasticity). 

The paper is structured as follows. In Section~\ref{sec:model}, we introduce the model of linear thermo--poroelasticity and its weak formulation. Afterwards, we consider a parabolic delay equation and discuss its discretization by the implicit Euler scheme in Section~\ref{sec:convergence}. This will be the basis for the convergence analysis for the half and fully decoupled schemes in Sections~\ref{sec:halfDecoupled} and~\ref{sec:fullyDecoupled}, respectively. Finally, we present several numerical experiments in Section~\ref{sec:num}. Here, we numerically analyze the assumptions on the physical parameters and compare our schemes to implicit and iterative methods. 

\subsection*{Notation} 
Within this paper, we use typical Matlab notation for vectors, i.e., we write $[x;y]$ for $[x^\top, y^\top]^\top$ with $x\in \R^n$, $y\in\R^m$. 
Moreover, for a Hilbert space $\mathcal{H}$ and the final time $T$, we abbreviate the Bochner spaces $L^\infty(0,T;\mathcal{H})$ and $W^{k,\infty}(0,T;\mathcal{H})$ by $L^\infty(\mathcal{H})$ and $W^{k,\infty}(\mathcal{H})$, respectively.
%
%
\section{Thermo--poroelasticity}\label{sec:model}
In this section, we introduce the model equations for the problem of linear thermo--poroelasticity, modeling the interplay of fluid flow, displacement of a surrounding porous material, and the temperature. As a simplifying assumption (which is, however, not necessary), we assume zero Dirichlet boundary conditions throughout this work. Further, we work with weak formulations only. For this, we introduce the Hilbert spaces 
\[
	\V \coloneqq [H^1_0(\Omega)]^d, \qquad
	\Q \coloneqq H^1_0(\Omega),  \qquad
	\cHV \coloneqq [L^2(\Omega)]^d,\qquad
	\cHQ \coloneqq  L^2(\Omega)
\]
including homogeneous Dirichlet boundary conditions, such that $\V, \cHV, \V^*$ and  $ \Q, \cHQ, \Q^*$ form Gelfand triples.

In the following subsection, we first consider the problem of linear poroelasticity (i.e., the coupling of displacement and fluid flow without the inclusion of temperature) before introducing the weak formulation of thermo--poroelasticity in Section~\ref{sec:model:thermo}. We later rephrase the latter model into a form that resembles poroelasticity, such that these considerations are an important basis.

%
\subsection{Classical poroelasticity}\label{sec:model:poro}
For the abstract formulation of the linear poroelasticity model, we introduce the bilinear forms~$a\colon\V\times\V\to\R$, $b\colon\Q\times\Q\to\R$, $c\colon\cHQ\times\cHQ\to\R$, and $d\colon\V\times\cHQ\to\R$ by 
\begin{align*}
	a(u,v) \coloneqq \int_\Omega \sigma(u) : \varepsilon(v) \dx, \qquad 
	&b(p,q) \coloneqq \int_\Omega \frac{\kappa}{\nu}\, \nabla p \cdot \nabla q \dx,\\
	c(p,q) \coloneqq \int_\Omega c_0\, p\, q \dx, \qquad 
	&d(u,q) \coloneqq \int_\Omega \alpha\, (\nabla \cdot u)\, q \dx 
\end{align*}
for $u,\, v \in \V$ and $p,\, q \in \Q$. Therein, $\sigma$ equals the stress tensor, which is defined through the Lamé coefficients $\lambda$ and $\mu$ as well as the symmetric gradient $\varepsilon(v) \coloneqq \frac{1}{2}\, (\nabla v + (\nabla v)^\top )$. That is, $\sigma(v) =2\mu \varepsilon(v) +\lambda \trace(\varepsilon(v))I$. Further, we use the classical double dot notation for matrices $A, B \in \R^{n\times m}$, i.e., $A : B = \trace(A^\top B)$. The physical constants are given by
\begin{itemize}[itemsep=0.2em]
	\item[] $\kappa$ -- permeability,
	\item[] $\nu$ -- fluid viscosity,
	\item[] $c_0$ -- inverse of the Biot modulus (often denoted by~$M$),
	\item[] $\alpha$ -- Biot--Willis fluid--solid coupling coefficient. 
\end{itemize}
The weak formulation of linear poroelasticity seeks abstract functions~$u\colon [0,T] \to \V$ for the displacement and $p\colon [0,T] \to \Q$ for the pressure such that
\begin{subequations}
\label{eq:poroBilinear}
	\begin{align}
		a(u, v) - d(v, p) 
		&= \langle\f, v\rangle, \label{eq:poroBilinear:a} \\
		d(\dot u, q) + c(\dot p, q) + b(p, q) 
		&= \langle\g, q\rangle \label{eq:poroBilinear:b} 
	\end{align}
\end{subequations}	
for all test functions $v \in \V$, $q \in \Q$. Introducing the corresponding operators~$\calA$, $\calB$, $\calC$, and~$\calD$, defined through $a$, $b$, $c$, and~$d$, respectively, system~\eqref{eq:poroBilinear} can also be written in the more compact operator form
\begin{subequations}
	\label{eq:poroOperator}
	\begin{align}
		\calA u - \calD^*p  
		&= \f, \label{eq:poroOperator:a} \\
		\calD \dot u + \calC \dot p + \calB p 
		&= \g. \label{eq:poroOperator:b} 
	\end{align}
\end{subequations}	
As initial condition, we consider $p(0)=p^0$ given, which then also determines $u(0)$ by~\eqref{eq:poroOperator:a}. In the case of poroelasticity, the operators have the following properties.
\begin{assumption}[operators I]
\label{ass:operatorsI}
The operators $\calA, \calB, \calC$, and $\calD$ satisfy
\begin{itemize}[itemsep=0.2em]
	\item $\calA\colon \V\to\V^*$ is continuous with constant $C_a$, symmetric, elliptic with constant $c_a$,
	\item $\calB\colon\Q\to\Q^*$ is continuous with constant $C_b$, symmetric, elliptic with constant $c_b$, 
	\item $\calC\colon\cHQ\to\cHQdual$ is continuous with constant $C_c$, symmetric, elliptic with constant $c_c$, 
	\item $\calD\colon\V\to\cHQdual$ is continuous with constant $C_d$.
\end{itemize}
\end{assumption}
\begin{remark}
Since $c$ equals -- up to the constant $c_0$ -- the $L^2$-inner product, we have $C_c = c_c = c_0$.    			
\end{remark}
Results on the unique solvability of the system in terms of a strong solution are shown in~\cite{Sho00}. For weak solutions, we have the following result.
\begin{theorem}[weak solution, poroelasticity]
\label{thm:existencePoro}
Let Assumption~\ref{ass:operatorsI} hold and consider initial data~$p(0)=p^0\in \cHQ$ together with right-hand sides~$\f\in H^1(0,T;\V^*)$, $\g\in L^2(0,T;\Q^*)$. Then there exists a unique (weak) solution pair
\[
	u \in L^2(0,T;\V), \qquad
	p \in L^2(0,T;\Q) \cap H^1(0,T;\Q^*).
\]
\end{theorem}
\begin{proof}
With~\eqref{eq:poroOperator:a} we first eliminate the variable $u = \calA^{-1}(f+\calD^*p)$, leading to 
\[
	\calM \dot p + \calB p 
	= \tilde g
	\coloneqq g - \calD\calA^{-1} \dot f
\]
with $\calM \coloneqq \calC + \calD\calA^{-1}\calD^*$. Due to the assumptions on $\calA$ and $\calC$, also $\calM$ defines an inner product on $\cHQ$. Moreover, $\calB$ is elliptic, and the right-hand side satisfies~$\tilde g \in L^2(0,T;\Q^*)$. Hence, we have a parabolic equation which yields a unique solution~$p$ in the asserted space; see~\cite[Ch.~3, Sect.~4.4]{LioM72}. Equation~\eqref{eq:poroOperator:a} further implies that~$u \in L^2(0,T;\V)$.
\end{proof}
%
%
\subsection{Linear thermo--poroelasticity}\label{sec:model:thermo}

We now add a third variable to the system, modeling the temperature. Following~\cite{Bru19,BruANR19} restricted to the linear case, we introduce additional bilinear forms~$\bb\colon\Q\times\Q\to\R$, $\cc, \ccc\colon\cHQ\times\cHQ\to\R$ and $\dd\colon\V\times\cHQ\to\R$ defined by
\begin{align*}
	\bb(\theta,\zeta) 
	\coloneqq \int_\Omega \widetilde \kappa\, \nabla \theta \cdot \nabla \zeta \dx, \qquad
	&\dd(u,\zeta) 
	\coloneqq \int_\Omega \beta\, (\nabla \cdot u)\, \zeta \dx, \\  
	\cc(\theta,q) 
	\coloneqq \int_\Omega \cc_0\, \theta\, q \dx, \qquad
	&\ccc(\theta, \zeta) 
	\coloneqq \int_\Omega \ccc_0\, \theta\, \zeta \dx,
\end{align*}
where the physical constants are given by
\begin{itemize}[itemsep=0.2em]
	\item[] $\widetilde \kappa$ -- thermal conductivity, 
	\item[] $\beta$ -- thermal stress coefficient,
	\item[] $\cc_0$ -- thermal dilatation coefficient,
	\item[] $\ccc_0$ -- thermal capacity. 
\end{itemize}
The linear problem of thermo--poroelasticity in weak form then reads: seek the displacement~$u\colon [0,T] \to \V$, the pressure~$p\colon [0,T] \to \Q$, and the temperature~$\theta\colon [0,T] \to \Q$ such that
\begin{subequations}
\label{eq:thermoporoBilinear}
\begin{align}
	a(u, v) - d(v, p) - \dd(v, \theta) 
	&= \langle\f, v\rangle, \label{eq:thermoporoBilinear:a} \\
	d(\dot u, q) + c(\dot p, q) - \cc(\dot \theta,q) + b(p, q) 
	&= \langle\g, q\rangle, \label{eq:thermoporoBilinear:b} \\
	\dd(\dot u, \zeta) + \ccc(\dot \theta, \zeta) - \cc(\dot p, \zeta) + \bb(\theta, \zeta) 
	&= \langle\h, \zeta\rangle
	\label{eq:thermoporoBilinear:c} 
\end{align}
\end{subequations}	
for all test functions $v \in \V,\, q, \zeta \in \Q$. The initial conditions read~$p(0)=p^0$ and~$\theta(0)=\theta^0$, which again determines $u(0)$ by~\eqref{eq:thermoporoBilinear:a}. 
\begin{remark}
In order to model more general boundary conditions, one may also define different solution spaces for $p$ and $\theta$. Here, however, we only consider homogeneous Dirichlet boundary conditions such that $\Q$ is suitable as space for both variables. 
\end{remark}
As before, the bilinear forms define operators, which we denote by $\calBB, \calCC, \calCCC$, and $\calDD$, respectively. With these operators, system~\eqref{eq:thermoporoBilinear} can be written in the form
\begin{subequations}
\label{eq:thermoporoOperator}
\begin{align}
	\calA u - \calD^*p - \calDD^* \theta 
	&= \f, \label{eq:thermoporoOperator:a} \\
	\calD \dot u + \calC \dot p - \calCC \dot \theta + \calB p 
	&= \g, \label{eq:thermoporoOperator:b} \\
	\calDD \dot u + \calCCC \dot \theta - \calCC \dot p + \calBB \theta 
	&= \h. \label{eq:thermoporoOperator:c} 
\end{align}
\end{subequations}	
We summarize properties of the newly introduced operators in the following assumption. 
\begin{assumption}[operators II]
\label{ass:operatorsII}
The operators $\calBB, \calCC, \calCCC$, and $\calDD$ satisfy
\begin{itemize}[itemsep=0.2em]
	\item $\calBB\colon\Q\to\Q^*$ is continuous with constant $C_{\bb}$, symmetric, elliptic with constant $c_{\bb}$, 
	\item $\calCC\colon\cHQ\to\cHQdual$ is continuous with constant $C_{\cc}$, symmetric, elliptic with constant $c_{\cc}$, 
	\item $\calCCC\colon\cHQ\to\cHQdual$ is continuous with constant $C_{\ccc}$, symmetric, elliptic with constant $c_{\ccc}$, 	
	\item $\calDD\colon\V\to\cHQdual$ is continuous with constant $C_{\dd}$.
\end{itemize}
\end{assumption}
\begin{remark}
Since $\cc$ and $\ccc$ are again $L^2$-inner products only scaled by a constant, we have $C_{\cc} = c_{\cc} = \cc_0$ and $C_{\ccc} = c_{\ccc} = \ccc_0$. 
\end{remark}
As in~\cite{BruABNR20} we make the following assumption.  
\begin{assumption}[ellipticity constants]
\label{ass:chat}
The constants $c_0$, $\cc_0$, and $\ccc_0$ satisfy
\[
	\cc_0 
	< \ccc_0
	\qquad\text{as well as}\qquad
	\cc_0 
	< c_0.
\] 
\end{assumption}
For the following analysis, we define the operator matrices
\[
	\bbB\colon \Q^2 \to (\Q^*)^2, \qquad 
	\bbC\colon \cHQsquare \to (\cHQdual)^2, \qquad 
	\bbD\colon \V \to (\cHQdual)^2
\]
given by
\[
	\bbB
	\coloneqq \begin{bmatrix} \calB & \\ & \calBB \end{bmatrix}, \qquad
	\bbC
	\coloneqq \begin{bmatrix} \phantom{-}\calC & -\calCC\ \ \\ -\calCC & \phantom{-}\calCCC\ \ \end{bmatrix}, \qquad
	\bbD
	\coloneqq \begin{bmatrix} \calD \\ \calDD \end{bmatrix}.
\]
The composed function spaces are equipped with the standard norms 
\[
	\| \qbf \|_{\Q^2}
	= \sqrt{\| q_1\|_{\Q}^2 + \| q_2\|_{\Q}^2}, \qquad 
	\| \vbf \|_{\cHQsquare}
	= \sqrt{\| v_1\|_{\cHQ}^2 + \| v_2\|_{\cHQ}^2},
\]
where $q_1,\,v_1$ and $q_2,\,v_2$ are the first and second components of $\qbf \in \Q^2$ and $\vbf \in \cHQ^2$, respectively.
Resulting stability and ellipticity estimates are summarized in the following lemma. 
\begin{lemma}[properties of $\bbB$, $\bbC$, and $\bbD$]
\label{lem:propertiesBlockOperators}
Let Assumptions~\ref{ass:operatorsI}, \ref{ass:operatorsII}, and~\ref{ass:chat} be satisfied. Then the operators $\bbB$, $\bbC$, and $\bbD$ are continuous with stability constants 
\begin{align*}
	C_{\bbB} 
	&\le \max(C_{b}, C_{\bb}), \\
	C_{\bbC} 
	&\le \max\Big(\sqrt{c_0^2  +\cc_0^2+ c_0  \cc_0+ \ccc_0\cc_0}, \sqrt{\cc_0^2  + \ccc_0^2+ c_0  \cc_0+ \ccc_0\cc_0}\Big)
	\le \sqrt{2c_0^2  + 2\ccc_0^2}, \\
	C_{\bbD} 
	&\le \sqrt{C_{d}^2 + C_{\dd}^2}.
\end{align*}
Moreover, the operators $\bbB$ and $\bbC$ are symmetric and elliptic with constants
\[
	c_{\bbB} \ge \min(c_{b},  c_{\bb}), \qquad 
	c_{\bbC} \ge \min\big(c_0-\cc_0, \ccc_0-\cc_0\big).
\]
\end{lemma}
\begin{proof}
For $\bbB$ we have with $\qbf = [q_1; q_2] \in \Q^2$, 
\[
	\|\bbB \qbf \|_{(\Q^*)^2} 
	= \sqrt{\| \calB q_1\|_{\Q^*}^2 + \| \calBB q_2\|_{\Q^*}^2} 
	\le \sqrt{C_{b}^2\, \|q_1\|_{\Q}^2 + C_{\bb}^2\, \|q_2\|_{\Q}^2}.
\]
Hence, the continuity constant~$C_{\bbB}$ is bounded by $\max(C_{b},  C_{\bb})$. Further, we note that 
\[
	\langle \bbB\qbf, \qbf\rangle
	= \langle \calB q_1, q_1\rangle + \langle \calBB q_2, q_2\rangle
	\ge c_{b}\, \|q_1\|^2_\Q + c_{\bb}\, \|q_2\|^2_\Q
	\ge \min(c_{b},  c_{\bb})\, \|\qbf\|^2_{\Q^2}.
\]
For the operator $\bbD$, we directly obtain 
\begin{align*}
	\|\bbD q\|_{(\cHQdual)^2} 
	= \sqrt{\| \calD q\|_{\cHQdual}^2 + \| \calDD q\|_{\cHQdual}^2}
	\le \sqrt{C_{d}^2\, \| q\|_{\V}^2 + C_{\dd}^2\, \| q\|_{\V}^2}
	=\sqrt{C_{d}^2 + C_{\dd}^2}\ \|q\|_{\V}.
\end{align*}
Finally, with $\qbf = [q_1; q_2] \in \cHQsquare$, we obtain for $\bbC$
\begin{align*}
	\|\bbC \qbf \|_{(\cHQdual)^2} 
	&= \sqrt{ \| \calC q_1 -\calCC q_2\|_{\cHQdual}^2 + \|\calCCC q_2 -\calCC q_1  \|_{\cHQdual}^2 } \\
	&\le \sqrt{ (c_0^2  +\cc_0^2) \| q_1  \|_{\cHQ}^2 + 2\,  (c_0  \cc_0+ \ccc_0\cc_0) \| q_2\|_{\cHQ}  \|q_1  \|_{\cHQ} +(\cc_0^2  + \ccc_0^2) \|q_2\|_{\cHQ}^2}\\
	&\le \sqrt{ (c_0^2  +\cc_0^2+ c_0  \cc_0+ \ccc_0\cc_0) \| q_1  \|_{\cHQ}^2  +(\cc_0^2  + \ccc_0^2+ c_0  \cc_0+ \ccc_0\cc_0) \|q_2\|_{\cHQ}^2} \\
	&\le \max\Big(\sqrt{c_0^2  +\cc_0^2+ c_0  \cc_0+ \ccc_0\cc_0}, \sqrt{\cc_0^2  + \ccc_0^2+ c_0  \cc_0+ \ccc_0\cc_0} \Big)\, \|\qbf\|_{\cHQsquare}. 
\end{align*}
Using Assumption~\ref{ass:chat}, this can be further bounded by 
\[
	\|\bbC \qbf \|_{(\cHQdual)^2}
	\le \max\Big(\sqrt{2c_0^2 + 2\ccc_0^2}, \sqrt{2c_0^2 + 2\ccc_0^2}\Big)\, \|\qbf\|_{\cHQsquare}
	\le \sqrt{ 2c_0^2+2\ccc_0^2}\ \|\qbf\|_{\cHQsquare}.
\]
For the ellipticity of~$\bbC$, we compute 
\begin{align*}
	\langle \bbC\qbf, \qbf\rangle 
	&= \langle \calC q_1, q_1\rangle - \langle \calCC q_1, q_2 \rangle - \langle \calCC q_2, q_1\rangle + \langle \calCCC q_2, q_2\rangle\\
	&\geq c_0\, \|q_1\|^2_{\cHQ} - 2\, \cc_0\, \|q_1\|_{\cHQ}\|q_2\|_{\cHQ} + \ccc_0\, \|q_2\|^2_{\cHQ} \\
	&\geq (c_0 - \cc_0) \|q_1\|^2_{\cHQ} +  (\ccc_0  -\cc_0) \|q_2\|^2_{\cHQ}\\
	&\geq \min\big(c_0-\cc_0, \ccc_0-\cc_0\big) \|\qbf\|^2_{\cHQsquare}.
	\qedhere
\end{align*} 
\end{proof}
The following theorem shows the existence of a unique weak solution for linear thermo--poroelasticity. 
\begin{theorem}[weak solution, thermo--poroelasticity]
\label{thm:existenceThermoPoro}
Consider Assumptions~\ref{ass:operatorsI}, \ref{ass:operatorsII}, and~\ref{ass:chat}. Further assume initial data~$p(0)=p^0, \theta(0)=\theta^0\in \cHQ$ as well as right-hand sides~$\f\in H^1(0,T;\V^*)$, $\g\in L^2(0,T;\Q^*)$, and~$\h\in L^2(0,T;\Q^*)$. Then there exists a unique (weak) solution triple
\[
	u \in L^2(0,T;\V), \qquad
	p, \theta \in L^2(0,T;\Q) \cap H^1(0,T;\Q^*).
\]
\end{theorem}
\begin{proof}
With the solution vector~$\pbf \coloneqq [p; \theta]$ and right-hand side $\gbf \coloneqq [\g; \h]$, system~\eqref{eq:thermoporoOperator} can be written as 
\begin{subequations}
\label{eq:theromoporoBlockFormulation}
\begin{align}
	\calA u - \bbD^*\pbf  
	&= \f, \\
	\bbD \dot u + \bbC \dot{\pbf} + \bbB \pbf 
	&= \mathbf{\g}. 
\end{align}
\end{subequations}
Note that this system has the same structure as linear poroelasticity. In order to apply the existence result of Theorem~\ref{thm:existencePoro}, we need to show that the operators~$\calA, \bbB, \bbC$, and~$\bbD$ satisfy Assumption~\ref{ass:operatorsI} if we replace the spaces~$\Q$ and $\cHQ$ by $\Q^2$ and $\cHQsquare$, respectively. For this, the only non-trivial property is the ellipticity of~$\bbC$, which follows from Lemma~\ref{lem:propertiesBlockOperators}.
\end{proof}
%
%
\subsection{Implicit time discretization}\label{sec:model:implicitDisc}
We close this section with the introduction of an implicit and fully coupled time discretization. More precisely, we apply the classical implicit Euler scheme to system~\eqref{eq:thermoporoOperator} based on an equidistant partition of the time interval with step size~$\tau$. Multiplying the second and first equation by~$\tau$, it reads 
\begin{align}
\label{eq:impliciEuler}	
	\begin{bmatrix} 
		\calA & -\calD^* & -\calDD^* \\ 
		\calD & \calC + \tau \calB & - \calCC\phantom{*} \\ 
		\calDD & -\calCC \phantom{*}& \calCCC + \tau \calBB 
	\end{bmatrix}
	\begin{bmatrix} 
		u^{n+1} \\ 
		p^{n+1} \\ 
		\theta^{n+1} 
	\end{bmatrix}
	= 
	\begin{bmatrix}
		0 & 0 & 0 \\ 
		\calD & \phantom{+}\calC & -\calCC \\ 
		\calDD & -\calCC & \phantom{+}\calCCC 
	\end{bmatrix}
	\begin{bmatrix} 
		u^{n} \\ 
		p^{n} \\ 
		\theta^{n} 
	\end{bmatrix}
	+ 
	\begin{bmatrix} 
		\phantom{\tau}\f^{n+1} \\ 
		\tau \g^{n+1} \\ 
		\tau \h^{n+1} 
	\end{bmatrix}.
\end{align}
\begin{theorem}[well-posedness of the implicit Euler scheme]
Given Assumption~\ref{ass:chat}, the iteration matrix of the implicit Euler scheme is invertible.
\end{theorem}
\begin{proof}
To write the implicit Euler scheme in terms of operator matrices, we define 
\[
	\bbM_\mathrm{imE}
	= \begin{bmatrix} \calA & -\bbD^* \\ \bbD & \phantom{-}\bbC\phantom{^*} \end{bmatrix}, \qquad
	\bbD = \begin{bmatrix} \calD \\ \calDD \end{bmatrix}, \qquad
	\bbC = \begin{bmatrix} \phantom{-}\calC_\tau & - \calCC\phantom{_\tau} \\ -\calCC\phantom{_\tau} & \phantom{-}\calCCC_\tau \end{bmatrix}
\]
with $\calC_\tau \coloneqq \calC + \tau \calB$, $\calCCC_\tau \coloneqq \calCCC + \tau \calBB$. The classical Schur complement yields
\[
	\bbM_\mathrm{imE}
	= \begin{bmatrix} I & 0 \\ \bbD \calA^{-1} & I \end{bmatrix}
	\begin{bmatrix} \calA & 0 \\ 0 & \bbC + \bbD A^{-1}\bbD^* \end{bmatrix}
	\begin{bmatrix} I & -\calA^{-1}\bbD^* \\ 0 & I \end{bmatrix}
\]
such that the invertibility of the iteration matrix $\bbM_\mathrm{imE}$ follows from the invertibility of the matrix $\bbC + \bbD \calA^{-1}\bbD^\top$. Since $\calA$ is positive definite, the matrix~$\bbD \calA^{-1}\bbD^\top$ is positive semi-definite. Hence, it is sufficient to show that~$\bbC$ is positive definite. To see this, we apply once more the Schur complement, leading to 
\[
	\bbC
	= \begin{bmatrix} I & 0 \\ -\calCC \calC_\tau^{-1} & I \end{bmatrix}
	\begin{bmatrix} \calC_\tau & 0 \\ 0\phantom{_\tau} & \calCCC_\tau - \calCC \calC_\tau^{-1}\calCC \end{bmatrix}
	\begin{bmatrix} I & -\calC_\tau^{-1}\calCC \\ 0 & I \end{bmatrix}.
\]
Since, $\calC_\tau$ is invertible (and even positive definite) and $\calCC$ is symmetric, we conclude that $\bbC$ is positive definite if and only if $\calCCC_\tau - \calCC \calC_\tau^{-1}\calCC$ is positive definite. A sufficient condition for this is the positivity of~$\calCCC - \calCC \calC^{-1}\calCC$, which is guaranteed by Assumption~\ref{ass:chat}.
\end{proof}
%
%
%
\section{An Abstract Convergence Result} 
\label{sec:convergence}

In~\cite{AltMU21}, the idea has been introduced of analyzing a decoupled scheme based on a connection to PDEs with delays, i.e., PDEs where the solution depends on itself in the past. More precisely, the semi-explicit scheme introduced in~\cite{AltMU21} may be understood as a classical implicit Euler discretization of a PDE with a delay and is the basis for our studies here. It turns out that the proposed decoupled discretizations for thermo--poroelasticity introduced in Sections~\ref{sec:halfDecoupled} and~\ref{sec:fullyDecoupled} below can be rephrased as a specific delay equation of parabolic type.  
Therefore, this section is devoted to the convergence proof of the implicit Euler scheme to such a delay equation. This then serves as a basis to prove convergence of the proposed decoupled schemes. 

We consider the operator equation
\begin{align}
\label{eq:auxEquation}
	\calE \dot p + \calK p + \calM \dot p_\tau 
	= r
\end{align}
with the delay term $p_\tau(t) \coloneqq p(t-\tau)$ and right-hand side~$r$. Here, $\tau$ denotes the delay time which is chosen as the time step size later on. Note that an equation of the form~\eqref{eq:auxEquation} not only calls for an initial value but for a so-called {\em history function}, i.e., the prescription of $p$ on the time interval $[-\tau,0]$. Moreover, we make the following assumptions. 
\begin{assumption}
\label{ass:parabolicDelaySetting}	
Let $\calX$ be a Hilbert space such that $\calX, \cHX, \calX^*$ forms a Gelfand triple. Equation~\eqref{eq:auxEquation} should be understood as an equation in $\calX^*$ for almost every time $t\in [0,T]$.
Moreover, the operators $\calE, \calM\colon\cHX\to\cHX$ and $\calK\colon \calX \to\calX^*$ are symmetric, continuous, and elliptic in their respective spaces. 
The ellipticity constants of $\calE$ and $\calK$ are denoted by~$c_\calE$ and $c_\calK$, respectively. Moreover, the continuity constant of $\calM$ equals $C_\calM$. 
\end{assumption}

The implicit Euler scheme with constant step size $\tau$ applied to~\eqref{eq:auxEquation} yields 
\begin{align}
\label{eq:auxEquation:Euler}
	\calE p^{n+1} - \calE p^n + \tau \calK p^{n+1} + \calM\, (p^n - p^{n-1})
	= \tau r^{n+1}. 
\end{align}
In the following theorem, we show under which conditions this provides a first-order approximation of~\eqref{eq:auxEquation}. Note that the requirement of having a history function in the continuous case carries over to the assumption that two initial values $p^0$ and $p^1$ are given. 
\begin{theorem}
\label{th:convergence:auxEquation}
Consider Assumption~\ref{ass:parabolicDelaySetting} with $C_\calM \le c_\calE$ and a continuous right-hand side $r\colon[0,T] \to \calX^*$. Then the Euler scheme~\eqref{eq:auxEquation:Euler} is first-order convergent. More precisely, if $p\in H^2(0,T; \calX)$ is the solution of~\eqref{eq:auxEquation}, we have
\[
	\|p(t^n) - p^n\|_\cHX^2
	+ \tau\, \sum_{j=2}^n \|p(t^j) - p^j\|^2_\calK 
	\le E_\mathrm{init} + C\, T\, \tau^2,
\]
where $C$ is a positive constant and $$E_\mathrm{init} = \|p(\tau) - p^1\|_\calE^2 + \|p(0) - p^0\|_\calM^2 \lesssim \|p(\tau) - p^1\|_\cHX^2 + \|p(0) - p^0\|_\cHX^2$$ contains the initial error. 
\end{theorem}	
\begin{proof}
Due to the assumptions, the operators $\calK$, $\calE$, and $\calM$ define norms, which are equivalent to the $\calX$- and $\cHX$-norms, respectively. In particular, we have
\[
	c_\calK\, \|p\|_{\calX}^2
	\le \langle \calK p, p\rangle
	= \| p\|_\calK^2, \qquad
	\| p\|^2_\calM
	\le C_\calM\, \| p\|^2_\cHX
	\le \tfrac{C_\calM}{c_\calE}\, \| p\|^2_\calE
	= \tfrac{C_\calM}{c_\calE}\, \langle\calE p, p\rangle.
\]
Moreover, due to symmetry, we have the identity
\[
	2\,\langle \calM p, p-q \rangle 
	= \|p\|^2_\calM - \|q\|^2_\calM + \|p-q\|^2_\calM.
\]	
Such an equality also holds for the inner product defined through~$\calE$. We define the error and the defect by 
\[
	e^n 
	\coloneqq p(t^n) - p^n, \qquad
	d^n 
	\coloneqq p(t^{n}) - p(t^{n-1}) - \tau \dot p(t^{n}),
\]
respectively. 
For the latter, the regularity assumption on $p$ implies that $\|d^n\|_\cHX \le C\, \tau^2$ with a constant depending on the second derivative of $p$. Taking the difference of~\eqref{eq:auxEquation} evaluated at time $t = t^{n+1}$ (and multiplied by $\tau$) and \eqref{eq:auxEquation:Euler}, we obtain the error equation 
\[
	\calE e^{n+1} - \calE e^n + \tau \calK e^{n+1} + \calM\, (e^n-e^{n-1}) 
	= \calE d^{n+1} + \calM d^n.
\]
With the test function $2e^{n+1}$, we obtain by Young's inequality 
\begin{align*}
	\|e^{n+1}\|_\calE^2 - &\|e^n\|_\calE^2 + \|e^{n+1}-e^n\|_\calE^2 
	+ 2\tau\, \|e^{n+1}\|^2_\calK 
	+ \|e^{n}\|_\calM^2 - \|e^{n-1}\|_\calM^2 + \|e^{n}-e^{n-1}\|_\calM^2  \\
	&= 2\, \big\langle \calE d^{n+1}, e^{n+1}\big\rangle 
	+ 2\, \big\langle \calM d^n, e^{n+1}\big\rangle 
	- 2\, \big\langle\calM\, (e^n-e^{n-1}), e^{n+1}-e^{n}\big\rangle \\
	&\le 2\, \|d^{n+1}\|_\calE \|e^{n+1}\|_\calE 
	+ 2\,\|d^{n}\|_\calM \|e^{n+1}\|_\calM 
	+ 2\,\| e^n-e^{n-1}\|_\calM \|e^{n+1}-e^{n}\|_\calM \\
	&\le C\, \tau^3 + \tau\, \|e^{n+1}\|^2_\calK 
	+ \| e^n-e^{n-1}\|_\calM^2 + \tfrac{C_\calM}{c_\calE}\, \|e^{n+1}-e^{n}\|_\calE^2.  
\end{align*}
The constant $C$ depends on the operators $\calE, \calK, \calM$ and the constant of the embedding $\cHX\hookrightarrow \calX$. Due to $C_\calM \le c_\calE$, we can eliminate the term involving $\|e^{n+1}-e^{n}\|_\calE$ and obtain
\begin{align*}
	\|e^{n+1}\|_\calE^2 - \|e^n\|_\calE^2 
	+ \tau\, \|e^{n+1}\|^2_\calK 
	+ \|e^{n}\|_\calM^2 - \|e^{n-1}\|_\calM^2 
	\le C\, \tau^3 . 
\end{align*}
Summing from $1$ to $n$, we get
\begin{align*}
	\|e^{n+1}\|^2_\calE 
	+ \tau\, \sum_{j=1}^{n} \|e^{j+1}\|^2_\calK 
	+ \|e^{n}\|_\calM^2
	\le \|e^1\|_\calE^2 + \|e^{0}\|_\calM^2 + \sum_{j=1}^{n}C\, \tau^3 
	= E_\mathrm{init} + C\, T\, \tau^2,
\end{align*}
which completes the proof.
\end{proof}
\begin{example}[poroelasticity]
We consider the equations of linear poroelasticity from~\eqref{eq:poroOperator} where we introduce a delay term in the first equation. More precisely, we replace~\eqref{eq:poroOperator:a} by $\calA u - \calD^*p_\tau = \f$, leading to the system
\[
	\calC \dot p + \calB p + \calD\calA^{-1}\calD^*\dot p_\tau
	= g - \calD \calA^{-1} \dot f.
\]
Considering $\calX = \Q$, $\cHX = \cHQ$ and the operators 
\[
	\calE 
	= \calC\colon\cHQ\to\cHQdual,\qquad
	\calK 
	= \calB\colon\Q\to\Q^*, \qquad
	\calM 
	= \calD\calA^{-1}\calD^*
	\colon \cHQ 
	\to \cHQ, 
\]
the system can be written in the form~\eqref{eq:auxEquation}. Further note that $\calC$ and $\calB$ are continuous, symmetric, and elliptic with constants~$c_c$ and~$c_b$, respectively. Since $\calA$ is symmetric, we conclude that
\[
	(\calM p, q)_{\cHQ}
	= \langle \calA^{-1}\calD^*p, \calD^*q\rangle_{\V, \V^*}
	= \langle p,  \calD\calA^{-1}\calD^*q\rangle_{\cHQ,\cHQdual}
	= (p, \calM q)_{\cHQ}.
\]
To see that $\calM$ is elliptic, we apply the inf--sup stability of $\calD$ (see, e.g., \cite{Bog79,AcoD17}) 
with constant $c_d$, leading to  
\[
	(\calM q, q)_{\cHQ}
	= \langle\calA^{-1}\calD^*q, \calD^*q\rangle_{\V, \V^*}
	\ge \frac{1}{C_a}\, \|\calD^*q\|^2_{\V^*}
	\ge \frac{c^2_d}{C_a}\, \|q\|^2_{\cHQ}.
\]
Finally, we have 
\[
	\|\calM \qbf \|_{\cHQdual} 
	= \big\|\calD\calA^{-1}\calD^* \qbf \big\|_{\cHQdual} 
	\le \frac{C_d^2}{c_a}\, \|\qbf\|_{\cHQ}.
\]
Hence, following Theorem~\ref{th:convergence:auxEquation}, the semi-explicit scheme
\begin{align*}
	\begin{bmatrix} 
		\calA & \phantom{-}0  \\ 
		\calD & \calC + \tau \calB
	\end{bmatrix}
	\begin{bmatrix} u^{n+1} \\ p^{n+1} \end{bmatrix}
	= 
	\begin{bmatrix} 0 & \phantom{-}\calD^* \\ \calD & \phantom{-}\calC \phantom{^*} \end{bmatrix}
	\begin{bmatrix} u^{n} \\ p^{n} \end{bmatrix}
	+ 
	\begin{bmatrix} \phantom{\tau}\f^{n+1} \\ \tau \g^{n+1} \end{bmatrix}
\end{align*}
converges with order 1 as long as $C_d^2 \le c_a c_0$, which guarantees $C_\calM \le c_\calE$. This is true for general elliptic--parabolic systems of the form~\eqref{eq:poroBilinear} and is in line with the convergence result obtained in~\cite{AltMU21}. For poroelasticity, this can be formulated solely in terms of the physical constants, namely $\alpha^2 M \le \mu+\lambda$; see~\cite{AltD24}.
\end{example}
%
%
\section{Half-Decoupled Schemes}
\label{sec:halfDecoupled}
This section is devoted to numerical schemes for thermo--poroelasticity, which decouple the computation of the deformation~$u$ from the other two variables. We first present the iterative scheme introduced in~\cite{BruABNR20} before introducing two novel schemes. 
%
%
\subsection{HF--M scheme from~\cite{BruABNR20}}
\label{sec:halfDecoupled:iterative}
The following iterative scheme requires the repeated alternating solution of a problem for the temperature and pressure (HF, heat and flow) and a problem for the displacement (M, mechanics) within each time step. That is, an additional inner iteration is necessary to achieve a reasonable decoupling. 
Moreover, the scheme involves stabilization parameters, which we denote by $L_\theta$ and $L_p$. Note, however, that the original scheme proposed in~\cite{BruABNR20} is formulated for nonlinear thermo--poroelasticity and its `five-field' formulation, which fits the application of lowest-order finite elements in $H(\ddiv)$ spaces. The scheme presented here is `translated' to the considered formulation~\eqref{eq:thermoporoBilinear} of the linear case.

Given the approximation $(u^n, p^n, \theta^n)$ at time $t = t^n$ and the previous inner iterate $(u^{n+1,k}, p^{n+1,k}, \theta^{n+1,k})$, we compute in a first step $\theta^{n+1,k+1}$ and $p^{n+1,k+1}$ by  
\begin{align*}
	\calDD u^{n+1,k} + (\calCCC+L_\theta) \theta^{n+1,k+1} - \calCC p^{n+1,k+1} 
	&+ \tau \calBB \theta^{n+1,k+1} \\
	&\quad= \tau \h^{n+1} + \calDD u^n + \calCCC \theta^n - \calCC p^n + L_\theta \theta^{n+1,k}, \\
	\calD u^{n+1,k} + (\calC+L_p) p^{n+1,k+1} - \calCC \theta^{n+1,k+1} &+ \tau \calB p^{n+1,k+1} \\
	&\quad= \tau \g^{n+1} + \calD u^n + \calC p^n - \calCC \theta^n + L_p p^{n+1,k},
\end{align*}
which is a coupled system involving~$\theta^{n+1,k+1}$ and~$p^{n+1,k+1}$.
Then, in a second step, we compute $u^{n+1,k+1}$ from 
\[
	\calA u^{n+1,k+1} - \calD^* p^{n+1,k+1} - \calDD \theta^{n+1,k+1} 
	= \f^{n+1}.
\]
The iteration can also be written in matrix form, namely
\begin{align}
	&\begin{bmatrix} 
		\calA & -\calD^* & -\calDD^*\\ 
		0 & \calC + L_p + \tau \calB & - \calCC\phantom{^*} \\ 
		0 & -\calCC\phantom{^*} & \calCCC + L_\theta + \tau \calBB 
	\end{bmatrix}
	\begin{bmatrix} 
		u^{n+1,k+1} \\ 
		p^{n+1,k+1} \\ 
		\theta^{n+1,k+1} 
	\end{bmatrix} \notag \\
	&\hspace{2.5cm}= 
	\begin{bmatrix} 
		\phantom{-}0 & 0 & 0 \\ 
		-\calD & L_p & 0 \\ 
		-\calDD & 0& L_\theta 
	\end{bmatrix}
	\begin{bmatrix} 
		u^{n+1,k} \\ 
		p^{n+1,k} \\ 
		\theta^{n+1,k} 
	\end{bmatrix}
	+
	\begin{bmatrix} 
		0 & \phantom{-}0 & \phantom{-}0 \\ 
		\calD & \phantom{-}\calC & - \calCC \\ 
		\calDD & -\calCC & \phantom{-}\calCCC 
	\end{bmatrix}
	\begin{bmatrix} 
		u^{n} \\ 
		p^{n} \\ 
		\theta^{n} 
	\end{bmatrix}
	+ 
	\begin{bmatrix} 
		\phantom{\tau}\f^{n+1} \\ 
		\tau \g^{n+1} \\ 
		\tau \h^{n+1} 
	\end{bmatrix}
	\label{eq:halfDecoupled:norway}
\end{align}
for $k = 0, \dots, K$. Here, the structure of the iteration matrix on the left-hand side clearly shows the decoupling of $u$ and $(p,\theta)$ in each step. 
In~\cite{BruABNR20}, the following convergence result of the inner iteration is shown. 
\begin{theorem}[cf.~{\cite[Th.~4.3]{BruABNR20}}]
Assume sufficiently large stabilization parameters $L_\theta$ and $L_p$. Then, iteration~\eqref{eq:halfDecoupled:norway} defines a contraction.
\end{theorem}
\begin{remark}
In the nonlinear case considered in~\cite{BruABNR20}, one also needs a step size restriction in order to guarantee convergence of the inner iteration.  
\end{remark}
The contraction property implies that for $K\to\infty$ we converge to one step of the (fully coupled) implicit Euler method. This, in turn, guarantees convergence of the overall method as long as the number of inner iteration steps is large enough. Unfortunately, this result does not provide an estimate of the number of needed steps. Moreover, the number of steps grows when decreasing $\tau$.
%
%
\subsection{Semi-explicit discretization (M--HF)}
\label{sec:halfDecoupled:semiexplicit}
We now present an alternative approach without inner iteration and without stabilization parameter. This means that one time step simply consists of the solution of two subsystems, namely one for the displacement~$u$ and one for the pair $(p,\theta)$. Avoiding inner iterations completely comes along with a {\em weak coupling condition}, i.e., convergence is only guaranteed in a particular parameter setting. 

The scheme reads as follows. Given the approximation $(u^n, p^n, \theta^n)$ from the previous time step, first compute $u^{n+1}$ by solving 
\[
	\calA u^{n+1} - \calD^* p^n - \calDD^* \theta^n 
	= f^{n+1}.
\]
This corresponds to an implicit--explicit discretization, where $p$ and $\theta$ are included explicitly. The remaining equations are discretized by the implicit Euler method, where we use the already updated displacement. In total, this leads to the time stepping scheme
\begin{align}
	\label{eq:halfDecoupled:semiExpl}
	\begin{bmatrix} 
		\calA & \phantom{-}0 & \phantom{-}0 \\ 
		\calD & \calC + \tau \calB & - \calCC \\ 
		\calDD & -\calCC & \calCCC + \tau \calBB 
	\end{bmatrix}
	\begin{bmatrix} u^{n+1} \\ p^{n+1} \\ \theta^{n+1} \end{bmatrix}
	= 
	\begin{bmatrix} 0 & \phantom{-}\calD^* & \phantom{-}\calDD^* \\ \calD &\phantom{-} \calC \phantom{^*}& - \calCC\phantom{^*} \\ \calDD & -\calCC \phantom{^*}& \phantom{-}\calCCC\phantom{^*}	\end{bmatrix}
	\begin{bmatrix} u^{n} \\ p^{n} \\ \theta^{n} \end{bmatrix}
	+ 
	\begin{bmatrix} \phantom{\tau}\f^{n+1} \\ \tau \g^{n+1} \\ \tau \h^{n+1} \end{bmatrix}.
\end{align}
This system is similar to the implicit time discretization from Section~\ref{sec:model:implicitDisc}. The only difference is that the pressure and temperature terms are moved to the left-hand side (i.e., they are treated explicitly). The following theorem shows that this modification maintains the first-order convergence as long as a certain coupling condition is satisfied. 
\begin{theorem}[first-order convergence, half decoupled]
\label{th:halfDecoupled:semiexplicit}
Assume sufficiently smooth right-hand sides and 
\[
	\omega_\mathrm{HD} 
	\coloneqq \frac{C_d^2 + C^2_{\dd}}{c_a\, \min\{ c_0-\cc_0, \ccc_0-\cc_0 \}}
	\le 1.
\]
Then the scheme~\eqref{eq:halfDecoupled:semiExpl} converges with order one.
\end{theorem}
\begin{proof}
The proposed scheme equals the implicit discretization of the delay system
\begin{subequations}
\label{eq:theromoporoBlockFormulation:delay}
\begin{align}
	\calA u - \bbD^*\pbf_\tau  
	&= \f, \\
	\bbD \dot u + \bbC \dot{\pbf} + \bbB \pbf 
	&= \mathbf{\g}
\end{align}
\end{subequations}
with the operator matrices $\bbB, \bbC$, and~$\bbD$ as introduced in the proof of Theorem~\ref{thm:existenceThermoPoro} and the delay term $\pbf_\tau(t) := \pbf(t-\tau)$ as before. Hence, we can apply the convergence result of~\cite{AltMU21} for elliptic--parabolic problems, which is valid under the {\em weak coupling condition} $C_{\bbD}^2 \le c_a c_\bbC$. From Lemma~\ref{lem:propertiesBlockOperators}, we know that this condition is satisfied for $\omega_\mathrm{HD} \le 1$.  
\end{proof}
\begin{remark}
In the case of thermo--poroelasticity, the continuity and ellipticity constants appearing in $\omega_\mathrm{HD}$ can be bounded in terms of the physical parameters. Moreover, observing that the convergence proof is based on the estimate $\langle \bbD v, \qbf \rangle \le \sqrt{\omega_\mathrm{HD}}\, \|v\|_a \|\qbf\|_\bbC$, this can be further improved. 
For homogeneous Dirichlet boundary conditions, it is known from~\cite[Sect.~6.3]{Cia88} that $a(v,v) \ge (\mu + \lambda)\, \|\nabla\cdot v\|_\cHQ^2$; see also~\cite{AltD24} for the derivation. With this, we conclude that
\begin{align*}
	\langle \bbD v, \qbf \rangle
	= \langle \calD v, q_1 \rangle + \langle \calDD v, q_2 \rangle 
	\le \frac{1}{\sqrt{\mu + \lambda}}\, \Big( \alpha\, \|q_1\|_{\cHQ} + \beta\, \|q_2\|_{\cHQ} \Big) \|v\|_a.
\end{align*}
Using the ellipticity estimate for $\bbC$ from Lemma~\ref{lem:propertiesBlockOperators}, we get 
\[
	\big( \alpha\, \|q_1\|_{\cHQ} + \beta\, \|q_2\|_{\cHQ} \big)^2
	\le 2\max(\alpha^2, \beta^2)\, \|\qbf\|^2_{\cHQsquare} 
	\le \frac{2\max(\alpha^2, \beta^2)}{\min\big(c_0-\cc_0, \ccc_0-\cc_0\big)}\, \|\qbf\|^2_\bbC.
\]
Hence, the condition $\omega_\mathrm{HD} \le 1$ can be relaxed to the coupling condition 
\[
	\frac{2\,\max(\alpha^2, \beta^2)}{(\mu + \lambda)\, \min(c_0-\cc_0, \ccc_0-\cc_0)}
	\le 1.
\]
\end{remark}
\begin{remark}
The convergence result can also be obtained as a corollary of Theorem~\ref{th:convergence:auxEquation}. For this, one first needs to show that the solutions of~\eqref{eq:theromoporoBlockFormulation} and~\eqref{eq:theromoporoBlockFormulation:delay} only differ by a term of order~$\tau$. Then, the claim follows by setting $\calX \coloneqq \calQ^2$, $\cHX \coloneqq (\cHQ)^2$ for the spaces and $\calE \coloneqq \bbC$, $\calK = \bbB$, $\calM \coloneqq \bbD\calA^{-1}\bbD^*$ as operators, leading again to the condition $C_{\bbD}^2 \le c_a c_\bbC$. 
\end{remark}
%
%
\subsection{A semi-explicit discretization with inner iteration (M--HF)}
\label{sec:halfDecoupled:semiexplicitIter}
The scheme in Section~\ref{sec:halfDecoupled:semiexplicit} has the obvious drawback that it is not applicable for all applications, since the coupling condition $\omega_\mathrm{HD} \le 1$ needs to be fulfilled. On the other hand, the iterative scheme in Section~\ref{sec:halfDecoupled:iterative} depends on stabilization parameters and an unknown number of inner iteration steps. With the following approach, we aim to overcome both disadvantages simultaneously. 
Based on the coupling parameter $\omega_\mathrm{HD}$ from Theorem~\ref{th:halfDecoupled:semiexplicit}, we define the {\em relaxation parameter} 
\[
	\gamma 
	\coloneqq \frac{2}{2+\omega_\mathrm{HD}}.
\]
Moreover, $K \in \N$ denotes the (fixed) number of inner iteration steps. In each time, step we initialize 
\begin{subequations}
\label{eq:halfDecoupled:semiExplIterative}
\begin{align}
\label{eq:halfDecoupled:semiExplIterative:init}
	x^{n+1,0} 
	= \begin{bmatrix} 
		u^{n+1,0} \\  
		p^{n+1,0} \\ 
		\theta^{n+1,0}
	\end{bmatrix} 
	\coloneqq \begin{bmatrix} 
		u^{n} \\  
		p^{n} \\ 
		\theta^{n}
	\end{bmatrix} 
	= x^n.
\end{align}
For $k = 0, \dots, K-1$, we solve the system 
\begin{align}
\label{eq:halfDecoupled:semiExplIterative:system}
	&\begin{bmatrix} 
		\calA & \phantom{-}0 & \phantom{-}0 \\ 
		\calD & \calC + \tau \calB & - \calCC \\ 
		\calDD & -\calCC & \calCCC + \tau \calBB 
	\end{bmatrix}
	\begin{bmatrix} 
		u^{n+1,k+1} \\ \hat p^{n+1,k+1} \\ \hat \theta^{n+1,k+1} 
	\end{bmatrix} \notag \\
	&\hspace{2.0cm}= 
	\begin{bmatrix} 
		0 & \calD^* & \calDD^* \\ 0 & 0\phantom{^*} &  0\phantom{^*} \\ 0 & 0\phantom{^*} & 0\phantom{^*} 
	\end{bmatrix}
	\begin{bmatrix} 
		u^{n+1,k} \\ p^{n+1,k} \\ \theta^{n+1,k} 
	\end{bmatrix}
	+
	\begin{bmatrix} 
		0 & \phantom{-}0 & \phantom{-}0\\ \calD & \phantom{-}\calC & - \calCC \\ \calDD & -\calCC & \phantom{-}\calCCC	
	\end{bmatrix}
	\begin{bmatrix} 
		u^{n} \\ p^{n} \\ \theta^{n} 
	\end{bmatrix}
	+ 
	\begin{bmatrix} 
		\phantom{\tau}\f^{n+1} \\ \tau \g^{n+1} \\ \tau \h^{n+1} 
	\end{bmatrix}
\end{align}
for $u^{n+1,k+1}$, $ \hat p^{n+1,k+1}$, and $ \hat \theta^{n+1,k+1}$.
Additionally, except for the last step, we dampen the pressure and the temperature variable by setting 
\begin{align}
\label{eq:halfDecoupled:semiExplIterative:damping}
	\begin{bmatrix} 
		p^{n+1,k+1} \\ \theta^{n+1,k+1}
	\end{bmatrix} 
	\coloneqq 
	\gamma 
	\begin{bmatrix} 
		\hat p^{n+1,k+1} \\ \hat \theta^{n+1,k+1} 
	\end{bmatrix} 
	+ (1-\gamma)
	\begin{bmatrix} 
		p^n \\ \theta^n 
	\end{bmatrix}.
\end{align}
Finally, the approximations at time $t^{n+1}$ are given by 
\begin{align}
\label{eq:halfDecoupled:semiExplIterative:final}	
	u^{n+1} 
	\coloneqq u^{n+1,K}, \qquad 
	p^{n+1} 
	\coloneqq \hat p^{n+1, K}, \qquad 
	\theta^{n+1} 
	\coloneqq \hat \theta^{n+1, K}.
\end{align}
\end{subequations}
In contrast to the iterative approach of Section~\ref{sec:halfDecoupled:iterative}, we have a fixed number of inner iteration steps, independent of $\tau$. Moreover, we provide a lower bound on $K$ to guarantee first-order convergence. Note, however, that the following result is shown for the finite-dimensional setting, i.e., after a spatial discretization, such that the infinite-dimensional function spaces $\calV,\ \calQ,\ \cHQ$, and $\cHV$ are replaced by finite-dimensional vector spaces. Note, however, that we keep the original notation for convenience. 
\begin{theorem}[first-order convergence, half decoupled, iterative]
We consider the finite-dimensional setting, i.e., $\V = \cHV = \R^{n_u}$, $\Q = \cHQ = \R^{n_p}$, such that $\calA,\calB,\calBB,\calC,\calCC,\calCCC,\calD,\calDD$ are matrices. Let the solution $(u,p,\theta)$ of problem~\eqref{eq:thermoporoOperator} satisfy the smoothness conditions $u \in C^2([0,T], \R^{n_u})$, $p \in C^2([0,T], \R^{n_p})$, and $\theta \in C^2([0,T], \R^{n_\theta})$. Moreover, we assume that the rights-hand sides satisfy $f \in C^2([0,T], \R^{n_u})$, $g \in C^1([0,T], \R^{n_p})$ and $h \in C^1([0,T], \R^{n_\theta})$. For the initial data, we assume consistency conditions of the form 
\begin{align*}
	\calB p^0 
	= g(0)-\calD\calA^{-1}\dot f(0) + \calO(\tau), \qquad
	\calBB \theta^0 
	= h(0)-\calDD \calA^{-1}\dot f(0) + \calO(\tau).
\end{align*}
Then the approximations $(u^n,p^n,\theta^n)$ given by scheme~\eqref{eq:halfDecoupled:semiExplIterative} with $K$ inner iterations, satisfying 
\[
	\frac{\omega_\mathrm{HD}^K}{(2+\omega_\mathrm{HD})^{K-1}} 
	< 1
\]
with $\omega_\mathrm{HD}$ from Theorem~\ref{th:halfDecoupled:semiexplicit}, are first-order convergent, i.e., 
\[
	\|u^n - u(t^n)\|^2_\calA + \|p^n-p(t^n)\|^2_\calB + \|\theta^n - \theta(t^n)\|^2_{\calBB} 
	\leq C \tau^2. 
\]
The constant $C > 0$ only depends on the solution, the right-hand sides, the time horizon, and the material parameters.
\end{theorem}
\begin{proof}
As in the proof of Theorem~\ref{thm:existenceThermoPoro}, we define the solution vector~$\pbf \coloneqq [p; \theta]$ and $\gbf \coloneqq [\g; \h]$ with which we can rewrite system~\eqref{eq:thermoporoOperator} in the form~\eqref{eq:theromoporoBlockFormulation}. With this, the consistency conditions of the initial data read
\[
	\bbB \pbf^0
	= \gbf(0) - \bbD \calA^{-1} \dot f(0) + \calO(\tau).
\]
Together with the condition on~$K$, we can now apply~\cite[Th.~4.4]{AltD24}.
%
\end{proof}
%
%
\section{Fully Decoupled Schemes}
\label{sec:fullyDecoupled}
We now turn to numerical schemes which fully decouple the three equations of~\eqref{eq:thermoporoOperator}, such that each can be solved separately. Again, we first gather known (iterative) schemes before we present a new semi-explicit approach. 
%
%
\subsection{M--F--H scheme from~\cite{KolV17}}
The following scheme is motivated by an operator splitting approach (see~\cite{Sam92} and~\cite{KolVV14}) and involves a parameter~$\sigma > 0$. One step of the scheme reads  
\begin{align}
	&\begin{bmatrix} 
		\calA & 0 & 0 \\ 
		\calD & \sigma \calC + \tau \calB & 0 \\ 
		\calDD & 0 & \sigma\calCCC + \tau \calBB 
	\end{bmatrix}
	\begin{bmatrix} 
		u^{n+1} \\ 
		p^{n+1} \\ 
		\theta^{n+1} 
	\end{bmatrix} \notag \\
	&\hspace{1.0cm}= 
	\begin{bmatrix}
		0 & \calD^* & \calDD^* \\ 
		\calD & (2\sigma-1)\, \calC & \calCC\phantom{^*} \\ 
		\calDD & \calCC\phantom{^*} & (2\sigma-1)\,\calCCC
	\end{bmatrix}
	\begin{bmatrix} 
		u^{n} \\ 
		p^{n} \\ 
		\theta^{n} 
	\end{bmatrix}
	+
	\begin{bmatrix}
		0 & \phantom{-}0 & \phantom{-}0 \\
		0 & (1-\sigma)\, \calC & - \calCC \\
		0 & -\calCC & (1-\sigma)\,\calCCC 
	\end{bmatrix}
	\begin{bmatrix} 
		u^{n-1} \\ 
		p^{n-1} \\ 
		\theta^{n-1} 
	\end{bmatrix}.
	\label{eq:fullyDecoupled:russia}
\end{align}
Note that this corresponds to the formulation in~\cite{KolV17}, which works without right-hand sides. Moreover, we would like to emphasize that this is a two-step scheme which needs, apart from the initial data $u^0, p^0, \theta^0$, also given values for~$p^1$ and~$\theta^1$. Due to the block triangular structure of the iteration matrix on the left-hand side, the scheme decouples. This means that we can first solve for $u^{n+1}$ and then for $p^{n+1}$ and $\theta^{n+1}$ in parallel.

For the spatially discretized setting, the scheme was analyzed in~\cite{KolV17} and it is claimed to be stable for $\sigma\ge\sigma_0$, where $\sigma_0$ neither depends on the mesh nor on the time step. However, no explicit formula for $\sigma_0$ is provided. 
\begin{remark}
One may also apply Theorem~\ref{th:convergence:auxEquation} here. 
%
%
For this, one considers the delay equation
\begin{align*}
	\sigma \begin{bmatrix}\ \calC & \phantom{-\calCC\ }\ \\ \phantom{\calCC} & \phantom{-}\calCCC\ \ \end{bmatrix}\dot {\bar \pbf} + \bbB {\bar \pbf} + \left( \bbD \calA^{-1}\bbD^* -\begin{bmatrix} (\sigma - 1)\,\calC & \calCC\ \ 
		\\ \calCC & (\sigma - 1)\,\calCCC\ \ 
	\end{bmatrix} \right)\dot {\bar \pbf}_\tau
	&= \gbf - \bbD \calA^{-1}\dot f,
\end{align*}
where we use the notation $\bar \pbf$, $\gbf$ as introduced in the proof of Theorem~\ref{thm:existenceThermoPoro}. 
Assuming $\sigma\ge1$, the resulting sufficient conditions, however, yield upper bounds on $\sigma$ 
which seem quite unrealistic. 
\end{remark}
\subsection{H--F--M scheme from~\cite{BruABNR20}}
We now turn to fully decoupled iterative methods as introduced in~\cite{BruABNR20}. There exist several variants, which consider the successive (approximate) solution of the different subsystems with a certain ordering. As already described in Section~\ref{sec:halfDecoupled:iterative}, the schemes introduced in~\cite{BruABNR20} are originally formulated for the nonlinear `five-field' formulation and adopted to the here considered system~\eqref{eq:thermoporoBilinear}. 
%

Let $L_\theta$ and $L_p$ denote stabilization parameters and $(u^n, p^n, \theta^n)$ the approximation from the previous time step. Then, given $(u^{n+1,k}, p^{n+1,k}, \theta^{n+1,k})$, we compute 
%
%
%
the next inner iterate $(u^{n+1,k}, p^{n+1,k}, \theta^{n+1,k})$ as the solution of  
\begin{align}
	&\begin{bmatrix} 
		\calA & -\calD^* & -\calDD^* \\ 
		0 & \calC + L_p + \tau \calB & - \calCC\phantom{^*} \\ 
		0 & \phantom{-}0\phantom{^*} & \calCCC + L_\theta + \tau \calBB 
	\end{bmatrix}
	\begin{bmatrix} 
		u^{n+1,k+1} \\ 
		p^{n+1,k+1} \\ 
		\theta^{n+1,k+1} 
	\end{bmatrix} \notag \\
	&\hspace{2.7cm}= 
	\begin{bmatrix} 
		\phantom{-}0 & 0 & 0 \\ 
		-\calD & L_p & 0 \\ 
		-\calDD & \calCC & L_\theta 
	\end{bmatrix}
	\begin{bmatrix} 
		u^{n+1,k} \\ 
		p^{n+1,k} \\ 
		\theta^{n+1,k} 
	\end{bmatrix}
	+
	\begin{bmatrix} 
		0 & \phantom{-}0 & \phantom{-}0 \\ 
		\calD & \phantom{-}\calC & - \calCC \\ 
		\calDD & -\calCC & \phantom{-}\calCCC 
	\end{bmatrix}
	\begin{bmatrix} 
		u^{n} \\
		p^{n} \\ 
		\theta^{n} 
	\end{bmatrix}
	+ 
	\begin{bmatrix} 
		\phantom{\tau}\f^{n+1} \\ 
		\tau \g^{n+1} \\ 
		\tau \h^{n+1} 
	\end{bmatrix}
	\label{eq:fullyDecoupled:norway}
\end{align}
for $k = 0, \dots, K$. Again, the needed computations in each (inner) step decouple as the matrix on the left-hand side has a block triangular structure. Hence, one can first solve for the temperature, then for the pressure, and finally for the displacement. Again, the following convergence result of the inner iteration can be shown. 
\begin{theorem}[cf.~{\cite[Rem.~4.3]{BruABNR20}}]\label{thm:norwayConvergence}
Assume sufficiently large stabilization parameters $L_\theta$ and $L_p$. Then, iteration~\eqref{eq:fullyDecoupled:norway} defines a contraction.
\end{theorem}
Again, we obtain convergence of the overall method as long as the number of inner iteration steps and the stabilization parameters are sufficiently large. 
%
%
\subsection{F--H--M scheme from~\cite{BruABNR20}}
For comparison, we present a second scheme from~\cite{BruABNR20}. Here, the ordering is changed and one first computes an update for the pressure, followed by the temperature and the displacement.  
In matrix form, the scheme reads 
\begin{align}
	&\begin{bmatrix} 
		\calA & -\calDD^*& -\calD^* \\ 
		0 & \calCCC + L_\theta + \tau \calBB & -\calCC\phantom{^*} \\
		0 & \phantom{-}0\phantom{^*} & \calC + L_p + \tau \calB 
	\end{bmatrix}
	\begin{bmatrix} 
		u^{n+1,k+1} \\ 
		\theta^{n+1,k+1} \\ 
		p^{n+1,k+1} 
	\end{bmatrix} \notag\\
	&\hspace{2.2cm}= 
	\begin{bmatrix} 
		0 & 0 & 0 \\ 
		-\calDD & L_\theta & 0 \\ 
		-\calD & \calCC & L_p 
	\end{bmatrix}
	\begin{bmatrix} 
		u^{n+1,k} \\ 
		\theta^{n+1,k} \\ 
		p^{n+1,k} 
	\end{bmatrix}
	+
	\begin{bmatrix} 
		0 & 0 & 0 \\ 
		\calDD &\phantom{-}\calCCC & -\calCC \\ 
		\calD & - \calCC & \phantom{-}\calC 
	\end{bmatrix}
	\begin{bmatrix} 
		u^{n} \\ 
		\theta^{n} \\ 
		p^{n} 
	\end{bmatrix}
	+ 
	\begin{bmatrix} 
		\phantom{\tau}\f^{n+1} \\ 
		\tau \h^{n+1} \\ 
		\tau \g^{n+1} 
	\end{bmatrix}
	\label{eq:fullyDecoupled:norway2}
\end{align}
for $k = 0, \dots, K$. Therein, $L_\theta$ and $L_p$ denote again stabilization parameters and a convergence result as in Theorem~\ref{thm:norwayConvergence} holds. 
%
%
\subsection{Fully decoupled semi-explicit discretization (M--F--H)}
One possible motivation for the half-decoupled scheme~\eqref{eq:halfDecoupled:semiExpl} was the interpretation as implicit Euler discretization of a delay equation, cf.~the proof of Theorem~\ref{th:halfDecoupled:semiexplicit}. That is, we consider an implicit Euler discretization of 
\[
	\calA u - \calD^*p_\tau - \calDD^* \theta_\tau 
	= \f
\]
rather than~\eqref{eq:thermoporoOperator:a}, where the delay time $\tau$ equals the time step size of the discretization scheme. 
Similarly, we may also replace~\eqref{eq:thermoporoOperator:b} by
\[
	\calD \dot u + \calC \dot p - \calCC \dot \theta_\tau + \calB p 
	= \g.
\]
Since this implements a delay term for the derivative of $\theta$, this results in a two-step scheme, namely 
\begin{align*}
	&\begin{bmatrix} 
	\calA & \phantom{-}0 & 0 \\ 
	\calD & \calC + \tau \calB & 0 \\ 
	\calDD & -\calCC & \calCCC + \tau \calBB 
	\end{bmatrix}
	\begin{bmatrix} u^{n+1} \\ p^{n+1} \\ \theta^{n+1} \end{bmatrix}\\
	&\hspace{2.5cm}= 
	\begin{bmatrix} 
	0 & \phantom{-}\calD^* & \calDD^* \\ 
	\calD & \phantom{-}\calC\phantom{^*} & \calCC\phantom{^*} \\ 
	\calDD & -\calCC\phantom{^*} & \calCCC\phantom{^*}	
	\end{bmatrix}
	\begin{bmatrix} u^{n} \\ p^{n} \\ \theta^{n} \end{bmatrix}
	+ 
	\begin{bmatrix} 0 & 0 & \phantom{-}0 \\ 0 & 0 & -\calCC \\ 0 & 0 & \phantom{-}0 \end{bmatrix}
	\begin{bmatrix} u^{n-1} \\ p^{n-1} \\ \theta^{n-1} \end{bmatrix}
	+ 
	\begin{bmatrix} \f^{n+1} \\ \tau \g^{n+1} \\ \tau \h^{n+1} \end{bmatrix}.
\end{align*}
This scheme obviously decouples, since the iteration matrix is block triangular. Nonetheless, we propose to apply an analog manipulation also to the third equation of the system, i.e., to~\eqref{eq:thermoporoOperator:c}. This provides a certain symmetry and yields the scheme 
\begin{align}
	&\begin{bmatrix} 
		\calA & 0 & 0 \\ 
		\calD & \calC + \tau \calB & 0 \\ 
		\calDD & 0& \calCCC + \tau \calBB 
	\end{bmatrix}
	\begin{bmatrix} 
		u^{n+1} \\ 
		p^{n+1} \\ 
		\theta^{n+1} 
	\end{bmatrix} \notag \\
	&\hspace{2.5cm}= 
	\begin{bmatrix} 
		0 & \calD^* & \calDD^* \\ 
		\calD & \calC\phantom{^*} & \calCC\phantom{^*} \\ 
		\calDD & \calCC\phantom{^*} & \calCCC\phantom{^*}
	\end{bmatrix}
	\begin{bmatrix} 
		u^{n} \\ 
		p^{n} \\ 
		\theta^{n} 
	\end{bmatrix}
	+ 
	\begin{bmatrix} 
		0 & \phantom{-}0 & \phantom{-}0 \\ 
		0 & \phantom{-}0 & -\calCC \\ 
		0 & -\calCC & \phantom{-}0 
	\end{bmatrix}
	\begin{bmatrix} 
		u^{n-1} \\ 
		p^{n-1} \\ 
		\theta^{n-1} 
	\end{bmatrix}
	+ 
	\begin{bmatrix} 
		\phantom{\tau}\f^{n+1} \\ 
		\tau \g^{n+1} \\ 
		\tau \h^{n+1} 
	\end{bmatrix}. 
	\label{eq:fullyDecoupled:semiExpl}
\end{align}
Note that this scheme coincides with~\eqref{eq:fullyDecoupled:russia} for $\sigma = 1$.
\subsubsection{Corresponding delay system}
Based on the operator form~\eqref{eq:thermoporoOperator} and as motivated above, we consider the delay system 
\begin{subequations}
\label{eq:delay:thermoporoOperator}
\begin{align}
	\calA {\bar u} - \calD^*{\bar p}_\tau - \calDD^* {\bar \theta}_\tau 
	&= \f, \label{eq:delay:thermoporoOperator:a} \\
	\calD \dot {\bar u} + \calC \dot {\bar p} - \calCC \dot {\bar \theta}_\tau + \calB {\bar p} 
	&= \g, \label{eq:delay:thermoporoOperator:b} \\
	\calDD \dot {\bar u} + \calCCC \dot {\bar \theta} - \calCC \dot {\bar p}_\tau  + \calBB {\bar \theta}
	&= \h. \label{eq:delay:thermoporoOperator:c} 
\end{align}
\end{subequations}
As initial data, we set ${\bar p}(0) = p^0$ and ${\bar \theta} = \theta^0$ as before. Moreover, we need history functions for $p$ and $\theta$ on the time interval $[-\tau,0]$. For this, we introduce smooth functions $\Phi_p, \Phi_\theta\colon [-\tau,0] \to \Q$ with 
\begin{align}
	\label{eq:history:ptheta}
	\Phi_p(-\tau) = \Phi_p(0) = p^0,\qquad
	\Phi_\theta(-\tau) = \Phi_\theta(0) = \theta^0
\end{align}
and conclude that ${\bar u}(0) = u(0) = u^0$ is again a consistent initial value for the deformation variable. 
\begin{lemma}[Difference of original and delay system]
Assume sufficiently smooth right-hand sides~$\f$, $\g$, $\h$ and smooth history functions $\Phi_p$, $\Phi_\theta$ that fulfill~\eqref{eq:history:ptheta} such that the solution~$({\bar u}, {\bar p}, {\bar \theta})$ of the delay system~\eqref{eq:delay:thermoporoOperator} satisfies~${\bar p}$, ${\bar\theta} \in W^{2,\infty}(\cHQ)$. Further assume $3\,\cc_0 \le 2 c_0$ as well as $3\,\cc_0 \le 2 \ccc_0$. Then the solution~$(u,p,\theta)$ of~\eqref{eq:thermoporoOperator} and $({\bar u}, {\bar p}, {\bar \theta})$ only differ by a term of order~$\tau$. More precisely, we have
\[
	\|{\bar p}(t) - p(t)\|_\Q^2 + \|{\bar \theta}(t) - \theta(t)\|_\Q^2 
	\lesssim \tau^2\, t\, \Big[ \| \ddot{\bar p}\|_{L^\infty(\cHQ)}^2 + \| \ddot{\bar \theta}\|_{L^\infty(\cHQ)}^2 \Big]
\]
and 
\[
	\|{\bar u}(t) - u(t)\|_\V^2 
	\lesssim \tau^2\, t\, \Big[ \| {\bar p}\|_{W^{2,\infty}(\cHQ)}^2 + \| {\bar \theta}\|_{W^{2,\infty}(\cHQ)}^2 \Big]
\]
for almost all $t\in [0,T]$.
\end{lemma}
\begin{proof}
Considering the difference of~\eqref{eq:thermoporoOperator} and~\eqref{eq:delay:thermoporoOperator}, we observe that the errors
\[
	e_u \coloneqq \bar u - u, \qquad
	e_p \coloneqq \bar p- p, \qquad
	e_\theta \coloneqq \bar \theta - \theta
\]
satisfy the system
\begin{subequations}
\label{eq:thermoporoOperator:err}
\begin{align}	
	\calA e_u - \calD^*e_p - \calDD^* e_\theta 
	&= - \calD^*(\bar p - {\bar p}_\tau) - \calDD^* ({\bar \theta} - {\bar \theta}_\tau), \label{eq:thermoporoOperator:err:a} \\
	\calD \dot e_u + \calC \dot e_p - \calCC \dot e_\theta + \calB e_p 
	&= - \calCC\,(\dot{\bar \theta} - \dot{\bar \theta}_\tau), \label{eq:thermoporoOperator:err:b} \\
	\calDD \dot e_u + \calCCC \dot e_\theta - \calCC \dot e_p + \calBB e_\theta 
	&= - \calCC\,(\dot{\bar p} - \dot{\bar p}_\tau). \label{eq:thermoporoOperator:err:c} 
\end{align}
\end{subequations}
Due to the special choice of the history functions, the initial errors satisfy $e_u(0) = 0$ and $e_p(0) = e_\theta(0) = 0$. Using first-order Taylor expansions, the right-hand side of~\eqref{eq:thermoporoOperator:err:a} can be written as 
\[
- \calD^*(\bar p - {\bar p}_\tau) - \calDD^* ({\bar \theta} - {\bar \theta}_\tau) = - \tau\calD^* \dot{\bar p}(\xi_p) - \tau\calDD^* \dot{\bar \theta} (\xi_\theta)
\]
for some  $\xi_p,\xi_\theta\in[t-\tau, t]$. 
In the same way, we can replace the right-hand sides of~\eqref{eq:thermoporoOperator:err:b} and~\eqref{eq:thermoporoOperator:err:c} using
\[
	\calCC\,(\dot{\bar \theta} - \dot{\bar \theta}_\tau)
	= \tau\,\calCC \ddot{\bar \theta} (\zeta_\theta), \qquad 
	\calCC\,(\dot{\bar p} - \dot{\bar p}_\tau)
	= \tau\,\calCC \ddot{\bar p} (\zeta_p)
\]
for some $\zeta_p, \zeta_\theta \in[t-\tau, t]$. 
Considering the derivatives of~\eqref{eq:thermoporoOperator:a} and~\eqref{eq:delay:thermoporoOperator:a}, we further obtain the equation 
\begin{align}	
	\calA \dot e_u - \calD^*\dot e_p - \calDD^* \dot e_\theta 
	= - \tau\calD^* \ddot{\bar p}(\xi'_p) - \tau\calDD^* \ddot{\bar \theta} (\xi'_\theta) \label{eq:thermoporoOperator:err:d} 
\end{align}
for some $\xi'_p,\xi'_\theta\in[t-\tau, t]$.

Next, we take the sum of~\eqref{eq:thermoporoOperator:err:d}, \eqref{eq:thermoporoOperator:err:b}, and~\eqref{eq:thermoporoOperator:err:c} with test functions $\dot e_u$, $\dot e_p$, and~$\dot e_\theta$, respectively. Using Young's inequality, we estimate 
\begin{align*}
	\|\dot e_u\|_a^2 + \|\dot e_p&\|_c^2 + \|\dot e_\theta\|_{\ccc}^2 
	+ \frac12 \ddt\, \|e_p\|_b^2 + \frac12 \ddt\, \|e_\theta\|_{\bb}^2 \\
	&= 2\, \cc(\dot e_\theta, \dot e_p) 
	- \tau\, \Big[ d(\dot e_u, \ddot{\bar p}(\xi'_p)) 
	+ \dd(\dot e_u, \ddot{\bar \theta}(\xi'_\theta)) 
	- \cc(\dot e_p, \ddot {\bar \theta}(\zeta_\theta)) 
	- \cc(\dot e_\theta, \ddot {\bar p}(\zeta_p)) \Big] \\
	&\le 2\, \cc_0 \|\dot e_\theta\|_\cHQ \|\dot e_p\|_\cHQ
	+ \tau\, C_d\, \|\dot e_u\|_\V \|\ddot {\bar p}\|_{L^\infty(\cHQ)} 
	+ \tau\, C_{\dd}\, \|\dot e_u\|_\V \|\ddot {\bar \theta}\|_{L^\infty(\cHQ)} \\
	&\qquad+ \tau\, \cc_0\, \|\dot e_p\|_\cHQ\|\ddot {\bar \theta}\|_{L^\infty(\cHQ)} 
	+ \tau\, \cc_0\, \|\dot e_\theta\|_\cHQ\|\ddot {\bar p}\|_{L^\infty(\cHQ)}\\
	&\le \cc_0\, \Big[ \|\dot e_\theta\|^2_\cHQ + \|\dot e_p\|^2_\cHQ \Big] \\
	&\qquad+ \frac{c_a}{2}\, \|\dot e_u\|^2_\V  
	+ \frac{\tau^2}{2} \frac{C_d^2}{c_a}\, \|\ddot {\bar p}\|^2_{L^\infty(\cHQ)}
	+ \frac{c_a}{2}\, \|\dot e_u\|_\V^2  
	+ \frac{\tau^2}{2} \frac{C_{\dd}}{c_a}\, \|\ddot {\bar \theta}\|_{L^\infty(\cHQ)}^2 \\
	&\qquad+ \frac{\cc_0}{2}\, \Big[ \|\dot e_p\|^2_\cHQ +  \tau^2\,\|\ddot {\bar \theta}\|_{L^\infty(\cHQ)}^2 \Big] 
	+ \frac{\cc_0}{2}\, \Big[ \|\dot e_\theta\|^2_\cHQ + \tau^2\,\|\ddot {\bar p}\|_{L^\infty(\cHQ)}^2 \Big]\\
	&= c_a \|\dot e_u\|^2_\V 
	+ \frac32\, \cc_0\, \Big[ \|\dot e_\theta\|^2_\cHQ + \|\dot e_p\|^2_\cHQ \Big] 
	+ C\, \tau^2 \Big[ \|\ddot {\bar p}\|_{L^\infty(\cHQ)}^2 + \|\ddot {\bar \theta}\|_{L^\infty(\cHQ)}^2 \Big].
\end{align*} 
Due to the ellipticity of $a$ and the assumptions on~$c_0$ and~$\ccc_0$, we can absorb the terms including~$\dot e_u, \dot e_p, \dot e_\theta$. An integration over the time interval $[0,t]$ then yields the stated estimate for the pressure and temperature variables. 

For an estimate of~$e_u$, we consider~\eqref{eq:thermoporoOperator:err:a} with test function~$e_u$, leading to 
\begin{align*}
	\| e_u\|_a^2
	&= d(e_u,e_p) + \dd(e_u, e_\theta)
	- \tau d(e_u, \dot{\bar p}(\xi_p)) - \tau \dd(e_u, \dot{\bar \theta} (\xi_\theta)) \\
	&\le C_d\, \|e_u\|_\V \|e_p\|_\cHQ + C_{\dd}\, \|e_u\|_\V \|e_\theta\|_\cHQ
	\\&\hspace{2.cm}+ \tau C_d\, \| e_u\|_\V \|\dot{\bar p}\|_{L^\infty(\cHQ)} + \tau C_{\dd}\, \| e_u\|_\V \|\dot{\bar \theta}\|_{L^\infty(\cHQ)} \\
	&\le \frac12\, \|e_u\|_a^2  
	+ 2\, \frac{C_d^2}{c_a} \Big( \|e_p\|^2_\cHQ + \tau^2\, \|\dot{\bar p}\|^2_{L^\infty(\cHQ)} \Big)
	+ 2\, \frac{C_{\dd}^2}{c_a} \Big( \|e_\theta\|^2_\cHQ
	 + \tau^2\, \|\dot{\bar \theta}\|^2_{L^\infty(\cHQ)} \Big).
\end{align*}
Hence, a combination with the above estimates for pressure and temperature finishes the proof. 
\end{proof}
\subsubsection{Convergence}
Finally, we prove convergence of the scheme~\eqref{eq:fullyDecoupled:semiExpl} by applying the abstract result from Section~\ref{sec:convergence}.  
\begin{theorem}[first-order convergence, fully decoupled]
\label{th:fullyDecoupled:semiexplicit}
Assume sufficiently smooth right-hand sides as well as 
\[
	\min(c^2_d, c^2_{\tilde d}) > C_a\cc_0
	\qquad\text{and}\qquad 
	\omega_\mathrm{FD} 
	\coloneqq \frac{C_{d}^2 + C_{\dd}^2 + c_a\cc_0}{c_a \min(c_0, \ccc_0)}
	\le 1
\]
where $c_d$ and $c_{\tilde d}$ are the inf--sup constants of $\calD$ and $\calDD$, respectively. Then the scheme~\eqref{eq:fullyDecoupled:semiExpl} converges with order one. 
\end{theorem}
\begin{proof}
With the solution vector~$\bar \pbf \coloneqq [\bar p; \bar \theta]$, the delay vector~$\bar \pbf_\tau \coloneqq [{\bar p}_\tau; {\bar \theta}_\tau]$, and right-hand side $\gbf \coloneqq [\g; \h]$, we can write the system \eqref{eq:delay:thermoporoOperator} as
\begin{subequations}
	\label{eq:delayThermoPoro}
\begin{align}
	\calA {\bar u} - \bbD^*{\bar \pbf}_\tau 
	&= \f, \label{eq:delayThermoPoro:a}\\
	\bbD \dot {\bar u} + \begin{bmatrix}\ \calC & \phantom{-\calCC\ }\ \\ \phantom{\calCC} & \phantom{-}\calCCC\ \ \end{bmatrix}\dot {\bar \pbf} - \begin{bmatrix} \phantom{-\calC} & \calCC\ \ \\ \calCC & \phantom{-\calCCC\ }\ \end{bmatrix} \dot {\bar \pbf}_\tau + \bbB {\bar \pbf} 
	&= \gbf \label{eq:delayThermoPoro:b}
\end{align}
\end{subequations}
with
\[
\bbB
\coloneqq \begin{bmatrix} \calB & \\ & \calBB \end{bmatrix}, \qquad
%
%
\bbD
\coloneqq \begin{bmatrix} \calD \\ \calDD \end{bmatrix}
\]
as before. 
Solving~\eqref{eq:delayThermoPoro:a} for $\bar u$, we obtain with~\eqref{eq:delayThermoPoro:b} 
\begin{align*}
	\begin{bmatrix}\ \calC & \phantom{-\calCC\ }\ \\ \phantom{\calCC} & \phantom{-}\calCCC\ \ \end{bmatrix}\dot {\bar \pbf} + \bbB {\bar \pbf} + \left( \bbD \calA^{-1}\bbD^* -\begin{bmatrix} \phantom{-\calC} & \calCC\ \ \\ \calCC & \phantom{-\calCCC\ }\ \end{bmatrix} \right)\dot {\bar \pbf}_\tau
	&= \gbf - \bbD \calA^{-1}\dot f.
\end{align*}
For the application of Theorem~\ref{th:convergence:auxEquation}, we define the operators
\begin{align*}
	\calE 
	\coloneqq\begin{bmatrix}\ \calC & \phantom{-\calCC\ }\ \\ \phantom{\calCC} & \phantom{-}\calCCC\ \ \end{bmatrix}, \qquad  
	\calK 
	\coloneqq \bbB,\qquad 
	\calM 
	\coloneqq \bbD \calA^{-1}\bbD^* -\begin{bmatrix} \phantom{-\calC} & \calCC\ \ \\ \calCC & \phantom{-\calCCC\ }\ \end{bmatrix}. 
\end{align*}
It is easy to verify that the operators $\calE$ and $\calK$ are continuous, symmetric, and elliptic with constants $c_\calE = \min(c_0, \ccc_0)$ and $c_\calK = c_{\bbB} \ge \min(c_{b}, c_{\bb})$, respectively. It remains to show that $\calM$ is symmetric, elliptic, and continuous. For the symmetry, we consider vectors $\pbf = [p_1; p_2], \qbf = [q_1; q_2] \in \Q^2$ and apply the symmetry of $\calA$, leading to 
\begin{align*}
	\big( \bbD \calA^{-1}\bbD^* \pbf, \qbf\big)_{\cHQsquare} 
	= \big( \calA^{-1}\bbD^* \pbf, \bbD^* \qbf\big)_{\cHQsquare}
	= \big( \pbf, \bbD \calA^{-1}\bbD^*  \qbf\big)_{\cHQsquare}.
\end{align*}
To prove the ellipticity, we apply the inf--sup stability of $\bbD$, which follows from the fact that~$\calD$ and~$\calDD$ are inf--sup stable and only differ by a constant. Hence, with $c_\bbD = \min(c_d, c_{\tilde d})$ we have 
\begin{align*}
	(\calM \qbf, \qbf)_{\cHQsquare}
	&\ge \langle\calA^{-1}\bbD^*\qbf, \bbD^*\qbf\rangle
	- \bigg\langle\begin{bmatrix} \phantom{-\calC} & \calCC\ \\ \calCC & \phantom{-\calCCC}\ \end{bmatrix}\qbf, \qbf\bigg\rangle\\
	%
	%
	&\ge \frac{c^2_\bbD}{C_a}\, \|\qbf\|^2_{\cHQsquare}
	- 2\, \cc_0\, \|q_1\|_{\cHQ} \|q_2\|_{\cHQ}
	\ge \left(\frac{c^2_\bbD}{C_a} - \cc_0\right)\, \|\qbf\|^2_{\cHQsquare},
\end{align*}
which is positive under the condition $c^2_\bbD = \min(c^2_d, c^2_{\tilde d}) > C_a\cc_0$. Finally, using Lemma~\ref{lem:propertiesBlockOperators} we compute 
\begin{align*}
	\|\calM \qbf \|_{(\cHQdual)^2} 
	%
	\le \left\|\bbD \calA^{-1}\bbD^*\qbf\right\|_{(\cHQdual)^2} 
	+ \left\|\begin{bmatrix} \phantom{-\calC} & \calCC\ \ \\ \calCC & \phantom{-\calCCC\ }\ \end{bmatrix}\qbf\ \right\|_{(\cHQdual)^2} 
	%
	%
	%
	\le \Big( \frac{C_{d}^2 + C_{\dd}^2}{c_a} + \cc_0 \Big)\, \|\qbf\|_{\cHQsquare}.
\end{align*}
Hence, following Theorem~\ref{th:convergence:auxEquation}, the fully decoupled semi-explicit scheme converges with order 1 as long as $C_\calM \le c_\calE$, which is satisfied if $\omega_\mathrm{FD} \le 1$.
\end{proof}
%
%
\section{Numerical Examples}
\label{sec:num}
In this final section, we provide numerical examples and computational comparisons of all presented methods. Moreover, we study the sharpness of the sufficient convergence conditions on a model problem. 
%
%
\subsection{Spatial discretization}
\label{sec:num:spaceDisc}
For the spatial discretization of the upcoming geothermal problem, we consider piecewise linear, respectively quadratic, finite elements for all three components. This turns the original system~\eqref{eq:thermoporoOperator} into the differential--algebraic equation system 
\begin{align*}
	\begin{bmatrix} 0 & \phantom{-}0 & \phantom{-}0 \\ D & \phantom{-}C & - \CC \\ \DD & -\CC & \phantom{-}\CCC	\end{bmatrix}
	\begin{bmatrix} \dot u \\ \dot p \\ \dot \theta \end{bmatrix}
	= 
	\begin{bmatrix} -A & \phantom{-}D^\top & \phantom{-}\DD^\top \\ & -B\phantom{^\top} & \\ & & -\BB\phantom{^\top} \end{bmatrix}
	\begin{bmatrix} u \\ p \\ \theta \end{bmatrix}
	+ 
	\begin{bmatrix} \f \\ \g \\ \h \end{bmatrix}.
\end{align*}
Note that we use the same notation for the semi-discrete vectors as for the corresponding original variables. According to the assumptions from Section~\ref{sec:model}, the matrices $A$, $B$, $\BB$, $C$, $\CC$, and $\CCC$ are symmetric positive definite. In this setting, the ellipticity (respectively continuity) constants 
can be identified with the smallest (respectively largest) eigenvalues of the respective matrices. 
Furthermore, we observe that also the matrix 
\[
	\begin{bmatrix} 
		\phantom{-}C & - \CC \\ 
		-\CC & \phantom{-}\CCC	
	\end{bmatrix}
\]
is symmetric positive definite if Assumption~\ref{ass:chat} is satisfied. Concerning the coupling matrices $D$ and $\DD$, we assume that the smallest singular values are positive. Note that these correspond to the inf--sup constants $c_D$ and $c_{\DD}$, respectively. The continuity constants $C_D$ and $C_{\DD}$ can be identified with the respective largest singular values.
\begin{remark}
The mentioned properties of the involved matrices imply that the resulting differential--algebraic equation has index~$1$; see \cite[Sect.~2.2]{BreCP96} for a precise definition and further details. 
\end{remark}
In all following experiments, we use the fully implicit midpoint rule with constant step size~$\tau$ and a triangulation with mesh size $h_\text{ref}$ to obtain a reference solution $(u^n_\text{ref}, p^n_\text{ref}, \theta^n_\text{ref})$. Here, the index $n$ corresponds to the time point $t^n = n\tau$. The overall error is then computed as the sum of the relative errors in the respective $A$-, $C$-, and $\CCC$-norms at the final time $t = T = N\tau$. More precisely, we define 
\[
	e_T
	= e_T(u_h^N,p_h^N,\theta_h^N)
	\coloneqq \frac{\|u^N_\text{ref} - u_h^N\|_A}{\|u^N_\text{ref}\|_A} 
	+ \frac{\|p^N_\text{ref} - p_h^N\|_C}{\|p^N_\text{ref}\|_C} 
	+ \frac{\|\theta^N_\text{ref} - \theta_h^N\|_{\CCC}}{\|\theta^N_\text{ref}\|_{\CCC}}.
\]
The implementation is realized with DOLFIN 2019.2.0 and the experiments are performed on a Lenovo ThinkPad X1 2-in-1 Gen 10 with an Intel Core Ultra 7 255U processor.
%
%
\subsection{Geothermal model problem}
\label{sec:num:geothermal}
As first numerical example, we consider the equations of thermo--poroelasticity on the unit square in two space dimensions with homogeneous Dirichlet boundary conditions. 
The used parameters are motivated from the experiments in~\cite{AntBB23}, slightly adjusted such that the weak coupling conditions are satisfied. More precisely, we set
\begin{center} 
	\begin{tabular}{c@{\quad} c@{\quad}|@{\quad} c@{\quad} c@{\quad}|@{\quad} c@{\quad} c@{\quad} c@{\quad} c}                          
		$\lambda$ & $\mu$  & $\frac{\kappa}{\nu}$ & $\widetilde \kappa$ & $c_0$ & $\cc_0$&$\ccc_0$ \\  
		\toprule                                           
		$1.2 \cdot 10^{10}$ & $6.0 \cdot 10^9$    &  $6.33 \cdot 10^2$   & $10^2$ &  $7.8 \cdot 10^3$   &$3.03 \cdot 10^{-11}$& $0.92\cdot 10^3 $\\ 
	\end{tabular}
\end{center}
together with the coupling parameters 
\[
	\alpha 
	= 0.97, \qquad
	\beta 
	= 3.96 \cdot 10^{6}.
\]
In order to investigate the convergence order in time, we solve the problem with the same spatial method as for the reference solution, i.e., $P1$--$P1$--$P1$ finite elements with mesh size $h_\text{ref} = h = 0.125$. 
\subsubsection{Half-decoupled schemes}
We compare the half-decoupled schemes presented in Section~\ref{sec:halfDecoupled}. 
For scheme~\eqref{eq:halfDecoupled:norway} from~\cite{AntBB23}, we set the stabilization parameters to $L_p = L_{\theta} = 0.025$ in a first run and $L_p = L_{\theta} = 0.001$ in a second. For both setups, we use $K = 5, 10$, and $20$ inner iterations. Figure~\ref{fig:halfDecoupled} shows that the error barely decreases for $K=5$ for both sets of stabilization parameters and even increases for $\tau < 2^{-4}$. For $K=10$ (and similarly for $K=20$), one observes a convergent behavior up to $\tau \approx 2^{-6}$ and an increase of the error afterwards. Such a divergent behavior for small step sizes is in line with the results of~\cite{AltMU24b}, which show the existence of an error term $\frac{\operatorname{tol}}{\tau}$ in the case of a similar iteration scheme for poroelasticity.
	
For the newly introduced semi-explicit scheme~\eqref{eq:halfDecoupled:semiExpl}, we first need to justify the convergence by computing 
\[
	\omega_\mathrm{HD} 
	= \frac{C_d^2 + C^2_{\dd}}{c_a\, \min\{ c_0-\cc_0, \ccc_0-\cc_0 \}}
	\approx 0.947   
	\le 1.
\]
Hence, Theorem~\ref{th:halfDecoupled:semiexplicit} guarantees first-order convergence, which is also observed in Figure~\ref{fig:halfDecoupled}.
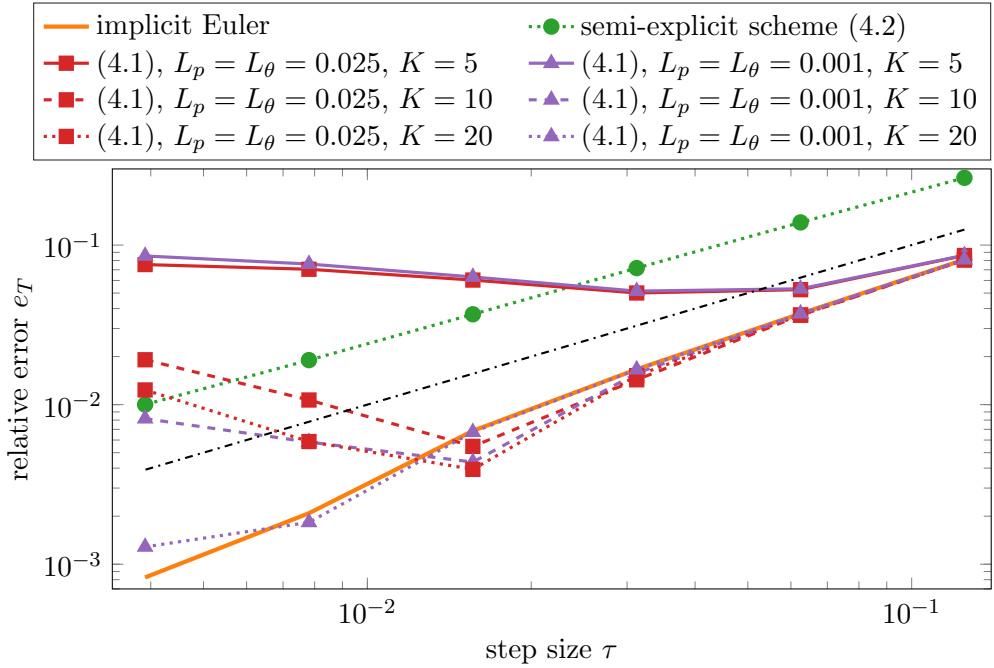
\begin{figure}
	\input{pics/conv_halfDecoupled_h0125}
	\caption{Convergence history for the half-decoupled schemes for the geothermal model problem from Section~\ref{sec:num:geothermal} with fixed mesh size $h=0.125$. The dash-dotted line shows order $1$.}
	\label{fig:halfDecoupled} 
\end{figure}
\subsubsection{Fully decoupled schemes}
We now turn to the fully decoupled schemes presented in Section~\ref{sec:fullyDecoupled} for the same problem. The resulting convergence histories are presented in Figure~\ref{fig:fullyDecoupled}.

For scheme~\eqref{eq:fullyDecoupled:russia}, we apply different stabilization parameters. For~$\sigma = 0.5$, the scheme does not converge. For $\sigma = 0.77$, one observes a convergent behavior up to $\tau \approx 2^{-5}$ and an increase of the error afterwards, cf.~Figure~\ref{fig:fullyDecoupled}. Recall that the scheme~\eqref{eq:fullyDecoupled:russia} coincides to the newly presented scheme~\eqref{eq:fullyDecoupled:semiExpl} for $\sigma = 1$. Hence, we have first-order convergence in this case. For large values of $\sigma$, the rate of convergence slows down. For $\sigma = 200$, for example, the error reduces but only reaches a relative error of $2.5$ for the finest time step size. Hence, we did not include this in the convergence plot.  

For scheme~\eqref{eq:fullyDecoupled:norway} from~\cite{AntBB23}, we consider the stabilization parameters $L_p = L_\theta = 0.025$. Similar to the half-decoupled scheme, one observes a convergent behavior up to $\tau \approx 2^{-5}$ for $K = 5$ and up to $\tau \approx 2^{-6}$ for $K=10$. For $K=20$, scheme~\eqref{eq:fullyDecoupled:norway} converges for all considered time step sizes. 
Scheme~\eqref{eq:fullyDecoupled:norway2} behaves similarly as one can see for $K=20$ and $L_p = L_\theta = 0.025$.

Finally, for the convergence of the newly introduced scheme~\eqref{eq:fullyDecoupled:semiExpl}, cf.~Theorem~\ref{th:fullyDecoupled:semiexplicit}, we have that  
\[
	\omega_\mathrm{FD} 
	= \frac{C_{d}^2 + C_{\dd}^2 + c_a\cc_0}{c_a \min(c_0, \ccc_0)}
	\approx 0.947 
	\le 1.
\]
Indeed, we observe first-order convergence for this scheme. Recall that the scheme gets along without any inner iteration. The resulting errors, however, are larger compared to the other schemes if the same time step size is used. 
\begin{figure}
	\input{pics/conv_fullyDecoupled_h0125}
	\caption{Convergence history for the fully decoupled schemes for the geothermal model problem from Section~\ref{sec:num:geothermal} with fixed mesh size $h=0.125$. The dash-dotted line shows order $1$. }
	\label{fig:fullyDecoupled} 
\end{figure}
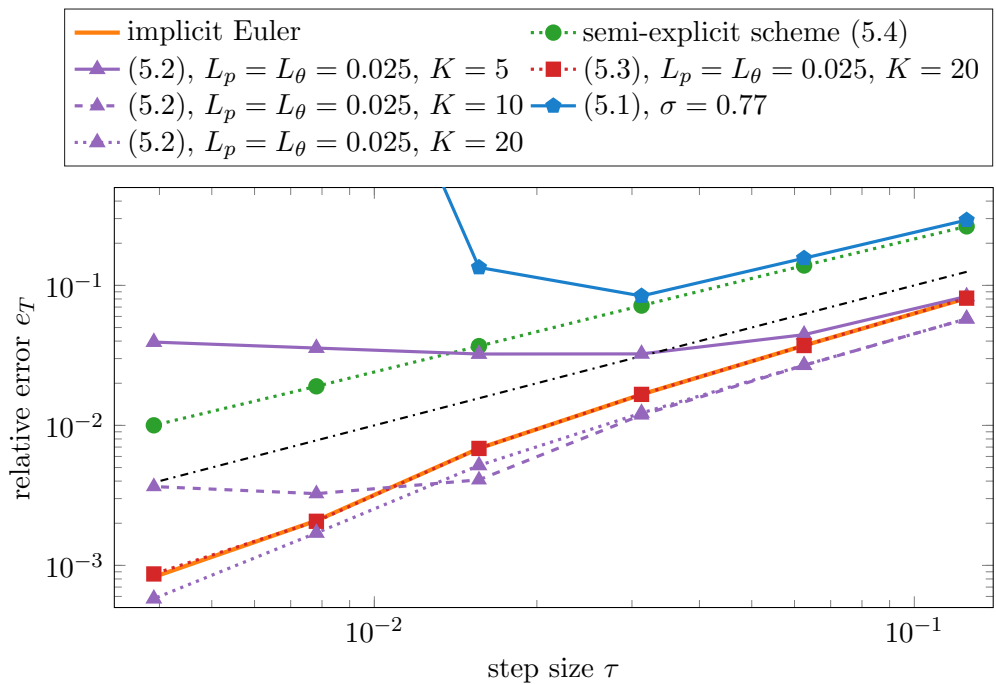
%
%
\subsection{Convergence in space and time}
In this experiment, we include the spatial errors. For this, we compute the reference solution again using the implicit midpoint rule but on a finer spatial grid. More precisely, we have $h_\text{ref} = 0.002$ for the computations using $P1$--$P1$--$P1$ finite elements and $h_\text{ref} = 0.0078$ within the $P2$--$P1$--$P1$ setting. Note that the latter approach is motivated by~\cite{ErnM09}.  

The physical parameters are chosen as in the previous example. In Figure~\ref{fig:spatialError}, we compare the implicit Euler scheme with the semi-explicit schemes~\eqref{eq:halfDecoupled:semiExpl} and \eqref{eq:fullyDecoupled:semiExpl} for different spatial mesh sizes. As expected, one can observe first-order convergence (in time) for all methods up to a certain plateau, which is defined by the spatial accuracy. 
%
\begin{figure}	
	\definecolor{mycolor0}{rgb}{0.12156862745098,0.466666666666667,0.705882352941177} 
	\definecolor{mycolor1}{rgb}{0.00000,0.44700,0.74100}
	\definecolor{mycolor2}{rgb}{0.85000,0.32500,0.09800}
	\definecolor{mycolor3}{rgb}{0.49400,0.18400,0.55600}
	\definecolor{mycolor4}{rgb}{0.92900,0.69400,0.12500}
	\definecolor{mycolor5}{rgb}{0.46600,0.67400,0.18800}
	\definecolor{mycolor6}{rgb}{0.30100,0.74500,0.93300}
	\definecolor{mycolor7}{rgb}{0.63500,0.07800,0.18400}
	\input{pics/conv_Exp2_a_R.tex}
	\input{pics/conv_Exp2_b_R.tex}
	\caption{Relative errors at the final time point $t=T=1$ for different
		spatial mesh sizes, namely $h=0.0625$ (dash-dotted), $h=0.0312$ (solid), $h=0.0156$ (dashed), and~$0.0078$ (dotted). For the spatial discretization, we use $P1$--$P1$--$P1$ (left) and $P2$--$P1$--$P1$ finite elements (right).
	}
	\label{fig:spatialError} 
\end{figure}
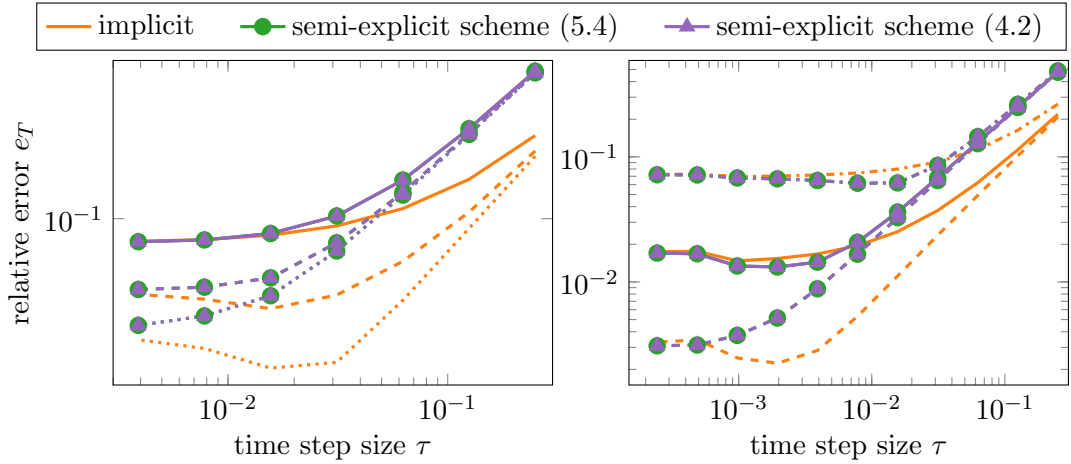
%
%
\subsection{Sharpness of the sufficient conditions}
In this final experiment, we study the sharpness of the weak coupling condition for the half decoupled scheme~\eqref{eq:halfDecoupled:semiExpl} as well as for the fully decoupled scheme~\eqref{eq:fullyDecoupled:semiExpl}. For this, we consider a toy problem with stiffness and diffusion matrices
\[
	A = \frac{1}{2-\sqrt2}\,\begin{bmatrix}2&-1&0\\-1&2&-1\\ 0&-1&2\end{bmatrix}, \qquad 
	B = 2, \qquad \BB = 1.
\]
For the inverse of the Biot modulus, and the thermal delatation coefficient, we set 
\[
	C 
	= 2, \qquad\quad
	\CC
	= 0.5.
\]
Finally, we introduce the coupling matrices 
\begin{align*}
	D 
	= \alpha\, \begin{bmatrix}2&1&2\end{bmatrix}, \qquad
	\DD 
	= \beta\, \begin{bmatrix}2&1&2\end{bmatrix}
\end{align*}
which scale with the parameters $\alpha, \beta$. In the following, we consider the case $\alpha = \beta \in (0,0.64)$. The thermal capacity~$\ccc_0$ is selected within the interval $(\cc_0, c_0+3.5)$. 

We set the time step size to $\tau = 0.1\cdot 2^{-8}$. As reference solution we use the solution of the implicit Euler scheme with time step size $\tau_\text{ref} = 0.1\cdot 2^{-9}$. 
In Figure~\ref{fig:sharpness:halfDecoupled:semiExpl1}, we visualize the convergence properties of the half decoupled scheme~\eqref{eq:halfDecoupled:semiExpl} in comparison with the corresponding weak coupling condition. The area where the weak coupling condition is fulfilled, is colored in green. Here, convergence is guaranteed by the presented convergence analysis. The yellow region indicates that the sufficient condition is not satisfied but convergence is still obtained. By this we mean that the relative error is smaller than $10^{-2}$. Finally, the method is not applicable (and actually diverges) in the red regions. 

The corresponding study for the fully decoupled scheme~\eqref{eq:fullyDecoupled:semiExpl} with the same color coding is shown in Figure~\ref{fig:sharpness:fullyDecoupled:semiExpl2}. 
\begin{figure}	
	\input{pics/pictureExperimentConvergency3col}
	\caption{Illustration of the sharpness of the weak coupling condition guaranteeing convergence for the half decoupled scheme~\eqref{eq:halfDecoupled:semiExpl}. }
	\label{fig:sharpness:halfDecoupled:semiExpl1} 
\end{figure}
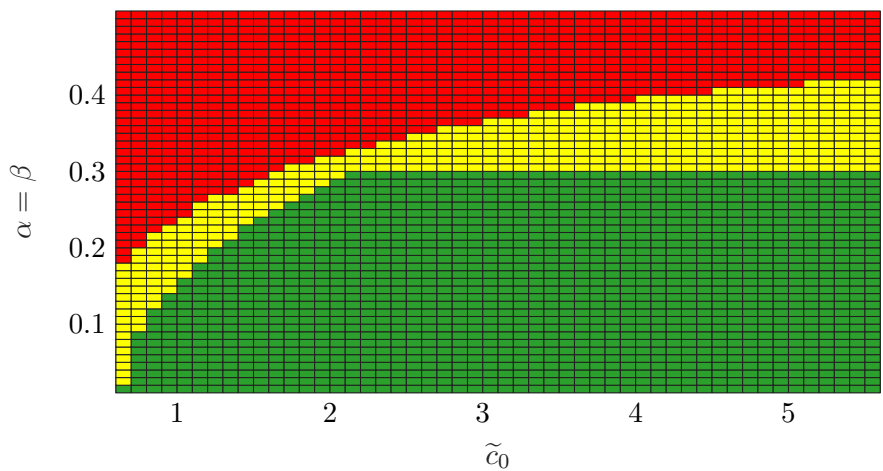
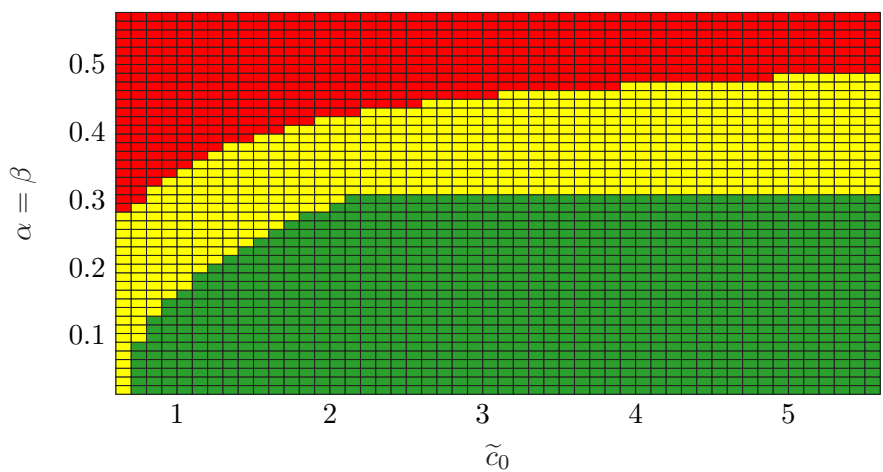
\begin{figure}	
	\input{pics/pictureExperimentConvergencyfullydec3col}
	\caption{Illustration of the sharpness of the weak coupling condition guaranteeing convergence for the fully decoupled scheme~\eqref{eq:fullyDecoupled:semiExpl}. }
	\label{fig:sharpness:fullyDecoupled:semiExpl2}
\end{figure}
%
%
\section{Conclusions}
In this work, we have proposed fully and partially decoupled time discretization schemes for linear thermo--poroelasticity. They allow to solve the three involved equations for the displacement, the pressure, and the temperature successively in each time step, avoiding a more involved (bigger) system matrix. Under a suitable coupling condition, which can be verified by the material parameters alone, we have shown first-order convergence of the introduced schemes. To achieve this, we have first proven an abstract convergence result for a delay equation of parabolic type and showed that the scheme can be reformulated as an implicit Euler scheme for such a delay equation. Finally, we presented a set of numerical experiments to illustrate the theoretical findings and compared our schemes to existing decoupling approaches. 
%
%
\section*{Acknowledgments} 
R.~Altmann acknowledges the support of the Deutsche Forschungsgemeinschaft (DFG, German Research Foundation) through the project 467107679.  
%
%
\bibliographystyle{alpha}
\bibliography{references}
\end{document}

%% file: def.tex
\definecolor{myBlue}{RGB}{30,144,255} 
\definecolor{myGreen}{RGB}{69,169,0} 
\definecolor{myRed}{RGB}{165,12,42} 
\definecolor{myOrange}{RGB}{225,92,22} 
\definecolor{color0}{rgb}{0.12156862745098,0.466666666666667,0.705882352941177}
\definecolor{color1}{rgb}{1,0.498039215686275,0.0549019607843137}
\definecolor{color2}{rgb}{0.172549019607843,0.627450980392157,0.172549019607843}
\definecolor{color3}{rgb}{0.83921568627451,0.152941176470588,0.156862745098039}
\definecolor{color4}{rgb}{0.580392156862745,0.403921568627451,0.741176470588235}
\definecolor{color5}{rgb}{0.1,0.5,0.8}

\newcommand{\delete}[1]{ }

\def\N{\mathbb{N}}

\def\R{\mathbb{R}}

\newcommand\dx{\,\mathrm{d}x}

\newcommand{\ddt}{\ensuremath{\frac{\mathrm{d}}{\mathrm{d}t}} }

\DeclareMathOperator{\ddiv}{div}

\DeclareMathOperator{\trace}{trace}

\newcommand{\calA}{\ensuremath{\mathcal{A}} }

\newcommand{\calB}{\ensuremath{\mathcal{B}} }
\newcommand{\calC}{\ensuremath{\mathcal{C}} }
\newcommand{\calD}{\ensuremath{\mathcal{D}} }
\newcommand{\calE}{\ensuremath{\mathcal{E}} }

\newcommand{\calK}{\ensuremath{\mathcal{K}} }

\newcommand{\calM}{\ensuremath{\mathcal{M}} }

\newcommand{\calO}{\ensuremath{\mathcal{O}} }
\newcommand{\Q}{\ensuremath{\mathcal{Q}} }
\newcommand{\calQ}{\ensuremath{\mathcal{Q}} }

\newcommand{\V}{\ensuremath{\mathcal{V}}}
\newcommand{\calV}{\ensuremath{\mathcal{V}}}
\newcommand{\calX}{\ensuremath{\mathcal{X}} }

\newcommand{\cHV}{\ensuremath{{\mathcal{H}_{\V}}} }
\newcommand{\cHQ}{\ensuremath{{\mathcal{H}_{\Q}}} }
\newcommand{\cHX}{\ensuremath{{\mathcal{H}_{\calX}}} }
\newcommand{\cHQdual}{\ensuremath{{\mathcal{H}^*_{\scalebox{.5}{$\Q$}}}}}
\newcommand{\cHQsquare}{\ensuremath{{\mathcal{H}^2_{\Q}}} }

\newcommand{\f}{\ensuremath{f}}
\newcommand{\g}{\ensuremath{g}}
\newcommand{\h}{\ensuremath{h}}

\newcommand{\pbf}{\mathbf{p}}
\newcommand{\qbf}{\mathbf{q}}
\newcommand{\vbf}{\mathbf{v}}
\newcommand{\gbf}{\mathbf{g}}
\newcommand{\bb}{\ensuremath{\widetilde b}}
\newcommand{\cc}{\ensuremath{\widehat c}}
\newcommand{\ccc}{\ensuremath{\widetilde c}}
\newcommand{\dd}{\ensuremath{\widetilde d}}
\newcommand{\BB}{\ensuremath{\widetilde B}}
\newcommand{\CC}{\ensuremath{\widehat C}}
\newcommand{\CCC}{\ensuremath{\widetilde C}}
\newcommand{\DD}{\ensuremath{\widetilde D}}
\newcommand{\calBB}{\ensuremath{\widetilde \calB}}
\newcommand{\calCC}{\ensuremath{\widehat \calC}}
\newcommand{\calCCC}{\ensuremath{\widetilde \calC}}
\newcommand{\calDD}{\ensuremath{\widetilde \calD}}
\newcommand{\bbB}{\ensuremath{\mathbb{B}}}
\newcommand{\bbC}{\ensuremath{ \mathbb{C}}}
\newcommand{\bbD}{\ensuremath{\mathbb{D}}}
\newcommand{\bbM}{\ensuremath{\mathbb{M}}}

\DeclareFontFamily{U}{matha}{\hyphenchar\font45}
\DeclareFontShape{U}{matha}{m}{n}{
	<-6> matha5 <6-7> matha6 <7-8> matha7
	<8-9> matha8 <9-10> matha9
	<10-12> matha10 <12-> matha12
}{}
\DeclareSymbolFont{matha}{U}{matha}{m}{n}

\DeclareFontFamily{U}{mathx}{\hyphenchar\font45}
\DeclareFontShape{U}{mathx}{m}{n}{
	<-6> mathx5 <6-7> mathx6 <7-8> mathx7
	<8-9> mathx8 <9-10> mathx9
	<10-12> mathx10 <12-> mathx12
}{}
\DeclareSymbolFont{mathx}{U}{mathx}{m}{n}

\DeclareMathDelimiter{\vvvert} {0}{matha}{"7E}{mathx}{"17}%

%% file: pics/conv_halfDecoupled_h0125.tex
%
%
\begin{tikzpicture}

\begin{axis}[%
width=4.6in,
height=2.2in,
scale only axis,
unbounded coords=jump,
xmode=log,
xmin=0.0034,
xmax=0.14,
xminorticks=true,
xlabel={step size $\tau$},
ymode=log,
ymin=7e-4,
ymax=0.3,
yminorticks=true,
ylabel={relative error $e_T$},
axis background/.style={fill=white},
legend columns=2, 
legend style={at={(1.0,1.02)}, anchor=south east, legend cell align=left, align=left, draw=white!15!black}
]
\addplot[ultra thick, color=color1]
  table[row sep=crcr]{%
0.125 0.0810623479734721\\
0.0625 0.0371848023960902\\
0.03125 0.0166515080865012\\
0.015625 0.00685134271307111\\
0.0078125 0.00208756704450796\\
0.00390625 0.000826112568333951\\
};
\addlegendentry{implicit Euler}

\addplot[very thick, color=color2, dotted, mark=*, mark options={solid, color2}, mark size = 2.5]
table[row sep=crcr]{%
	0.125 0.263938639357525\\
	0.0625 0.138659294870395\\
	0.03125 0.0717272181706092\\
	0.015625 0.0368342816579371\\
	0.0078125 0.0190114232149586\\
	0.00390625 0.0100102692055827\\
};
\addlegendentry{semi-explicit scheme~\eqref{eq:halfDecoupled:semiExpl}}

\addplot[very thick, mark=square*, color=color3, mark size = 2.3]
table[row sep=crcr]{%
0.125 0.0859967431969647\\
0.0625 0.0525402895902035\\
0.03125 0.0500818003762832\\
0.015625 0.0603558264395541\\
0.0078125 0.0705358491131495\\
0.00390625 0.0754863209431561\\
};
\addlegendentry{\eqref{eq:halfDecoupled:norway}, $L_p=L_\theta=0.025$, $K=5$} 

\addplot[very thick, color=color4, mark=triangle*, mark size = 2.5]
  table[row sep=crcr]{%
	0.125 0.0863204753787782\\
	0.0625 0.0532185234633776\\
	0.03125 0.0514568888986619\\
	0.015625 0.0631164762319283\\
	0.0078125 0.0760637743501307\\
	0.00390625 0.0854614494243923\\
};
\addlegendentry{\eqref{eq:halfDecoupled:norway}, $L_p=L_\theta=0.001$, $K=5$}

%
%

\addplot[very thick, color=color3, mark=square*, dashed, mark options={solid}, mark size = 2.3]
  table[row sep=crcr]{%
0.125 0.0807124751283114\\
0.0625 0.0363185937668704\\
0.03125 0.0143160658246744\\
0.015625 0.00547593043436552\\
0.0078125 0.0106911247108005\\
0.00390625 0.0191114540769381\\
};
\addlegendentry{\eqref{eq:halfDecoupled:norway}, $L_p=L_\theta= 0.025$, $K=10$}

\addplot[very thick, color=color4, dashed, mark=triangle*, mark options={solid}, mark size = 2.5]
table[row sep=crcr]{%
	0.125 0.081035375990723\\
	0.0625 0.0369999055532716\\
	0.03125 0.0157195647925136\\
	0.015625 0.00436670276180446\\
	0.0078125 0.00582515714650982\\
	0.00390625 0.00816976841877012\\
};
\addlegendentry{\eqref{eq:halfDecoupled:norway}, $L_p=L_\theta=0.001$, $K=10$}

\addplot[very thick, color=color3, mark=square*, mark options={solid}, dotted, mark size = 2.3]
table[row sep=crcr]{%
0.125 0.0807259882930852\\
0.0625 0.0364751030899783\\
0.03125 0.0151909628127497\\
0.015625 0.00393104445611374\\
0.0078125 0.00587129807715794\\
0.00390625 0.0123609522408593\\
};
\addlegendentry{\eqref{eq:halfDecoupled:norway}, $L_p=L_\theta=0.025$, $K=20$\quad} 

\addplot[very thick, color=color4, dotted, mark=triangle*, mark options={solid}, mark size = 2.5]
table[row sep=crcr]{%
	0.125 0.0810488943998956\\
	0.0625 0.0371564019253317\\
	0.03125 0.0165925134586077\\
	0.015625 0.00672707475913439\\
	0.0078125 0.00182992985692426\\
	0.00390625 0.00128373197742472\\
};
\addlegendentry{\eqref{eq:halfDecoupled:norway}, $L_p=L_\theta=0.001$, $K=20$}

\addplot[thick, color=black, dashdotted, forget plot]
  table[row sep=crcr]{%
0.125 0.125\\
0.0625 0.0625\\
0.03125 0.03125\\
0.015625 0.015625\\
0.0078125 0.0078125\\
0.00390625 0.00390625\\
};

\end{axis}

\end{tikzpicture}%

%% file: pics/conv_fullyDecoupled_h0125.tex
%
%
\begin{tikzpicture}
	
	\begin{axis}[%
		width=4.6in,
		height=2.2in,
		scale only axis,
		unbounded coords=jump,
		xmode=log,
		xmin=0.0033,
		xmax=0.14,
		xminorticks=true,
		xlabel={step size $\tau$},
		ymode=log,
		ymin=5e-4,
		ymax=0.5,
		yminorticks=true,
		ylabel={relative error $e_T$},
		axis background/.style={fill=white},
		legend columns=2, 
		legend style={at={(1.0,1.05)}, anchor=south east, legend cell align=left, align=left, draw=white!15!black}
		]
		\addplot[ultra thick, color=color1]
		table[row sep=crcr]{%
			0.125 0.0810623479734721\\
			0.0625 0.0371848023960902\\
			0.03125 0.0166515080865012\\
			0.015625 0.00685134271307111\\
			0.0078125 0.00208756704450796\\
			0.00390625 0.000826112568333951\\
		};
		\addlegendentry{implicit Euler}
		
		\addplot[very thick, color=color2, dotted, mark=*, mark options={solid, color2}, mark size = 2.5]
		table[row sep=crcr]{%
			0.125 0.263938639357526\\
			0.0625 0.138659294870399\\
			0.03125 0.0717272181706004\\
			0.015625 0.036834281657943\\
			0.0078125 0.0190114232149427\\
			0.00390625 0.0100102692055386\\
		};
		\addlegendentry{semi-explicit scheme~\eqref{eq:fullyDecoupled:semiExpl}}

		\addplot[very thick, color=color4, mark=triangle*, mark options={solid}, mark size = 2.5]
		table[row sep=crcr]{%
				0.125 0.0834559626144228\\
				0.0625 0.0444539602291034\\
				0.03125 0.0324100980723395\\
				0.015625 0.032344380271942\\
				0.0078125 0.0356742263967174\\
				0.00390625 0.0393623688709015\\
			};
		
		\addlegendentry{\eqref{eq:fullyDecoupled:norway}, $L_p=L_\theta=0.025$, $K=5$}
		

		\addplot[very thick, color=color3, mark=square*, mark options={solid}, dotted, mark size = 2.3]
		table[row sep=crcr]{%
			0.125 0.0810623478816527\\
			0.0625 0.0371847866713893\\
			0.03125 0.0166509179891272\\
			0.015625 0.00684552099608857\\
			0.0078125 0.00206545492787944\\
			0.00390625 0.000868665410016283\\
		};
		\addlegendentry{\eqref{eq:fullyDecoupled:norway2}, $L_p=L_\theta=0.025$, $K=20$}

		\addplot[very thick, color=color4, dashed, mark=triangle*, mark options={solid}, mark size = 2.3]
		table[row sep=crcr]{%
			0.125 0.0577297297813001\\
			0.0625 0.0268524075346228\\
			0.03125 0.0118673022558561\\
			0.015625 0.00408460077100495\\
			0.0078125 0.00325308472596812\\
			0.00390625 0.00365708377199668\\
		};
		\addlegendentry{\eqref{eq:fullyDecoupled:norway}, $L_p=L_\theta=0.025$, $K=10$}
		
		\addplot[very thick, color=color5, mark=pentagon*, mark size = 2.5]
		table[row sep=crcr]{%
			0.125 0.292711138177936\\
			0.0625 0.156373327950703\\
			0.03125 0.0840361943456055\\
			0.015625 0.134376813705807\\
			0.0078125 59.2530302638505\\
			0.00390625 93576990.7859928\\
		};
		\addlegendentry{\eqref{eq:fullyDecoupled:russia}, $\sigma =0.77$}

		\addplot[very thick, color=color4, dotted, mark=triangle*, mark options={solid}, mark size = 2.5]
		table[row sep=crcr]{%
			0.125 0.0577358492848598\\
			0.0625 0.0269228726492995\\
			0.03125 0.0122587004110119\\
			0.015625 0.00517728242075251\\
			0.0078125 0.00170484693682207\\
			0.00390625 0.000578377976332179\\
		};
		\addlegendentry{\eqref{eq:fullyDecoupled:norway}, $L_p=L_\theta=0.025$, $K=20$}



		\addplot[thick, color=black, dashdotted, forget plot]
		table[row sep=crcr]{%
			0.125 0.125\\
			0.0625 0.0625\\
			0.03125 0.03125\\
			0.015625 0.015625\\
			0.0078125 0.0078125\\
			0.00390625 0.00390625\\
		};
		
	\end{axis}
	
\end{tikzpicture}%

%% file: pics/conv_Exp2_a_R.tex
%
%
%
\begin{tikzpicture}
	
	\begin{axis}[%
		width=2.3in,
		height=1.7in,
		scale only axis,
		xmode=log,
		xmin=0.003,
		xmax=0.3,
		xminorticks=true,
		xlabel={time step size $\tau$},
		ymode=log,
		ymin=0.017,
		ymax=0.54,
		yminorticks=true,
		ylabel={relative error $e_T$},
		axis background/.style={fill=white},
		legend columns = 5, 
		legend style={at={(2.17,1.03)}, anchor=south east, legend cell align=left, align=left}
		]
		
		\addplot[very thick, color=color1]
		table[row sep=crcr]{%
			0.25 0.24244323529415\\
			0.125 0.152130623708931\\
			0.0625 0.111353726411227\\
			0.03125 0.0926279595336345\\
			0.015625 0.0838968049938586\\
			0.0078125 0.0799987080970739\\
			0.00390625 0.0785985146352103\\
			};
		\addlegendentry{implicit \quad} 

		\addplot[very thick, color=color1, dashed, forget plot]
		table[row sep=crcr]{%
			0.25 0.20587123809399\\
			0.125 0.107292369555625\\
			0.0625 0.0635476195348524\\
			0.03125 0.0443457840820217\\
			0.015625 0.0383611333766378\\
			0.0078125 0.0424343817050633\\
			0.00390625 0.0446676440850379\\
			};

		\addplot[very thick, color=color1, dotted, forget plot]
		table[row sep=crcr]{%
			0.25 0.194725290744904\\
			0.125 0.090125323629555\\
			0.0625 0.0419163696150752\\
			0.03125 0.021652351602816\\
			0.015625 0.0203513504506943\\
			0.0078125 0.02510020448623\\
			0.00390625 0.0275423266220289\\
			};

		\addplot[very thick, color=color2, mark=*, mark options={solid, color2}, mark size = 2.7]
		table[row sep=crcr]{%
			0.25 0.482011941474498\\
			0.125 0.261052752938183\\
			0.0625 0.150972958474527\\
			0.03125 0.102957430858102\\
			0.015625 0.0854271846209207\\
			0.0078125 0.0797788062579869\\
			0.00390625 0.0782646299635763\\
			}; 
		\addlegendentry{semi-explicit scheme~\eqref{eq:fullyDecoupled:semiExpl} \quad} 

		\addplot[very thick, color=color2, mark=*, mark options={solid, color2}, mark size = 2.7, dashed, forget plot]
		table[row sep=crcr]{%
			0.25 0.474949498969165\\
			0.125 0.24856938157751\\
			0.0625 0.132206486496296\\
			0.03125 0.0773797370128077\\
			0.015625 0.0532535113377423\\
			0.0078125 0.0482248737159523\\
			0.00390625 0.0470424775820515\\
			};
		
		\addplot[very thick, color=color2, mark=*, mark options={solid, color2}, mark size = 2.7, dotted, forget plot]
		table[row sep=crcr]{%
			0.25 0.473201459369606\\
			0.125 0.245563645956832\\
			0.0625 0.12872823407453\\
			0.03125 0.0709061691356375\\
			0.015625 0.0440166085773051\\
			0.0078125 0.0354888475595594\\
			0.00390625 0.0320922119543256\\
			}; 

		\addplot[very thick, color=color4, mark=triangle*, mark options={solid}, mark size = 2.5]
		table[row sep=crcr]{%
			0.25 0.482011941474496\\
			0.125 0.261052752938184\\
			0.0625 0.150972958474491\\
			0.03125 0.102957430858087\\
			0.015625 0.0854271846209318\\
			0.0078125 0.0797788062580502\\
			0.00390625 0.0782646299635726\\
		}; 
		\addlegendentry{semi-explicit scheme~\eqref{eq:halfDecoupled:semiExpl} } 
		
		\addplot[very thick, color=color4, mark=triangle*, mark options={solid}, mark size = 2.5, dashed, forget plot]
		table[row sep=crcr]{%
			0.25 0.474949498969184\\
			0.125 0.248569381577523\\
			0.0625 0.132206486496288\\
			0.03125 0.07737973701292\\
			0.015625 0.0532535113378238\\
			0.0078125 0.0482248737160453\\
			0.00390625 0.0470424775820059\\
			};
		
		\addplot[very thick, color=color4, mark=triangle*, mark options={solid}, mark size = 2.5, dotted, forget plot]
		table[row sep=crcr]{%
			0.25 0.473201459369596\\
			0.125 0.245563645956842\\
			0.0625 0.128728234074534\\
			0.03125 0.070906169136722\\
			0.015625 0.0440166085771921\\
			0.0078125 0.0354888475599813\\
			0.00390625 0.0320922119551743\\
			};

	\end{axis}
	

%% file: pics/conv_Exp2_b_R.tex
%
%
%
	
	\begin{axis}[%
		width=2.3in,
		height=1.7in,
		at={(2.7in,0.0in)},
		scale only axis,
		xmode=log,
		xmin=0.00015,
		xmax=0.3,
		xminorticks=true,
		xlabel={time step size $\tau$},
		ymode=log,
		ymin=0.0015,
		ymax=0.59,
		yminorticks=true,
		axis background/.style={fill=white},
		legend columns = 5, 
		legend style={at={(2.0,1.03)}, anchor=south east, legend cell align=left, align=left}
		]
		
		\addplot[very thick, color=color1, dashdotted]
		table[row sep=crcr]{%
			0.25 0.264351021100559\\
			0.125 0.163355868432891\\
			0.0625 0.114445280598765\\
			0.03125 0.0910772582001014\\
			0.015625 0.0798218961166923\\
			0.0078125 0.0743311412328873\\
			0.00390625 0.0716248977935887\\
			0.001953125 0.0702822844833653\\
			0.0009765625 0.0696137045623316\\
			0.00048828125 0.0726070803632282\\
			0.000244140625 0.0724394564367411\\
			};

		\addplot[very thick, color=color1]
		table[row sep=crcr]{%
			0.25 0.219391512193428\\
			0.125 0.114023385652173\\
			0.0625 0.0622297352973607\\
			0.03125 0.037234260289203\\
			0.015625 0.0252433106744452\\
			0.0078125 0.0195016703527733\\
			0.00390625 0.0167392176173132\\
			0.001953125 0.0153968152823162\\
			0.0009765625 0.0147377089935785\\
			0.00048828125 0.0176018721439784\\
			0.000244140625 0.0174311103600445\\
			};
		
		\addplot[very thick, color=color1, dashed]
		table[row sep=crcr]{%
			0.25 0.208207794210437\\
			0.125 0.10178947570732\\
			0.0625 0.0492117518325607\\
			0.03125 0.0235912658491455\\
			0.015625 0.0111617272338988\\
			0.0078125 0.00530690083989715\\
			0.00390625 0.00283016051066354\\
			0.001953125 0.00223915372142952\\
			0.0009765625 0.00246827364194632\\
			0.00048828125 0.00345121641756022\\
			0.000244140625 0.00329966937666113\\
			};

		\addplot[very thick, color=color2, mark=*, mark options={solid, color2}, mark size = 2.7, dashdotted]
		table[row sep=crcr]{%
			0.25 0.486808056638974\\
			0.125 0.263226713500305\\
			0.0625 0.145143933988891\\
			0.03125 0.0855570654287995\\
			0.015625 0.0619322821526808\\
			0.0078125 0.0615630421446912\\
			0.00390625 0.0647235017318998\\
			0.001953125 0.0667398700258338\\
			0.0009765625 0.0678229291748887\\
			0.00048828125 0.0716954295901103\\
			0.000244140625 0.071982646218066\\
			};

		\addplot[very thick, color=color2, mark=*, mark options={solid, color2}, mark size = 2.7]
		table[row sep=crcr]{%
			0.25 0.477736203489559\\
			0.125 0.251088725309882\\
			0.0625 0.130563431316275\\
			0.03125 0.0679167001776337\\
			0.015625 0.0362069410739145\\
			0.0078125 0.0207675780952759\\
			0.00390625 0.0143974877004487\\
			0.001953125 0.0131748756872731\\
			0.0009765625 0.0134175514573525\\
			0.00048828125 0.0167994454109269\\
			0.000244140625 0.0170237610929885\\
			};
				
		\addplot[very thick, color=color2, mark=*, mark options={solid, color2}, mark size = 2.7, dashed]
		table[row sep=crcr]{%
			0.25 0.475611182960448\\
			0.125 0.248392246234055\\
			0.0625 0.1275982638252\\
			0.03125 0.0646918154708138\\
			0.015625 0.0326454980204429\\
			0.0078125 0.0166698534198754\\
			0.00390625 0.00882391534505888\\
			0.001953125 0.00514983887501691\\
			0.0009765625 0.00374070725885023\\
			0.00048828125 0.00313162192007798\\
			0.000244140625 0.00308723236861582\\
			};

		\addplot[very thick, color=color4, mark=triangle*, mark options={solid}, mark size = 2.5, dashdotted]
		table[row sep=crcr]{%
			0.25 0.486808056638967\\
			0.125 0.263226713500303\\
			0.0625 0.145143933988893\\
			0.03125 0.0855570654288114\\
			0.015625 0.0619322821526746\\
			0.0078125 0.0615630421446741\\
			0.00390625 0.0647235017319093\\
			0.001953125 0.0667398700258369\\
			0.0009765625 0.0678229291747881\\
			0.00048828125 0.0716954295899568\\
			0.000244140625 0.0719826462183797\\
			};

		\addplot[very thick, color=color4, mark=triangle*, mark options={solid}, mark size = 2.5]
		table[row sep=crcr]{%
			0.25 0.477736203489558\\
			0.125 0.251088725309883\\
			0.0625 0.130563431316245\\
			0.03125 0.0679167001776271\\
			0.015625 0.0362069410739268\\
			0.0078125 0.0207675780953277\\
			0.00390625 0.0143974877004177\\
			0.001953125 0.0131748756872196\\
			0.0009765625 0.0134175514573976\\
			0.00048828125 0.0167994454110626\\
			0.000244140625 0.0170237610926539\\
			};
		
		\addplot[very thick, color=color4, mark=triangle*, mark options={solid}, mark size = 2.5, dashed]
		table[row sep=crcr]{%
			0.25 0.475611182960451\\
			0.125 0.248392246234063\\
			0.0625 0.127598263825201\\
			0.03125 0.0646918154708428\\
			0.015625 0.0326454980203592\\
			0.0078125 0.0166698534199608\\
			0.00390625 0.00882391534512231\\
			0.001953125 0.00514983887452808\\
			0.0009765625 0.00374070725882531\\
			0.00048828125 0.00313162191991918\\
			0.000244140625 0.0030872323687579\\
			};
		
	\end{axis}
	
\end{tikzpicture}%

%% file: pics/pictureExperimentConvergency3col.tex
%
%
\definecolor{mycolor1}{rgb}{0.12941,0.12941,0.12941}%
\begin{tikzpicture}

\begin{axis}[%
width=4.0in,
height=2.0in,
scale only axis,
xmin=1,
xmax=51,
xlabel style={font=\color{mycolor1}},
xlabel={$\ccc_0$},
xtick={5, 15, 25, 35, 45},
xticklabels={$1$, $2$, $3$, $4$, $5$},   
ymin=1,
ymax=51,
ytick={10, 20, 30, 40},
yticklabels={0.1, 0.2, 0.3, 0.4},  
ylabel style={font=\color{mycolor1}},
ylabel={$\alpha=\beta$},
axis background/.style={fill=white},
colormap={custom}{color(0)=(color2) color(10)=(yellow) color(100)=(red)},
]

\addplot[%
surf,
shader=flat corner, draw=mycolor1, 
mesh/rows=51]
table[row sep=crcr, point meta=\thisrow{c}] {%
x	y	c\\
1	1	0\\
1	2	10\\
1	3	10\\
1	4	10\\
1	5	10\\
1	6	10\\
1	7	10\\
1	8	10\\
1	9	10\\
1	10	10\\
1	11	10\\
1	12	10\\
1	13	10\\
1	14	10\\
1	15	10\\
1	16	10\\
1	17	10\\
1	18	100\\
1	19	100\\
1	20	100\\
1	21	100\\
1	22	100\\
1	23	100\\
1	24	100\\
1	25	100\\
1	26	100\\
1	27	100\\
1	28	100\\
1	29	100\\
1	30	100\\
1	31	100\\
1	32	100\\
1	33	100\\
1	34	100\\
1	35	100\\
1	36	100\\
1	37	100\\
1	38	100\\
1	39	100\\
1	40	100\\
1	41	100\\
1	42	100\\
1	43	100\\
1	44	100\\
1	45	100\\
1	46	100\\
1	47	100\\
1	48	100\\
1	49	100\\
1	50	100\\
1	51	100\\
2	1	0\\
2	2	0\\
2	3	0\\
2	4	0\\
2	5	0\\
2	6	0\\
2	7	0\\
2	8	0\\
2	9	10\\
2	10	10\\
2	11	10\\
2	12	10\\
2	13	10\\
2	14	10\\
2	15	10\\
2	16	10\\
2	17	10\\
2	18	10\\
2	19	10\\
2	20	100\\
2	21	100\\
2	22	100\\
2	23	100\\
2	24	100\\
2	25	100\\
2	26	100\\
2	27	100\\
2	28	100\\
2	29	100\\
2	30	100\\
2	31	100\\
2	32	100\\
2	33	100\\
2	34	100\\
2	35	100\\
2	36	100\\
2	37	100\\
2	38	100\\
2	39	100\\
2	40	100\\
2	41	100\\
2	42	100\\
2	43	100\\
2	44	100\\
2	45	100\\
2	46	100\\
2	47	100\\
2	48	100\\
2	49	100\\
2	50	100\\
2	51	100\\
3	1	0\\
3	2	0\\
3	3	0\\
3	4	0\\
3	5	0\\
3	6	0\\
3	7	0\\
3	8	0\\
3	9	0\\
3	10	0\\
3	11	0\\
3	12	10\\
3	13	10\\
3	14	10\\
3	15	10\\
3	16	10\\
3	17	10\\
3	18	10\\
3	19	10\\
3	20	10\\
3	21	10\\
3	22	100\\
3	23	100\\
3	24	100\\
3	25	100\\
3	26	100\\
3	27	100\\
3	28	100\\
3	29	100\\
3	30	100\\
3	31	100\\
3	32	100\\
3	33	100\\
3	34	100\\
3	35	100\\
3	36	100\\
3	37	100\\
3	38	100\\
3	39	100\\
3	40	100\\
3	41	100\\
3	42	100\\
3	43	100\\
3	44	100\\
3	45	100\\
3	46	100\\
3	47	100\\
3	48	100\\
3	49	100\\
3	50	100\\
3	51	100\\
4	1	0\\
4	2	0\\
4	3	0\\
4	4	0\\
4	5	0\\
4	6	0\\
4	7	0\\
4	8	0\\
4	9	0\\
4	10	0\\
4	11	0\\
4	12	0\\
4	13	0\\
4	14	10\\
4	15	10\\
4	16	10\\
4	17	10\\
4	18	10\\
4	19	10\\
4	20	10\\
4	21	10\\
4	22	10\\
4	23	100\\
4	24	100\\
4	25	100\\
4	26	100\\
4	27	100\\
4	28	100\\
4	29	100\\
4	30	100\\
4	31	100\\
4	32	100\\
4	33	100\\
4	34	100\\
4	35	100\\
4	36	100\\
4	37	100\\
4	38	100\\
4	39	100\\
4	40	100\\
4	41	100\\
4	42	100\\
4	43	100\\
4	44	100\\
4	45	100\\
4	46	100\\
4	47	100\\
4	48	100\\
4	49	100\\
4	50	100\\
4	51	100\\
5	1	0\\
5	2	0\\
5	3	0\\
5	4	0\\
5	5	0\\
5	6	0\\
5	7	0\\
5	8	0\\
5	9	0\\
5	10	0\\
5	11	0\\
5	12	0\\
5	13	0\\
5	14	0\\
5	15	0\\
5	16	10\\
5	17	10\\
5	18	10\\
5	19	10\\
5	20	10\\
5	21	10\\
5	22	10\\
5	23	10\\
5	24	100\\
5	25	100\\
5	26	100\\
5	27	100\\
5	28	100\\
5	29	100\\
5	30	100\\
5	31	100\\
5	32	100\\
5	33	100\\
5	34	100\\
5	35	100\\
5	36	100\\
5	37	100\\
5	38	100\\
5	39	100\\
5	40	100\\
5	41	100\\
5	42	100\\
5	43	100\\
5	44	100\\
5	45	100\\
5	46	100\\
5	47	100\\
5	48	100\\
5	49	100\\
5	50	100\\
5	51	100\\
6	1	0\\
6	2	0\\
6	3	0\\
6	4	0\\
6	5	0\\
6	6	0\\
6	7	0\\
6	8	0\\
6	9	0\\
6	10	0\\
6	11	0\\
6	12	0\\
6	13	0\\
6	14	0\\
6	15	0\\
6	16	0\\
6	17	0\\
6	18	10\\
6	19	10\\
6	20	10\\
6	21	10\\
6	22	10\\
6	23	10\\
6	24	10\\
6	25	10\\
6	26	100\\
6	27	100\\
6	28	100\\
6	29	100\\
6	30	100\\
6	31	100\\
6	32	100\\
6	33	100\\
6	34	100\\
6	35	100\\
6	36	100\\
6	37	100\\
6	38	100\\
6	39	100\\
6	40	100\\
6	41	100\\
6	42	100\\
6	43	100\\
6	44	100\\
6	45	100\\
6	46	100\\
6	47	100\\
6	48	100\\
6	49	100\\
6	50	100\\
6	51	100\\
7	1	0\\
7	2	0\\
7	3	0\\
7	4	0\\
7	5	0\\
7	6	0\\
7	7	0\\
7	8	0\\
7	9	0\\
7	10	0\\
7	11	0\\
7	12	0\\
7	13	0\\
7	14	0\\
7	15	0\\
7	16	0\\
7	17	0\\
7	18	0\\
7	19	0\\
7	20	10\\
7	21	10\\
7	22	10\\
7	23	10\\
7	24	10\\
7	25	10\\
7	26	10\\
7	27	100\\
7	28	100\\
7	29	100\\
7	30	100\\
7	31	100\\
7	32	100\\
7	33	100\\
7	34	100\\
7	35	100\\
7	36	100\\
7	37	100\\
7	38	100\\
7	39	100\\
7	40	100\\
7	41	100\\
7	42	100\\
7	43	100\\
7	44	100\\
7	45	100\\
7	46	100\\
7	47	100\\
7	48	100\\
7	49	100\\
7	50	100\\
7	51	100\\
8	1	0\\
8	2	0\\
8	3	0\\
8	4	0\\
8	5	0\\
8	6	0\\
8	7	0\\
8	8	0\\
8	9	0\\
8	10	0\\
8	11	0\\
8	12	0\\
8	13	0\\
8	14	0\\
8	15	0\\
8	16	0\\
8	17	0\\
8	18	0\\
8	19	0\\
8	20	0\\
8	21	10\\
8	22	10\\
8	23	10\\
8	24	10\\
8	25	10\\
8	26	10\\
8	27	100\\
8	28	100\\
8	29	100\\
8	30	100\\
8	31	100\\
8	32	100\\
8	33	100\\
8	34	100\\
8	35	100\\
8	36	100\\
8	37	100\\
8	38	100\\
8	39	100\\
8	40	100\\
8	41	100\\
8	42	100\\
8	43	100\\
8	44	100\\
8	45	100\\
8	46	100\\
8	47	100\\
8	48	100\\
8	49	100\\
8	50	100\\
8	51	100\\
9	1	0\\
9	2	0\\
9	3	0\\
9	4	0\\
9	5	0\\
9	6	0\\
9	7	0\\
9	8	0\\
9	9	0\\
9	10	0\\
9	11	0\\
9	12	0\\
9	13	0\\
9	14	0\\
9	15	0\\
9	16	0\\
9	17	0\\
9	18	0\\
9	19	0\\
9	20	0\\
9	21	0\\
9	22	0\\
9	23	10\\
9	24	10\\
9	25	10\\
9	26	10\\
9	27	10\\
9	28	100\\
9	29	100\\
9	30	100\\
9	31	100\\
9	32	100\\
9	33	100\\
9	34	100\\
9	35	100\\
9	36	100\\
9	37	100\\
9	38	100\\
9	39	100\\
9	40	100\\
9	41	100\\
9	42	100\\
9	43	100\\
9	44	100\\
9	45	100\\
9	46	100\\
9	47	100\\
9	48	100\\
9	49	100\\
9	50	100\\
9	51	100\\
10	1	0\\
10	2	0\\
10	3	0\\
10	4	0\\
10	5	0\\
10	6	0\\
10	7	0\\
10	8	0\\
10	9	0\\
10	10	0\\
10	11	0\\
10	12	0\\
10	13	0\\
10	14	0\\
10	15	0\\
10	16	0\\
10	17	0\\
10	18	0\\
10	19	0\\
10	20	0\\
10	21	0\\
10	22	0\\
10	23	0\\
10	24	10\\
10	25	10\\
10	26	10\\
10	27	10\\
10	28	10\\
10	29	100\\
10	30	100\\
10	31	100\\
10	32	100\\
10	33	100\\
10	34	100\\
10	35	100\\
10	36	100\\
10	37	100\\
10	38	100\\
10	39	100\\
10	40	100\\
10	41	100\\
10	42	100\\
10	43	100\\
10	44	100\\
10	45	100\\
10	46	100\\
10	47	100\\
10	48	100\\
10	49	100\\
10	50	100\\
10	51	100\\
11	1	0\\
11	2	0\\
11	3	0\\
11	4	0\\
11	5	0\\
11	6	0\\
11	7	0\\
11	8	0\\
11	9	0\\
11	10	0\\
11	11	0\\
11	12	0\\
11	13	0\\
11	14	0\\
11	15	0\\
11	16	0\\
11	17	0\\
11	18	0\\
11	19	0\\
11	20	0\\
11	21	0\\
11	22	0\\
11	23	0\\
11	24	0\\
11	25	10\\
11	26	10\\
11	27	10\\
11	28	10\\
11	29	10\\
11	30	100\\
11	31	100\\
11	32	100\\
11	33	100\\
11	34	100\\
11	35	100\\
11	36	100\\
11	37	100\\
11	38	100\\
11	39	100\\
11	40	100\\
11	41	100\\
11	42	100\\
11	43	100\\
11	44	100\\
11	45	100\\
11	46	100\\
11	47	100\\
11	48	100\\
11	49	100\\
11	50	100\\
11	51	100\\
12	1	0\\
12	2	0\\
12	3	0\\
12	4	0\\
12	5	0\\
12	6	0\\
12	7	0\\
12	8	0\\
12	9	0\\
12	10	0\\
12	11	0\\
12	12	0\\
12	13	0\\
12	14	0\\
12	15	0\\
12	16	0\\
12	17	0\\
12	18	0\\
12	19	0\\
12	20	0\\
12	21	0\\
12	22	0\\
12	23	0\\
12	24	0\\
12	25	0\\
12	26	10\\
12	27	10\\
12	28	10\\
12	29	10\\
12	30	10\\
12	31	100\\
12	32	100\\
12	33	100\\
12	34	100\\
12	35	100\\
12	36	100\\
12	37	100\\
12	38	100\\
12	39	100\\
12	40	100\\
12	41	100\\
12	42	100\\
12	43	100\\
12	44	100\\
12	45	100\\
12	46	100\\
12	47	100\\
12	48	100\\
12	49	100\\
12	50	100\\
12	51	100\\
13	1	0\\
13	2	0\\
13	3	0\\
13	4	0\\
13	5	0\\
13	6	0\\
13	7	0\\
13	8	0\\
13	9	0\\
13	10	0\\
13	11	0\\
13	12	0\\
13	13	0\\
13	14	0\\
13	15	0\\
13	16	0\\
13	17	0\\
13	18	0\\
13	19	0\\
13	20	0\\
13	21	0\\
13	22	0\\
13	23	0\\
13	24	0\\
13	25	0\\
13	26	0\\
13	27	10\\
13	28	10\\
13	29	10\\
13	30	10\\
13	31	100\\
13	32	100\\
13	33	100\\
13	34	100\\
13	35	100\\
13	36	100\\
13	37	100\\
13	38	100\\
13	39	100\\
13	40	100\\
13	41	100\\
13	42	100\\
13	43	100\\
13	44	100\\
13	45	100\\
13	46	100\\
13	47	100\\
13	48	100\\
13	49	100\\
13	50	100\\
13	51	100\\
14	1	0\\
14	2	0\\
14	3	0\\
14	4	0\\
14	5	0\\
14	6	0\\
14	7	0\\
14	8	0\\
14	9	0\\
14	10	0\\
14	11	0\\
14	12	0\\
14	13	0\\
14	14	0\\
14	15	0\\
14	16	0\\
14	17	0\\
14	18	0\\
14	19	0\\
14	20	0\\
14	21	0\\
14	22	0\\
14	23	0\\
14	24	0\\
14	25	0\\
14	26	0\\
14	27	0\\
14	28	10\\
14	29	10\\
14	30	10\\
14	31	10\\
14	32	100\\
14	33	100\\
14	34	100\\
14	35	100\\
14	36	100\\
14	37	100\\
14	38	100\\
14	39	100\\
14	40	100\\
14	41	100\\
14	42	100\\
14	43	100\\
14	44	100\\
14	45	100\\
14	46	100\\
14	47	100\\
14	48	100\\
14	49	100\\
14	50	100\\
14	51	100\\
15	1	0\\
15	2	0\\
15	3	0\\
15	4	0\\
15	5	0\\
15	6	0\\
15	7	0\\
15	8	0\\
15	9	0\\
15	10	0\\
15	11	0\\
15	12	0\\
15	13	0\\
15	14	0\\
15	15	0\\
15	16	0\\
15	17	0\\
15	18	0\\
15	19	0\\
15	20	0\\
15	21	0\\
15	22	0\\
15	23	0\\
15	24	0\\
15	25	0\\
15	26	0\\
15	27	0\\
15	28	0\\
15	29	10\\
15	30	10\\
15	31	10\\
15	32	100\\
15	33	100\\
15	34	100\\
15	35	100\\
15	36	100\\
15	37	100\\
15	38	100\\
15	39	100\\
15	40	100\\
15	41	100\\
15	42	100\\
15	43	100\\
15	44	100\\
15	45	100\\
15	46	100\\
15	47	100\\
15	48	100\\
15	49	100\\
15	50	100\\
15	51	100\\
16	1	0\\
16	2	0\\
16	3	0\\
16	4	0\\
16	5	0\\
16	6	0\\
16	7	0\\
16	8	0\\
16	9	0\\
16	10	0\\
16	11	0\\
16	12	0\\
16	13	0\\
16	14	0\\
16	15	0\\
16	16	0\\
16	17	0\\
16	18	0\\
16	19	0\\
16	20	0\\
16	21	0\\
16	22	0\\
16	23	0\\
16	24	0\\
16	25	0\\
16	26	0\\
16	27	0\\
16	28	0\\
16	29	0\\
16	30	10\\
16	31	10\\
16	32	10\\
16	33	100\\
16	34	100\\
16	35	100\\
16	36	100\\
16	37	100\\
16	38	100\\
16	39	100\\
16	40	100\\
16	41	100\\
16	42	100\\
16	43	100\\
16	44	100\\
16	45	100\\
16	46	100\\
16	47	100\\
16	48	100\\
16	49	100\\
16	50	100\\
16	51	100\\
17	1	0\\
17	2	0\\
17	3	0\\
17	4	0\\
17	5	0\\
17	6	0\\
17	7	0\\
17	8	0\\
17	9	0\\
17	10	0\\
17	11	0\\
17	12	0\\
17	13	0\\
17	14	0\\
17	15	0\\
17	16	0\\
17	17	0\\
17	18	0\\
17	19	0\\
17	20	0\\
17	21	0\\
17	22	0\\
17	23	0\\
17	24	0\\
17	25	0\\
17	26	0\\
17	27	0\\
17	28	0\\
17	29	0\\
17	30	10\\
17	31	10\\
17	32	10\\
17	33	100\\
17	34	100\\
17	35	100\\
17	36	100\\
17	37	100\\
17	38	100\\
17	39	100\\
17	40	100\\
17	41	100\\
17	42	100\\
17	43	100\\
17	44	100\\
17	45	100\\
17	46	100\\
17	47	100\\
17	48	100\\
17	49	100\\
17	50	100\\
17	51	100\\
18	1	0\\
18	2	0\\
18	3	0\\
18	4	0\\
18	5	0\\
18	6	0\\
18	7	0\\
18	8	0\\
18	9	0\\
18	10	0\\
18	11	0\\
18	12	0\\
18	13	0\\
18	14	0\\
18	15	0\\
18	16	0\\
18	17	0\\
18	18	0\\
18	19	0\\
18	20	0\\
18	21	0\\
18	22	0\\
18	23	0\\
18	24	0\\
18	25	0\\
18	26	0\\
18	27	0\\
18	28	0\\
18	29	0\\
18	30	10\\
18	31	10\\
18	32	10\\
18	33	10\\
18	34	100\\
18	35	100\\
18	36	100\\
18	37	100\\
18	38	100\\
18	39	100\\
18	40	100\\
18	41	100\\
18	42	100\\
18	43	100\\
18	44	100\\
18	45	100\\
18	46	100\\
18	47	100\\
18	48	100\\
18	49	100\\
18	50	100\\
18	51	100\\
19	1	0\\
19	2	0\\
19	3	0\\
19	4	0\\
19	5	0\\
19	6	0\\
19	7	0\\
19	8	0\\
19	9	0\\
19	10	0\\
19	11	0\\
19	12	0\\
19	13	0\\
19	14	0\\
19	15	0\\
19	16	0\\
19	17	0\\
19	18	0\\
19	19	0\\
19	20	0\\
19	21	0\\
19	22	0\\
19	23	0\\
19	24	0\\
19	25	0\\
19	26	0\\
19	27	0\\
19	28	0\\
19	29	0\\
19	30	10\\
19	31	10\\
19	32	10\\
19	33	10\\
19	34	100\\
19	35	100\\
19	36	100\\
19	37	100\\
19	38	100\\
19	39	100\\
19	40	100\\
19	41	100\\
19	42	100\\
19	43	100\\
19	44	100\\
19	45	100\\
19	46	100\\
19	47	100\\
19	48	100\\
19	49	100\\
19	50	100\\
19	51	100\\
20	1	0\\
20	2	0\\
20	3	0\\
20	4	0\\
20	5	0\\
20	6	0\\
20	7	0\\
20	8	0\\
20	9	0\\
20	10	0\\
20	11	0\\
20	12	0\\
20	13	0\\
20	14	0\\
20	15	0\\
20	16	0\\
20	17	0\\
20	18	0\\
20	19	0\\
20	20	0\\
20	21	0\\
20	22	0\\
20	23	0\\
20	24	0\\
20	25	0\\
20	26	0\\
20	27	0\\
20	28	0\\
20	29	0\\
20	30	10\\
20	31	10\\
20	32	10\\
20	33	10\\
20	34	10\\
20	35	100\\
20	36	100\\
20	37	100\\
20	38	100\\
20	39	100\\
20	40	100\\
20	41	100\\
20	42	100\\
20	43	100\\
20	44	100\\
20	45	100\\
20	46	100\\
20	47	100\\
20	48	100\\
20	49	100\\
20	50	100\\
20	51	100\\
21	1	0\\
21	2	0\\
21	3	0\\
21	4	0\\
21	5	0\\
21	6	0\\
21	7	0\\
21	8	0\\
21	9	0\\
21	10	0\\
21	11	0\\
21	12	0\\
21	13	0\\
21	14	0\\
21	15	0\\
21	16	0\\
21	17	0\\
21	18	0\\
21	19	0\\
21	20	0\\
21	21	0\\
21	22	0\\
21	23	0\\
21	24	0\\
21	25	0\\
21	26	0\\
21	27	0\\
21	28	0\\
21	29	0\\
21	30	10\\
21	31	10\\
21	32	10\\
21	33	10\\
21	34	10\\
21	35	100\\
21	36	100\\
21	37	100\\
21	38	100\\
21	39	100\\
21	40	100\\
21	41	100\\
21	42	100\\
21	43	100\\
21	44	100\\
21	45	100\\
21	46	100\\
21	47	100\\
21	48	100\\
21	49	100\\
21	50	100\\
21	51	100\\
22	1	0\\
22	2	0\\
22	3	0\\
22	4	0\\
22	5	0\\
22	6	0\\
22	7	0\\
22	8	0\\
22	9	0\\
22	10	0\\
22	11	0\\
22	12	0\\
22	13	0\\
22	14	0\\
22	15	0\\
22	16	0\\
22	17	0\\
22	18	0\\
22	19	0\\
22	20	0\\
22	21	0\\
22	22	0\\
22	23	0\\
22	24	0\\
22	25	0\\
22	26	0\\
22	27	0\\
22	28	0\\
22	29	0\\
22	30	10\\
22	31	10\\
22	32	10\\
22	33	10\\
22	34	10\\
22	35	10\\
22	36	100\\
22	37	100\\
22	38	100\\
22	39	100\\
22	40	100\\
22	41	100\\
22	42	100\\
22	43	100\\
22	44	100\\
22	45	100\\
22	46	100\\
22	47	100\\
22	48	100\\
22	49	100\\
22	50	100\\
22	51	100\\
23	1	0\\
23	2	0\\
23	3	0\\
23	4	0\\
23	5	0\\
23	6	0\\
23	7	0\\
23	8	0\\
23	9	0\\
23	10	0\\
23	11	0\\
23	12	0\\
23	13	0\\
23	14	0\\
23	15	0\\
23	16	0\\
23	17	0\\
23	18	0\\
23	19	0\\
23	20	0\\
23	21	0\\
23	22	0\\
23	23	0\\
23	24	0\\
23	25	0\\
23	26	0\\
23	27	0\\
23	28	0\\
23	29	0\\
23	30	10\\
23	31	10\\
23	32	10\\
23	33	10\\
23	34	10\\
23	35	10\\
23	36	100\\
23	37	100\\
23	38	100\\
23	39	100\\
23	40	100\\
23	41	100\\
23	42	100\\
23	43	100\\
23	44	100\\
23	45	100\\
23	46	100\\
23	47	100\\
23	48	100\\
23	49	100\\
23	50	100\\
23	51	100\\
24	1	0\\
24	2	0\\
24	3	0\\
24	4	0\\
24	5	0\\
24	6	0\\
24	7	0\\
24	8	0\\
24	9	0\\
24	10	0\\
24	11	0\\
24	12	0\\
24	13	0\\
24	14	0\\
24	15	0\\
24	16	0\\
24	17	0\\
24	18	0\\
24	19	0\\
24	20	0\\
24	21	0\\
24	22	0\\
24	23	0\\
24	24	0\\
24	25	0\\
24	26	0\\
24	27	0\\
24	28	0\\
24	29	0\\
24	30	10\\
24	31	10\\
24	32	10\\
24	33	10\\
24	34	10\\
24	35	10\\
24	36	100\\
24	37	100\\
24	38	100\\
24	39	100\\
24	40	100\\
24	41	100\\
24	42	100\\
24	43	100\\
24	44	100\\
24	45	100\\
24	46	100\\
24	47	100\\
24	48	100\\
24	49	100\\
24	50	100\\
24	51	100\\
25	1	0\\
25	2	0\\
25	3	0\\
25	4	0\\
25	5	0\\
25	6	0\\
25	7	0\\
25	8	0\\
25	9	0\\
25	10	0\\
25	11	0\\
25	12	0\\
25	13	0\\
25	14	0\\
25	15	0\\
25	16	0\\
25	17	0\\
25	18	0\\
25	19	0\\
25	20	0\\
25	21	0\\
25	22	0\\
25	23	0\\
25	24	0\\
25	25	0\\
25	26	0\\
25	27	0\\
25	28	0\\
25	29	0\\
25	30	10\\
25	31	10\\
25	32	10\\
25	33	10\\
25	34	10\\
25	35	10\\
25	36	10\\
25	37	100\\
25	38	100\\
25	39	100\\
25	40	100\\
25	41	100\\
25	42	100\\
25	43	100\\
25	44	100\\
25	45	100\\
25	46	100\\
25	47	100\\
25	48	100\\
25	49	100\\
25	50	100\\
25	51	100\\
26	1	0\\
26	2	0\\
26	3	0\\
26	4	0\\
26	5	0\\
26	6	0\\
26	7	0\\
26	8	0\\
26	9	0\\
26	10	0\\
26	11	0\\
26	12	0\\
26	13	0\\
26	14	0\\
26	15	0\\
26	16	0\\
26	17	0\\
26	18	0\\
26	19	0\\
26	20	0\\
26	21	0\\
26	22	0\\
26	23	0\\
26	24	0\\
26	25	0\\
26	26	0\\
26	27	0\\
26	28	0\\
26	29	0\\
26	30	10\\
26	31	10\\
26	32	10\\
26	33	10\\
26	34	10\\
26	35	10\\
26	36	10\\
26	37	100\\
26	38	100\\
26	39	100\\
26	40	100\\
26	41	100\\
26	42	100\\
26	43	100\\
26	44	100\\
26	45	100\\
26	46	100\\
26	47	100\\
26	48	100\\
26	49	100\\
26	50	100\\
26	51	100\\
27	1	0\\
27	2	0\\
27	3	0\\
27	4	0\\
27	5	0\\
27	6	0\\
27	7	0\\
27	8	0\\
27	9	0\\
27	10	0\\
27	11	0\\
27	12	0\\
27	13	0\\
27	14	0\\
27	15	0\\
27	16	0\\
27	17	0\\
27	18	0\\
27	19	0\\
27	20	0\\
27	21	0\\
27	22	0\\
27	23	0\\
27	24	0\\
27	25	0\\
27	26	0\\
27	27	0\\
27	28	0\\
27	29	0\\
27	30	10\\
27	31	10\\
27	32	10\\
27	33	10\\
27	34	10\\
27	35	10\\
27	36	10\\
27	37	100\\
27	38	100\\
27	39	100\\
27	40	100\\
27	41	100\\
27	42	100\\
27	43	100\\
27	44	100\\
27	45	100\\
27	46	100\\
27	47	100\\
27	48	100\\
27	49	100\\
27	50	100\\
27	51	100\\
28	1	0\\
28	2	0\\
28	3	0\\
28	4	0\\
28	5	0\\
28	6	0\\
28	7	0\\
28	8	0\\
28	9	0\\
28	10	0\\
28	11	0\\
28	12	0\\
28	13	0\\
28	14	0\\
28	15	0\\
28	16	0\\
28	17	0\\
28	18	0\\
28	19	0\\
28	20	0\\
28	21	0\\
28	22	0\\
28	23	0\\
28	24	0\\
28	25	0\\
28	26	0\\
28	27	0\\
28	28	0\\
28	29	0\\
28	30	10\\
28	31	10\\
28	32	10\\
28	33	10\\
28	34	10\\
28	35	10\\
28	36	10\\
28	37	10\\
28	38	100\\
28	39	100\\
28	40	100\\
28	41	100\\
28	42	100\\
28	43	100\\
28	44	100\\
28	45	100\\
28	46	100\\
28	47	100\\
28	48	100\\
28	49	100\\
28	50	100\\
28	51	100\\
29	1	0\\
29	2	0\\
29	3	0\\
29	4	0\\
29	5	0\\
29	6	0\\
29	7	0\\
29	8	0\\
29	9	0\\
29	10	0\\
29	11	0\\
29	12	0\\
29	13	0\\
29	14	0\\
29	15	0\\
29	16	0\\
29	17	0\\
29	18	0\\
29	19	0\\
29	20	0\\
29	21	0\\
29	22	0\\
29	23	0\\
29	24	0\\
29	25	0\\
29	26	0\\
29	27	0\\
29	28	0\\
29	29	0\\
29	30	10\\
29	31	10\\
29	32	10\\
29	33	10\\
29	34	10\\
29	35	10\\
29	36	10\\
29	37	10\\
29	38	100\\
29	39	100\\
29	40	100\\
29	41	100\\
29	42	100\\
29	43	100\\
29	44	100\\
29	45	100\\
29	46	100\\
29	47	100\\
29	48	100\\
29	49	100\\
29	50	100\\
29	51	100\\
30	1	0\\
30	2	0\\
30	3	0\\
30	4	0\\
30	5	0\\
30	6	0\\
30	7	0\\
30	8	0\\
30	9	0\\
30	10	0\\
30	11	0\\
30	12	0\\
30	13	0\\
30	14	0\\
30	15	0\\
30	16	0\\
30	17	0\\
30	18	0\\
30	19	0\\
30	20	0\\
30	21	0\\
30	22	0\\
30	23	0\\
30	24	0\\
30	25	0\\
30	26	0\\
30	27	0\\
30	28	0\\
30	29	0\\
30	30	10\\
30	31	10\\
30	32	10\\
30	33	10\\
30	34	10\\
30	35	10\\
30	36	10\\
30	37	10\\
30	38	100\\
30	39	100\\
30	40	100\\
30	41	100\\
30	42	100\\
30	43	100\\
30	44	100\\
30	45	100\\
30	46	100\\
30	47	100\\
30	48	100\\
30	49	100\\
30	50	100\\
30	51	100\\
31	1	0\\
31	2	0\\
31	3	0\\
31	4	0\\
31	5	0\\
31	6	0\\
31	7	0\\
31	8	0\\
31	9	0\\
31	10	0\\
31	11	0\\
31	12	0\\
31	13	0\\
31	14	0\\
31	15	0\\
31	16	0\\
31	17	0\\
31	18	0\\
31	19	0\\
31	20	0\\
31	21	0\\
31	22	0\\
31	23	0\\
31	24	0\\
31	25	0\\
31	26	0\\
31	27	0\\
31	28	0\\
31	29	0\\
31	30	10\\
31	31	10\\
31	32	10\\
31	33	10\\
31	34	10\\
31	35	10\\
31	36	10\\
31	37	10\\
31	38	10\\
31	39	100\\
31	40	100\\
31	41	100\\
31	42	100\\
31	43	100\\
31	44	100\\
31	45	100\\
31	46	100\\
31	47	100\\
31	48	100\\
31	49	100\\
31	50	100\\
31	51	100\\
32	1	0\\
32	2	0\\
32	3	0\\
32	4	0\\
32	5	0\\
32	6	0\\
32	7	0\\
32	8	0\\
32	9	0\\
32	10	0\\
32	11	0\\
32	12	0\\
32	13	0\\
32	14	0\\
32	15	0\\
32	16	0\\
32	17	0\\
32	18	0\\
32	19	0\\
32	20	0\\
32	21	0\\
32	22	0\\
32	23	0\\
32	24	0\\
32	25	0\\
32	26	0\\
32	27	0\\
32	28	0\\
32	29	0\\
32	30	10\\
32	31	10\\
32	32	10\\
32	33	10\\
32	34	10\\
32	35	10\\
32	36	10\\
32	37	10\\
32	38	10\\
32	39	100\\
32	40	100\\
32	41	100\\
32	42	100\\
32	43	100\\
32	44	100\\
32	45	100\\
32	46	100\\
32	47	100\\
32	48	100\\
32	49	100\\
32	50	100\\
32	51	100\\
33	1	0\\
33	2	0\\
33	3	0\\
33	4	0\\
33	5	0\\
33	6	0\\
33	7	0\\
33	8	0\\
33	9	0\\
33	10	0\\
33	11	0\\
33	12	0\\
33	13	0\\
33	14	0\\
33	15	0\\
33	16	0\\
33	17	0\\
33	18	0\\
33	19	0\\
33	20	0\\
33	21	0\\
33	22	0\\
33	23	0\\
33	24	0\\
33	25	0\\
33	26	0\\
33	27	0\\
33	28	0\\
33	29	0\\
33	30	10\\
33	31	10\\
33	32	10\\
33	33	10\\
33	34	10\\
33	35	10\\
33	36	10\\
33	37	10\\
33	38	10\\
33	39	100\\
33	40	100\\
33	41	100\\
33	42	100\\
33	43	100\\
33	44	100\\
33	45	100\\
33	46	100\\
33	47	100\\
33	48	100\\
33	49	100\\
33	50	100\\
33	51	100\\
34	1	0\\
34	2	0\\
34	3	0\\
34	4	0\\
34	5	0\\
34	6	0\\
34	7	0\\
34	8	0\\
34	9	0\\
34	10	0\\
34	11	0\\
34	12	0\\
34	13	0\\
34	14	0\\
34	15	0\\
34	16	0\\
34	17	0\\
34	18	0\\
34	19	0\\
34	20	0\\
34	21	0\\
34	22	0\\
34	23	0\\
34	24	0\\
34	25	0\\
34	26	0\\
34	27	0\\
34	28	0\\
34	29	0\\
34	30	10\\
34	31	10\\
34	32	10\\
34	33	10\\
34	34	10\\
34	35	10\\
34	36	10\\
34	37	10\\
34	38	10\\
34	39	100\\
34	40	100\\
34	41	100\\
34	42	100\\
34	43	100\\
34	44	100\\
34	45	100\\
34	46	100\\
34	47	100\\
34	48	100\\
34	49	100\\
34	50	100\\
34	51	100\\
35	1	0\\
35	2	0\\
35	3	0\\
35	4	0\\
35	5	0\\
35	6	0\\
35	7	0\\
35	8	0\\
35	9	0\\
35	10	0\\
35	11	0\\
35	12	0\\
35	13	0\\
35	14	0\\
35	15	0\\
35	16	0\\
35	17	0\\
35	18	0\\
35	19	0\\
35	20	0\\
35	21	0\\
35	22	0\\
35	23	0\\
35	24	0\\
35	25	0\\
35	26	0\\
35	27	0\\
35	28	0\\
35	29	0\\
35	30	10\\
35	31	10\\
35	32	10\\
35	33	10\\
35	34	10\\
35	35	10\\
35	36	10\\
35	37	10\\
35	38	10\\
35	39	10\\
35	40	100\\
35	41	100\\
35	42	100\\
35	43	100\\
35	44	100\\
35	45	100\\
35	46	100\\
35	47	100\\
35	48	100\\
35	49	100\\
35	50	100\\
35	51	100\\
36	1	0\\
36	2	0\\
36	3	0\\
36	4	0\\
36	5	0\\
36	6	0\\
36	7	0\\
36	8	0\\
36	9	0\\
36	10	0\\
36	11	0\\
36	12	0\\
36	13	0\\
36	14	0\\
36	15	0\\
36	16	0\\
36	17	0\\
36	18	0\\
36	19	0\\
36	20	0\\
36	21	0\\
36	22	0\\
36	23	0\\
36	24	0\\
36	25	0\\
36	26	0\\
36	27	0\\
36	28	0\\
36	29	0\\
36	30	10\\
36	31	10\\
36	32	10\\
36	33	10\\
36	34	10\\
36	35	10\\
36	36	10\\
36	37	10\\
36	38	10\\
36	39	10\\
36	40	100\\
36	41	100\\
36	42	100\\
36	43	100\\
36	44	100\\
36	45	100\\
36	46	100\\
36	47	100\\
36	48	100\\
36	49	100\\
36	50	100\\
36	51	100\\
37	1	0\\
37	2	0\\
37	3	0\\
37	4	0\\
37	5	0\\
37	6	0\\
37	7	0\\
37	8	0\\
37	9	0\\
37	10	0\\
37	11	0\\
37	12	0\\
37	13	0\\
37	14	0\\
37	15	0\\
37	16	0\\
37	17	0\\
37	18	0\\
37	19	0\\
37	20	0\\
37	21	0\\
37	22	0\\
37	23	0\\
37	24	0\\
37	25	0\\
37	26	0\\
37	27	0\\
37	28	0\\
37	29	0\\
37	30	10\\
37	31	10\\
37	32	10\\
37	33	10\\
37	34	10\\
37	35	10\\
37	36	10\\
37	37	10\\
37	38	10\\
37	39	10\\
37	40	100\\
37	41	100\\
37	42	100\\
37	43	100\\
37	44	100\\
37	45	100\\
37	46	100\\
37	47	100\\
37	48	100\\
37	49	100\\
37	50	100\\
37	51	100\\
38	1	0\\
38	2	0\\
38	3	0\\
38	4	0\\
38	5	0\\
38	6	0\\
38	7	0\\
38	8	0\\
38	9	0\\
38	10	0\\
38	11	0\\
38	12	0\\
38	13	0\\
38	14	0\\
38	15	0\\
38	16	0\\
38	17	0\\
38	18	0\\
38	19	0\\
38	20	0\\
38	21	0\\
38	22	0\\
38	23	0\\
38	24	0\\
38	25	0\\
38	26	0\\
38	27	0\\
38	28	0\\
38	29	0\\
38	30	10\\
38	31	10\\
38	32	10\\
38	33	10\\
38	34	10\\
38	35	10\\
38	36	10\\
38	37	10\\
38	38	10\\
38	39	10\\
38	40	100\\
38	41	100\\
38	42	100\\
38	43	100\\
38	44	100\\
38	45	100\\
38	46	100\\
38	47	100\\
38	48	100\\
38	49	100\\
38	50	100\\
38	51	100\\
39	1	0\\
39	2	0\\
39	3	0\\
39	4	0\\
39	5	0\\
39	6	0\\
39	7	0\\
39	8	0\\
39	9	0\\
39	10	0\\
39	11	0\\
39	12	0\\
39	13	0\\
39	14	0\\
39	15	0\\
39	16	0\\
39	17	0\\
39	18	0\\
39	19	0\\
39	20	0\\
39	21	0\\
39	22	0\\
39	23	0\\
39	24	0\\
39	25	0\\
39	26	0\\
39	27	0\\
39	28	0\\
39	29	0\\
39	30	10\\
39	31	10\\
39	32	10\\
39	33	10\\
39	34	10\\
39	35	10\\
39	36	10\\
39	37	10\\
39	38	10\\
39	39	10\\
39	40	100\\
39	41	100\\
39	42	100\\
39	43	100\\
39	44	100\\
39	45	100\\
39	46	100\\
39	47	100\\
39	48	100\\
39	49	100\\
39	50	100\\
39	51	100\\
40	1	0\\
40	2	0\\
40	3	0\\
40	4	0\\
40	5	0\\
40	6	0\\
40	7	0\\
40	8	0\\
40	9	0\\
40	10	0\\
40	11	0\\
40	12	0\\
40	13	0\\
40	14	0\\
40	15	0\\
40	16	0\\
40	17	0\\
40	18	0\\
40	19	0\\
40	20	0\\
40	21	0\\
40	22	0\\
40	23	0\\
40	24	0\\
40	25	0\\
40	26	0\\
40	27	0\\
40	28	0\\
40	29	0\\
40	30	10\\
40	31	10\\
40	32	10\\
40	33	10\\
40	34	10\\
40	35	10\\
40	36	10\\
40	37	10\\
40	38	10\\
40	39	10\\
40	40	10\\
40	41	100\\
40	42	100\\
40	43	100\\
40	44	100\\
40	45	100\\
40	46	100\\
40	47	100\\
40	48	100\\
40	49	100\\
40	50	100\\
40	51	100\\
41	1	0\\
41	2	0\\
41	3	0\\
41	4	0\\
41	5	0\\
41	6	0\\
41	7	0\\
41	8	0\\
41	9	0\\
41	10	0\\
41	11	0\\
41	12	0\\
41	13	0\\
41	14	0\\
41	15	0\\
41	16	0\\
41	17	0\\
41	18	0\\
41	19	0\\
41	20	0\\
41	21	0\\
41	22	0\\
41	23	0\\
41	24	0\\
41	25	0\\
41	26	0\\
41	27	0\\
41	28	0\\
41	29	0\\
41	30	10\\
41	31	10\\
41	32	10\\
41	33	10\\
41	34	10\\
41	35	10\\
41	36	10\\
41	37	10\\
41	38	10\\
41	39	10\\
41	40	10\\
41	41	100\\
41	42	100\\
41	43	100\\
41	44	100\\
41	45	100\\
41	46	100\\
41	47	100\\
41	48	100\\
41	49	100\\
41	50	100\\
41	51	100\\
42	1	0\\
42	2	0\\
42	3	0\\
42	4	0\\
42	5	0\\
42	6	0\\
42	7	0\\
42	8	0\\
42	9	0\\
42	10	0\\
42	11	0\\
42	12	0\\
42	13	0\\
42	14	0\\
42	15	0\\
42	16	0\\
42	17	0\\
42	18	0\\
42	19	0\\
42	20	0\\
42	21	0\\
42	22	0\\
42	23	0\\
42	24	0\\
42	25	0\\
42	26	0\\
42	27	0\\
42	28	0\\
42	29	0\\
42	30	10\\
42	31	10\\
42	32	10\\
42	33	10\\
42	34	10\\
42	35	10\\
42	36	10\\
42	37	10\\
42	38	10\\
42	39	10\\
42	40	10\\
42	41	100\\
42	42	100\\
42	43	100\\
42	44	100\\
42	45	100\\
42	46	100\\
42	47	100\\
42	48	100\\
42	49	100\\
42	50	100\\
42	51	100\\
43	1	0\\
43	2	0\\
43	3	0\\
43	4	0\\
43	5	0\\
43	6	0\\
43	7	0\\
43	8	0\\
43	9	0\\
43	10	0\\
43	11	0\\
43	12	0\\
43	13	0\\
43	14	0\\
43	15	0\\
43	16	0\\
43	17	0\\
43	18	0\\
43	19	0\\
43	20	0\\
43	21	0\\
43	22	0\\
43	23	0\\
43	24	0\\
43	25	0\\
43	26	0\\
43	27	0\\
43	28	0\\
43	29	0\\
43	30	10\\
43	31	10\\
43	32	10\\
43	33	10\\
43	34	10\\
43	35	10\\
43	36	10\\
43	37	10\\
43	38	10\\
43	39	10\\
43	40	10\\
43	41	100\\
43	42	100\\
43	43	100\\
43	44	100\\
43	45	100\\
43	46	100\\
43	47	100\\
43	48	100\\
43	49	100\\
43	50	100\\
43	51	100\\
44	1	0\\
44	2	0\\
44	3	0\\
44	4	0\\
44	5	0\\
44	6	0\\
44	7	0\\
44	8	0\\
44	9	0\\
44	10	0\\
44	11	0\\
44	12	0\\
44	13	0\\
44	14	0\\
44	15	0\\
44	16	0\\
44	17	0\\
44	18	0\\
44	19	0\\
44	20	0\\
44	21	0\\
44	22	0\\
44	23	0\\
44	24	0\\
44	25	0\\
44	26	0\\
44	27	0\\
44	28	0\\
44	29	0\\
44	30	10\\
44	31	10\\
44	32	10\\
44	33	10\\
44	34	10\\
44	35	10\\
44	36	10\\
44	37	10\\
44	38	10\\
44	39	10\\
44	40	10\\
44	41	100\\
44	42	100\\
44	43	100\\
44	44	100\\
44	45	100\\
44	46	100\\
44	47	100\\
44	48	100\\
44	49	100\\
44	50	100\\
44	51	100\\
45	1	0\\
45	2	0\\
45	3	0\\
45	4	0\\
45	5	0\\
45	6	0\\
45	7	0\\
45	8	0\\
45	9	0\\
45	10	0\\
45	11	0\\
45	12	0\\
45	13	0\\
45	14	0\\
45	15	0\\
45	16	0\\
45	17	0\\
45	18	0\\
45	19	0\\
45	20	0\\
45	21	0\\
45	22	0\\
45	23	0\\
45	24	0\\
45	25	0\\
45	26	0\\
45	27	0\\
45	28	0\\
45	29	0\\
45	30	10\\
45	31	10\\
45	32	10\\
45	33	10\\
45	34	10\\
45	35	10\\
45	36	10\\
45	37	10\\
45	38	10\\
45	39	10\\
45	40	10\\
45	41	100\\
45	42	100\\
45	43	100\\
45	44	100\\
45	45	100\\
45	46	100\\
45	47	100\\
45	48	100\\
45	49	100\\
45	50	100\\
45	51	100\\
46	1	0\\
46	2	0\\
46	3	0\\
46	4	0\\
46	5	0\\
46	6	0\\
46	7	0\\
46	8	0\\
46	9	0\\
46	10	0\\
46	11	0\\
46	12	0\\
46	13	0\\
46	14	0\\
46	15	0\\
46	16	0\\
46	17	0\\
46	18	0\\
46	19	0\\
46	20	0\\
46	21	0\\
46	22	0\\
46	23	0\\
46	24	0\\
46	25	0\\
46	26	0\\
46	27	0\\
46	28	0\\
46	29	0\\
46	30	10\\
46	31	10\\
46	32	10\\
46	33	10\\
46	34	10\\
46	35	10\\
46	36	10\\
46	37	10\\
46	38	10\\
46	39	10\\
46	40	10\\
46	41	10\\
46	42	100\\
46	43	100\\
46	44	100\\
46	45	100\\
46	46	100\\
46	47	100\\
46	48	100\\
46	49	100\\
46	50	100\\
46	51	100\\
47	1	0\\
47	2	0\\
47	3	0\\
47	4	0\\
47	5	0\\
47	6	0\\
47	7	0\\
47	8	0\\
47	9	0\\
47	10	0\\
47	11	0\\
47	12	0\\
47	13	0\\
47	14	0\\
47	15	0\\
47	16	0\\
47	17	0\\
47	18	0\\
47	19	0\\
47	20	0\\
47	21	0\\
47	22	0\\
47	23	0\\
47	24	0\\
47	25	0\\
47	26	0\\
47	27	0\\
47	28	0\\
47	29	0\\
47	30	10\\
47	31	10\\
47	32	10\\
47	33	10\\
47	34	10\\
47	35	10\\
47	36	10\\
47	37	10\\
47	38	10\\
47	39	10\\
47	40	10\\
47	41	10\\
47	42	100\\
47	43	100\\
47	44	100\\
47	45	100\\
47	46	100\\
47	47	100\\
47	48	100\\
47	49	100\\
47	50	100\\
47	51	100\\
48	1	0\\
48	2	0\\
48	3	0\\
48	4	0\\
48	5	0\\
48	6	0\\
48	7	0\\
48	8	0\\
48	9	0\\
48	10	0\\
48	11	0\\
48	12	0\\
48	13	0\\
48	14	0\\
48	15	0\\
48	16	0\\
48	17	0\\
48	18	0\\
48	19	0\\
48	20	0\\
48	21	0\\
48	22	0\\
48	23	0\\
48	24	0\\
48	25	0\\
48	26	0\\
48	27	0\\
48	28	0\\
48	29	0\\
48	30	10\\
48	31	10\\
48	32	10\\
48	33	10\\
48	34	10\\
48	35	10\\
48	36	10\\
48	37	10\\
48	38	10\\
48	39	10\\
48	40	10\\
48	41	10\\
48	42	100\\
48	43	100\\
48	44	100\\
48	45	100\\
48	46	100\\
48	47	100\\
48	48	100\\
48	49	100\\
48	50	100\\
48	51	100\\
49	1	0\\
49	2	0\\
49	3	0\\
49	4	0\\
49	5	0\\
49	6	0\\
49	7	0\\
49	8	0\\
49	9	0\\
49	10	0\\
49	11	0\\
49	12	0\\
49	13	0\\
49	14	0\\
49	15	0\\
49	16	0\\
49	17	0\\
49	18	0\\
49	19	0\\
49	20	0\\
49	21	0\\
49	22	0\\
49	23	0\\
49	24	0\\
49	25	0\\
49	26	0\\
49	27	0\\
49	28	0\\
49	29	0\\
49	30	10\\
49	31	10\\
49	32	10\\
49	33	10\\
49	34	10\\
49	35	10\\
49	36	10\\
49	37	10\\
49	38	10\\
49	39	10\\
49	40	10\\
49	41	10\\
49	42	100\\
49	43	100\\
49	44	100\\
49	45	100\\
49	46	100\\
49	47	100\\
49	48	100\\
49	49	100\\
49	50	100\\
49	51	100\\
50	1	0\\
50	2	0\\
50	3	0\\
50	4	0\\
50	5	0\\
50	6	0\\
50	7	0\\
50	8	0\\
50	9	0\\
50	10	0\\
50	11	0\\
50	12	0\\
50	13	0\\
50	14	0\\
50	15	0\\
50	16	0\\
50	17	0\\
50	18	0\\
50	19	0\\
50	20	0\\
50	21	0\\
50	22	0\\
50	23	0\\
50	24	0\\
50	25	0\\
50	26	0\\
50	27	0\\
50	28	0\\
50	29	0\\
50	30	10\\
50	31	10\\
50	32	10\\
50	33	10\\
50	34	10\\
50	35	10\\
50	36	10\\
50	37	10\\
50	38	10\\
50	39	10\\
50	40	10\\
50	41	10\\
50	42	100\\
50	43	100\\
50	44	100\\
50	45	100\\
50	46	100\\
50	47	100\\
50	48	100\\
50	49	100\\
50	50	100\\
50	51	100\\
51	1	0\\
51	2	0\\
51	3	0\\
51	4	0\\
51	5	0\\
51	6	0\\
51	7	0\\
51	8	0\\
51	9	0\\
51	10	0\\
51	11	0\\
51	12	0\\
51	13	0\\
51	14	0\\
51	15	0\\
51	16	0\\
51	17	0\\
51	18	0\\
51	19	0\\
51	20	0\\
51	21	0\\
51	22	0\\
51	23	0\\
51	24	0\\
51	25	0\\
51	26	0\\
51	27	0\\
51	28	0\\
51	29	0\\
51	30	10\\
51	31	10\\
51	32	10\\
51	33	10\\
51	34	10\\
51	35	10\\
51	36	10\\
51	37	10\\
51	38	10\\
51	39	10\\
51	40	10\\
51	41	10\\
51	42	100\\
51	43	100\\
51	44	100\\
51	45	100\\
51	46	100\\
51	47	100\\
51	48	100\\
51	49	100\\
51	50	100\\
51	51	100\\
};
\end{axis}
\end{tikzpicture}%

%% file: pics/pictureExperimentConvergencyfullydec3col.tex
%
%
\definecolor{mycolor1}{rgb}{0.12941,0.12941,0.12941}%
\begin{tikzpicture}

\begin{axis}[%
width=4.0in,
height=2.0in,
scale only axis,
xmin=1,
xmax=51,
xlabel style={font=\color{mycolor1}},
xlabel={$\ccc_0$},
xtick={5, 15, 25, 35, 45},
xticklabels={$1$, $2$, $3$, $4$, $5$},   
ymin=1,
ymax=45,
ytick={7.815, 15.625, 23.4375, 31.25, 39.0625, 46.875},
yticklabels={0.1, 0.2, 0.3, 0.4, 0.5, 0.6}, 
ylabel style={font=\color{mycolor1}},
ylabel={$\alpha=\beta$},
axis background/.style={fill=white},
colormap={custom}{color(0)=(color2) color(10)=(yellow) color(100)=(red)},
]

\addplot[%
surf,
shader=flat corner, draw=mycolor1, 
mesh/rows=51]
table[row sep=crcr, point meta=\thisrow{c}] {%
x	y	c\\
1	1	10\\
1	2	10\\
1	3	10\\
1	4	10\\
1	5	10\\
1	6	10\\
1	7	10\\
1	8	10\\
1	9	10\\
1	10	10\\
1	11	10\\
1	12	10\\
1	13	10\\
1	14	10\\
1	15	10\\
1	16	10\\
1	17	10\\
1	18	10\\
1	19	10\\
1	20	10\\
1	21	10\\
1	22	100\\
1	23	100\\
1	24	100\\
1	25	100\\
1	26	100\\
1	27	100\\
1	28	100\\
1	29	100\\
1	30	100\\
1	31	100\\
1	32	100\\
1	33	100\\
1	34	100\\
1	35	100\\
1	36	100\\
1	37	100\\
1	38	100\\
1	39	100\\
1	40	100\\
1	41	100\\
1	42	100\\
1	43	100\\
1	44	100\\
1	45	100\\
1	46	100\\
1	47	100\\
1	48	100\\
1	49	100\\
1	50	100\\
1	51	100\\
2	1	0\\
2	2	0\\
2	3	0\\
2	4	0\\
2	5	0\\
2	6	0\\
2	7	10\\
2	8	10\\
2	9	10\\
2	10	10\\
2	11	10\\
2	12	10\\
2	13	10\\
2	14	10\\
2	15	10\\
2	16	10\\
2	17	10\\
2	18	10\\
2	19	10\\
2	20	10\\
2	21	10\\
2	22	10\\
2	23	100\\
2	24	100\\
2	25	100\\
2	26	100\\
2	27	100\\
2	28	100\\
2	29	100\\
2	30	100\\
2	31	100\\
2	32	100\\
2	33	100\\
2	34	100\\
2	35	100\\
2	36	100\\
2	37	100\\
2	38	100\\
2	39	100\\
2	40	100\\
2	41	100\\
2	42	100\\
2	43	100\\
2	44	100\\
2	45	100\\
2	46	100\\
2	47	100\\
2	48	100\\
2	49	100\\
2	50	100\\
2	51	100\\
3	1	0\\
3	2	0\\
3	3	0\\
3	4	0\\
3	5	0\\
3	6	0\\
3	7	0\\
3	8	0\\
3	9	0\\
3	10	10\\
3	11	10\\
3	12	10\\
3	13	10\\
3	14	10\\
3	15	10\\
3	16	10\\
3	17	10\\
3	18	10\\
3	19	10\\
3	20	10\\
3	21	10\\
3	22	10\\
3	23	10\\
3	24	10\\
3	25	100\\
3	26	100\\
3	27	100\\
3	28	100\\
3	29	100\\
3	30	100\\
3	31	100\\
3	32	100\\
3	33	100\\
3	34	100\\
3	35	100\\
3	36	100\\
3	37	100\\
3	38	100\\
3	39	100\\
3	40	100\\
3	41	100\\
3	42	100\\
3	43	100\\
3	44	100\\
3	45	100\\
3	46	100\\
3	47	100\\
3	48	100\\
3	49	100\\
3	50	100\\
3	51	100\\
4	1	0\\
4	2	0\\
4	3	0\\
4	4	0\\
4	5	0\\
4	6	0\\
4	7	0\\
4	8	0\\
4	9	0\\
4	10	0\\
4	11	0\\
4	12	10\\
4	13	10\\
4	14	10\\
4	15	10\\
4	16	10\\
4	17	10\\
4	18	10\\
4	19	10\\
4	20	10\\
4	21	10\\
4	22	10\\
4	23	10\\
4	24	10\\
4	25	10\\
4	26	100\\
4	27	100\\
4	28	100\\
4	29	100\\
4	30	100\\
4	31	100\\
4	32	100\\
4	33	100\\
4	34	100\\
4	35	100\\
4	36	100\\
4	37	100\\
4	38	100\\
4	39	100\\
4	40	100\\
4	41	100\\
4	42	100\\
4	43	100\\
4	44	100\\
4	45	100\\
4	46	100\\
4	47	100\\
4	48	100\\
4	49	100\\
4	50	100\\
4	51	100\\
5	1	0\\
5	2	0\\
5	3	0\\
5	4	0\\
5	5	0\\
5	6	0\\
5	7	0\\
5	8	0\\
5	9	0\\
5	10	0\\
5	11	0\\
5	12	0\\
5	13	10\\
5	14	10\\
5	15	10\\
5	16	10\\
5	17	10\\
5	18	10\\
5	19	10\\
5	20	10\\
5	21	10\\
5	22	10\\
5	23	10\\
5	24	10\\
5	25	10\\
5	26	10\\
5	27	100\\
5	28	100\\
5	29	100\\
5	30	100\\
5	31	100\\
5	32	100\\
5	33	100\\
5	34	100\\
5	35	100\\
5	36	100\\
5	37	100\\
5	38	100\\
5	39	100\\
5	40	100\\
5	41	100\\
5	42	100\\
5	43	100\\
5	44	100\\
5	45	100\\
5	46	100\\
5	47	100\\
5	48	100\\
5	49	100\\
5	50	100\\
5	51	100\\
6	1	0\\
6	2	0\\
6	3	0\\
6	4	0\\
6	5	0\\
6	6	0\\
6	7	0\\
6	8	0\\
6	9	0\\
6	10	0\\
6	11	0\\
6	12	0\\
6	13	0\\
6	14	0\\
6	15	10\\
6	16	10\\
6	17	10\\
6	18	10\\
6	19	10\\
6	20	10\\
6	21	10\\
6	22	10\\
6	23	10\\
6	24	10\\
6	25	10\\
6	26	10\\
6	27	10\\
6	28	100\\
6	29	100\\
6	30	100\\
6	31	100\\
6	32	100\\
6	33	100\\
6	34	100\\
6	35	100\\
6	36	100\\
6	37	100\\
6	38	100\\
6	39	100\\
6	40	100\\
6	41	100\\
6	42	100\\
6	43	100\\
6	44	100\\
6	45	100\\
6	46	100\\
6	47	100\\
6	48	100\\
6	49	100\\
6	50	100\\
6	51	100\\
7	1	0\\
7	2	0\\
7	3	0\\
7	4	0\\
7	5	0\\
7	6	0\\
7	7	0\\
7	8	0\\
7	9	0\\
7	10	0\\
7	11	0\\
7	12	0\\
7	13	0\\
7	14	0\\
7	15	0\\
7	16	10\\
7	17	10\\
7	18	10\\
7	19	10\\
7	20	10\\
7	21	10\\
7	22	10\\
7	23	10\\
7	24	10\\
7	25	10\\
7	26	10\\
7	27	10\\
7	28	10\\
7	29	100\\
7	30	100\\
7	31	100\\
7	32	100\\
7	33	100\\
7	34	100\\
7	35	100\\
7	36	100\\
7	37	100\\
7	38	100\\
7	39	100\\
7	40	100\\
7	41	100\\
7	42	100\\
7	43	100\\
7	44	100\\
7	45	100\\
7	46	100\\
7	47	100\\
7	48	100\\
7	49	100\\
7	50	100\\
7	51	100\\
8	1	0\\
8	2	0\\
8	3	0\\
8	4	0\\
8	5	0\\
8	6	0\\
8	7	0\\
8	8	0\\
8	9	0\\
8	10	0\\
8	11	0\\
8	12	0\\
8	13	0\\
8	14	0\\
8	15	0\\
8	16	0\\
8	17	10\\
8	18	10\\
8	19	10\\
8	20	10\\
8	21	10\\
8	22	10\\
8	23	10\\
8	24	10\\
8	25	10\\
8	26	10\\
8	27	10\\
8	28	10\\
8	29	10\\
8	30	100\\
8	31	100\\
8	32	100\\
8	33	100\\
8	34	100\\
8	35	100\\
8	36	100\\
8	37	100\\
8	38	100\\
8	39	100\\
8	40	100\\
8	41	100\\
8	42	100\\
8	43	100\\
8	44	100\\
8	45	100\\
8	46	100\\
8	47	100\\
8	48	100\\
8	49	100\\
8	50	100\\
8	51	100\\
9	1	0\\
9	2	0\\
9	3	0\\
9	4	0\\
9	5	0\\
9	6	0\\
9	7	0\\
9	8	0\\
9	9	0\\
9	10	0\\
9	11	0\\
9	12	0\\
9	13	0\\
9	14	0\\
9	15	0\\
9	16	0\\
9	17	0\\
9	18	10\\
9	19	10\\
9	20	10\\
9	21	10\\
9	22	10\\
9	23	10\\
9	24	10\\
9	25	10\\
9	26	10\\
9	27	10\\
9	28	10\\
9	29	10\\
9	30	100\\
9	31	100\\
9	32	100\\
9	33	100\\
9	34	100\\
9	35	100\\
9	36	100\\
9	37	100\\
9	38	100\\
9	39	100\\
9	40	100\\
9	41	100\\
9	42	100\\
9	43	100\\
9	44	100\\
9	45	100\\
9	46	100\\
9	47	100\\
9	48	100\\
9	49	100\\
9	50	100\\
9	51	100\\
10	1	0\\
10	2	0\\
10	3	0\\
10	4	0\\
10	5	0\\
10	6	0\\
10	7	0\\
10	8	0\\
10	9	0\\
10	10	0\\
10	11	0\\
10	12	0\\
10	13	0\\
10	14	0\\
10	15	0\\
10	16	0\\
10	17	0\\
10	18	0\\
10	19	10\\
10	20	10\\
10	21	10\\
10	22	10\\
10	23	10\\
10	24	10\\
10	25	10\\
10	26	10\\
10	27	10\\
10	28	10\\
10	29	10\\
10	30	10\\
10	31	100\\
10	32	100\\
10	33	100\\
10	34	100\\
10	35	100\\
10	36	100\\
10	37	100\\
10	38	100\\
10	39	100\\
10	40	100\\
10	41	100\\
10	42	100\\
10	43	100\\
10	44	100\\
10	45	100\\
10	46	100\\
10	47	100\\
10	48	100\\
10	49	100\\
10	50	100\\
10	51	100\\
11	1	0\\
11	2	0\\
11	3	0\\
11	4	0\\
11	5	0\\
11	6	0\\
11	7	0\\
11	8	0\\
11	9	0\\
11	10	0\\
11	11	0\\
11	12	0\\
11	13	0\\
11	14	0\\
11	15	0\\
11	16	0\\
11	17	0\\
11	18	0\\
11	19	0\\
11	20	10\\
11	21	10\\
11	22	10\\
11	23	10\\
11	24	10\\
11	25	10\\
11	26	10\\
11	27	10\\
11	28	10\\
11	29	10\\
11	30	10\\
11	31	100\\
11	32	100\\
11	33	100\\
11	34	100\\
11	35	100\\
11	36	100\\
11	37	100\\
11	38	100\\
11	39	100\\
11	40	100\\
11	41	100\\
11	42	100\\
11	43	100\\
11	44	100\\
11	45	100\\
11	46	100\\
11	47	100\\
11	48	100\\
11	49	100\\
11	50	100\\
11	51	100\\
12	1	0\\
12	2	0\\
12	3	0\\
12	4	0\\
12	5	0\\
12	6	0\\
12	7	0\\
12	8	0\\
12	9	0\\
12	10	0\\
12	11	0\\
12	12	0\\
12	13	0\\
12	14	0\\
12	15	0\\
12	16	0\\
12	17	0\\
12	18	0\\
12	19	0\\
12	20	0\\
12	21	10\\
12	22	10\\
12	23	10\\
12	24	10\\
12	25	10\\
12	26	10\\
12	27	10\\
12	28	10\\
12	29	10\\
12	30	10\\
12	31	10\\
12	32	100\\
12	33	100\\
12	34	100\\
12	35	100\\
12	36	100\\
12	37	100\\
12	38	100\\
12	39	100\\
12	40	100\\
12	41	100\\
12	42	100\\
12	43	100\\
12	44	100\\
12	45	100\\
12	46	100\\
12	47	100\\
12	48	100\\
12	49	100\\
12	50	100\\
12	51	100\\
13	1	0\\
13	2	0\\
13	3	0\\
13	4	0\\
13	5	0\\
13	6	0\\
13	7	0\\
13	8	0\\
13	9	0\\
13	10	0\\
13	11	0\\
13	12	0\\
13	13	0\\
13	14	0\\
13	15	0\\
13	16	0\\
13	17	0\\
13	18	0\\
13	19	0\\
13	20	0\\
13	21	0\\
13	22	10\\
13	23	10\\
13	24	10\\
13	25	10\\
13	26	10\\
13	27	10\\
13	28	10\\
13	29	10\\
13	30	10\\
13	31	10\\
13	32	100\\
13	33	100\\
13	34	100\\
13	35	100\\
13	36	100\\
13	37	100\\
13	38	100\\
13	39	100\\
13	40	100\\
13	41	100\\
13	42	100\\
13	43	100\\
13	44	100\\
13	45	100\\
13	46	100\\
13	47	100\\
13	48	100\\
13	49	100\\
13	50	100\\
13	51	100\\
14	1	0\\
14	2	0\\
14	3	0\\
14	4	0\\
14	5	0\\
14	6	0\\
14	7	0\\
14	8	0\\
14	9	0\\
14	10	0\\
14	11	0\\
14	12	0\\
14	13	0\\
14	14	0\\
14	15	0\\
14	16	0\\
14	17	0\\
14	18	0\\
14	19	0\\
14	20	0\\
14	21	0\\
14	22	10\\
14	23	10\\
14	24	10\\
14	25	10\\
14	26	10\\
14	27	10\\
14	28	10\\
14	29	10\\
14	30	10\\
14	31	10\\
14	32	10\\
14	33	100\\
14	34	100\\
14	35	100\\
14	36	100\\
14	37	100\\
14	38	100\\
14	39	100\\
14	40	100\\
14	41	100\\
14	42	100\\
14	43	100\\
14	44	100\\
14	45	100\\
14	46	100\\
14	47	100\\
14	48	100\\
14	49	100\\
14	50	100\\
14	51	100\\
15	1	0\\
15	2	0\\
15	3	0\\
15	4	0\\
15	5	0\\
15	6	0\\
15	7	0\\
15	8	0\\
15	9	0\\
15	10	0\\
15	11	0\\
15	12	0\\
15	13	0\\
15	14	0\\
15	15	0\\
15	16	0\\
15	17	0\\
15	18	0\\
15	19	0\\
15	20	0\\
15	21	0\\
15	22	0\\
15	23	10\\
15	24	10\\
15	25	10\\
15	26	10\\
15	27	10\\
15	28	10\\
15	29	10\\
15	30	10\\
15	31	10\\
15	32	10\\
15	33	100\\
15	34	100\\
15	35	100\\
15	36	100\\
15	37	100\\
15	38	100\\
15	39	100\\
15	40	100\\
15	41	100\\
15	42	100\\
15	43	100\\
15	44	100\\
15	45	100\\
15	46	100\\
15	47	100\\
15	48	100\\
15	49	100\\
15	50	100\\
15	51	100\\
16	1	0\\
16	2	0\\
16	3	0\\
16	4	0\\
16	5	0\\
16	6	0\\
16	7	0\\
16	8	0\\
16	9	0\\
16	10	0\\
16	11	0\\
16	12	0\\
16	13	0\\
16	14	0\\
16	15	0\\
16	16	0\\
16	17	0\\
16	18	0\\
16	19	0\\
16	20	0\\
16	21	0\\
16	22	0\\
16	23	0\\
16	24	10\\
16	25	10\\
16	26	10\\
16	27	10\\
16	28	10\\
16	29	10\\
16	30	10\\
16	31	10\\
16	32	10\\
16	33	100\\
16	34	100\\
16	35	100\\
16	36	100\\
16	37	100\\
16	38	100\\
16	39	100\\
16	40	100\\
16	41	100\\
16	42	100\\
16	43	100\\
16	44	100\\
16	45	100\\
16	46	100\\
16	47	100\\
16	48	100\\
16	49	100\\
16	50	100\\
16	51	100\\
17	1	0\\
17	2	0\\
17	3	0\\
17	4	0\\
17	5	0\\
17	6	0\\
17	7	0\\
17	8	0\\
17	9	0\\
17	10	0\\
17	11	0\\
17	12	0\\
17	13	0\\
17	14	0\\
17	15	0\\
17	16	0\\
17	17	0\\
17	18	0\\
17	19	0\\
17	20	0\\
17	21	0\\
17	22	0\\
17	23	0\\
17	24	10\\
17	25	10\\
17	26	10\\
17	27	10\\
17	28	10\\
17	29	10\\
17	30	10\\
17	31	10\\
17	32	10\\
17	33	10\\
17	34	100\\
17	35	100\\
17	36	100\\
17	37	100\\
17	38	100\\
17	39	100\\
17	40	100\\
17	41	100\\
17	42	100\\
17	43	100\\
17	44	100\\
17	45	100\\
17	46	100\\
17	47	100\\
17	48	100\\
17	49	100\\
17	50	100\\
17	51	100\\
18	1	0\\
18	2	0\\
18	3	0\\
18	4	0\\
18	5	0\\
18	6	0\\
18	7	0\\
18	8	0\\
18	9	0\\
18	10	0\\
18	11	0\\
18	12	0\\
18	13	0\\
18	14	0\\
18	15	0\\
18	16	0\\
18	17	0\\
18	18	0\\
18	19	0\\
18	20	0\\
18	21	0\\
18	22	0\\
18	23	0\\
18	24	10\\
18	25	10\\
18	26	10\\
18	27	10\\
18	28	10\\
18	29	10\\
18	30	10\\
18	31	10\\
18	32	10\\
18	33	10\\
18	34	100\\
18	35	100\\
18	36	100\\
18	37	100\\
18	38	100\\
18	39	100\\
18	40	100\\
18	41	100\\
18	42	100\\
18	43	100\\
18	44	100\\
18	45	100\\
18	46	100\\
18	47	100\\
18	48	100\\
18	49	100\\
18	50	100\\
18	51	100\\
19	1	0\\
19	2	0\\
19	3	0\\
19	4	0\\
19	5	0\\
19	6	0\\
19	7	0\\
19	8	0\\
19	9	0\\
19	10	0\\
19	11	0\\
19	12	0\\
19	13	0\\
19	14	0\\
19	15	0\\
19	16	0\\
19	17	0\\
19	18	0\\
19	19	0\\
19	20	0\\
19	21	0\\
19	22	0\\
19	23	0\\
19	24	10\\
19	25	10\\
19	26	10\\
19	27	10\\
19	28	10\\
19	29	10\\
19	30	10\\
19	31	10\\
19	32	10\\
19	33	10\\
19	34	100\\
19	35	100\\
19	36	100\\
19	37	100\\
19	38	100\\
19	39	100\\
19	40	100\\
19	41	100\\
19	42	100\\
19	43	100\\
19	44	100\\
19	45	100\\
19	46	100\\
19	47	100\\
19	48	100\\
19	49	100\\
19	50	100\\
19	51	100\\
20	1	0\\
20	2	0\\
20	3	0\\
20	4	0\\
20	5	0\\
20	6	0\\
20	7	0\\
20	8	0\\
20	9	0\\
20	10	0\\
20	11	0\\
20	12	0\\
20	13	0\\
20	14	0\\
20	15	0\\
20	16	0\\
20	17	0\\
20	18	0\\
20	19	0\\
20	20	0\\
20	21	0\\
20	22	0\\
20	23	0\\
20	24	10\\
20	25	10\\
20	26	10\\
20	27	10\\
20	28	10\\
20	29	10\\
20	30	10\\
20	31	10\\
20	32	10\\
20	33	10\\
20	34	100\\
20	35	100\\
20	36	100\\
20	37	100\\
20	38	100\\
20	39	100\\
20	40	100\\
20	41	100\\
20	42	100\\
20	43	100\\
20	44	100\\
20	45	100\\
20	46	100\\
20	47	100\\
20	48	100\\
20	49	100\\
20	50	100\\
20	51	100\\
21	1	0\\
21	2	0\\
21	3	0\\
21	4	0\\
21	5	0\\
21	6	0\\
21	7	0\\
21	8	0\\
21	9	0\\
21	10	0\\
21	11	0\\
21	12	0\\
21	13	0\\
21	14	0\\
21	15	0\\
21	16	0\\
21	17	0\\
21	18	0\\
21	19	0\\
21	20	0\\
21	21	0\\
21	22	0\\
21	23	0\\
21	24	10\\
21	25	10\\
21	26	10\\
21	27	10\\
21	28	10\\
21	29	10\\
21	30	10\\
21	31	10\\
21	32	10\\
21	33	10\\
21	34	10\\
21	35	100\\
21	36	100\\
21	37	100\\
21	38	100\\
21	39	100\\
21	40	100\\
21	41	100\\
21	42	100\\
21	43	100\\
21	44	100\\
21	45	100\\
21	46	100\\
21	47	100\\
21	48	100\\
21	49	100\\
21	50	100\\
21	51	100\\
22	1	0\\
22	2	0\\
22	3	0\\
22	4	0\\
22	5	0\\
22	6	0\\
22	7	0\\
22	8	0\\
22	9	0\\
22	10	0\\
22	11	0\\
22	12	0\\
22	13	0\\
22	14	0\\
22	15	0\\
22	16	0\\
22	17	0\\
22	18	0\\
22	19	0\\
22	20	0\\
22	21	0\\
22	22	0\\
22	23	0\\
22	24	10\\
22	25	10\\
22	26	10\\
22	27	10\\
22	28	10\\
22	29	10\\
22	30	10\\
22	31	10\\
22	32	10\\
22	33	10\\
22	34	10\\
22	35	100\\
22	36	100\\
22	37	100\\
22	38	100\\
22	39	100\\
22	40	100\\
22	41	100\\
22	42	100\\
22	43	100\\
22	44	100\\
22	45	100\\
22	46	100\\
22	47	100\\
22	48	100\\
22	49	100\\
22	50	100\\
22	51	100\\
23	1	0\\
23	2	0\\
23	3	0\\
23	4	0\\
23	5	0\\
23	6	0\\
23	7	0\\
23	8	0\\
23	9	0\\
23	10	0\\
23	11	0\\
23	12	0\\
23	13	0\\
23	14	0\\
23	15	0\\
23	16	0\\
23	17	0\\
23	18	0\\
23	19	0\\
23	20	0\\
23	21	0\\
23	22	0\\
23	23	0\\
23	24	10\\
23	25	10\\
23	26	10\\
23	27	10\\
23	28	10\\
23	29	10\\
23	30	10\\
23	31	10\\
23	32	10\\
23	33	10\\
23	34	10\\
23	35	100\\
23	36	100\\
23	37	100\\
23	38	100\\
23	39	100\\
23	40	100\\
23	41	100\\
23	42	100\\
23	43	100\\
23	44	100\\
23	45	100\\
23	46	100\\
23	47	100\\
23	48	100\\
23	49	100\\
23	50	100\\
23	51	100\\
24	1	0\\
24	2	0\\
24	3	0\\
24	4	0\\
24	5	0\\
24	6	0\\
24	7	0\\
24	8	0\\
24	9	0\\
24	10	0\\
24	11	0\\
24	12	0\\
24	13	0\\
24	14	0\\
24	15	0\\
24	16	0\\
24	17	0\\
24	18	0\\
24	19	0\\
24	20	0\\
24	21	0\\
24	22	0\\
24	23	0\\
24	24	10\\
24	25	10\\
24	26	10\\
24	27	10\\
24	28	10\\
24	29	10\\
24	30	10\\
24	31	10\\
24	32	10\\
24	33	10\\
24	34	10\\
24	35	100\\
24	36	100\\
24	37	100\\
24	38	100\\
24	39	100\\
24	40	100\\
24	41	100\\
24	42	100\\
24	43	100\\
24	44	100\\
24	45	100\\
24	46	100\\
24	47	100\\
24	48	100\\
24	49	100\\
24	50	100\\
24	51	100\\
25	1	0\\
25	2	0\\
25	3	0\\
25	4	0\\
25	5	0\\
25	6	0\\
25	7	0\\
25	8	0\\
25	9	0\\
25	10	0\\
25	11	0\\
25	12	0\\
25	13	0\\
25	14	0\\
25	15	0\\
25	16	0\\
25	17	0\\
25	18	0\\
25	19	0\\
25	20	0\\
25	21	0\\
25	22	0\\
25	23	0\\
25	24	10\\
25	25	10\\
25	26	10\\
25	27	10\\
25	28	10\\
25	29	10\\
25	30	10\\
25	31	10\\
25	32	10\\
25	33	10\\
25	34	10\\
25	35	100\\
25	36	100\\
25	37	100\\
25	38	100\\
25	39	100\\
25	40	100\\
25	41	100\\
25	42	100\\
25	43	100\\
25	44	100\\
25	45	100\\
25	46	100\\
25	47	100\\
25	48	100\\
25	49	100\\
25	50	100\\
25	51	100\\
26	1	0\\
26	2	0\\
26	3	0\\
26	4	0\\
26	5	0\\
26	6	0\\
26	7	0\\
26	8	0\\
26	9	0\\
26	10	0\\
26	11	0\\
26	12	0\\
26	13	0\\
26	14	0\\
26	15	0\\
26	16	0\\
26	17	0\\
26	18	0\\
26	19	0\\
26	20	0\\
26	21	0\\
26	22	0\\
26	23	0\\
26	24	10\\
26	25	10\\
26	26	10\\
26	27	10\\
26	28	10\\
26	29	10\\
26	30	10\\
26	31	10\\
26	32	10\\
26	33	10\\
26	34	10\\
26	35	10\\
26	36	100\\
26	37	100\\
26	38	100\\
26	39	100\\
26	40	100\\
26	41	100\\
26	42	100\\
26	43	100\\
26	44	100\\
26	45	100\\
26	46	100\\
26	47	100\\
26	48	100\\
26	49	100\\
26	50	100\\
26	51	100\\
27	1	0\\
27	2	0\\
27	3	0\\
27	4	0\\
27	5	0\\
27	6	0\\
27	7	0\\
27	8	0\\
27	9	0\\
27	10	0\\
27	11	0\\
27	12	0\\
27	13	0\\
27	14	0\\
27	15	0\\
27	16	0\\
27	17	0\\
27	18	0\\
27	19	0\\
27	20	0\\
27	21	0\\
27	22	0\\
27	23	0\\
27	24	10\\
27	25	10\\
27	26	10\\
27	27	10\\
27	28	10\\
27	29	10\\
27	30	10\\
27	31	10\\
27	32	10\\
27	33	10\\
27	34	10\\
27	35	10\\
27	36	100\\
27	37	100\\
27	38	100\\
27	39	100\\
27	40	100\\
27	41	100\\
27	42	100\\
27	43	100\\
27	44	100\\
27	45	100\\
27	46	100\\
27	47	100\\
27	48	100\\
27	49	100\\
27	50	100\\
27	51	100\\
28	1	0\\
28	2	0\\
28	3	0\\
28	4	0\\
28	5	0\\
28	6	0\\
28	7	0\\
28	8	0\\
28	9	0\\
28	10	0\\
28	11	0\\
28	12	0\\
28	13	0\\
28	14	0\\
28	15	0\\
28	16	0\\
28	17	0\\
28	18	0\\
28	19	0\\
28	20	0\\
28	21	0\\
28	22	0\\
28	23	0\\
28	24	10\\
28	25	10\\
28	26	10\\
28	27	10\\
28	28	10\\
28	29	10\\
28	30	10\\
28	31	10\\
28	32	10\\
28	33	10\\
28	34	10\\
28	35	10\\
28	36	100\\
28	37	100\\
28	38	100\\
28	39	100\\
28	40	100\\
28	41	100\\
28	42	100\\
28	43	100\\
28	44	100\\
28	45	100\\
28	46	100\\
28	47	100\\
28	48	100\\
28	49	100\\
28	50	100\\
28	51	100\\
29	1	0\\
29	2	0\\
29	3	0\\
29	4	0\\
29	5	0\\
29	6	0\\
29	7	0\\
29	8	0\\
29	9	0\\
29	10	0\\
29	11	0\\
29	12	0\\
29	13	0\\
29	14	0\\
29	15	0\\
29	16	0\\
29	17	0\\
29	18	0\\
29	19	0\\
29	20	0\\
29	21	0\\
29	22	0\\
29	23	0\\
29	24	10\\
29	25	10\\
29	26	10\\
29	27	10\\
29	28	10\\
29	29	10\\
29	30	10\\
29	31	10\\
29	32	10\\
29	33	10\\
29	34	10\\
29	35	10\\
29	36	100\\
29	37	100\\
29	38	100\\
29	39	100\\
29	40	100\\
29	41	100\\
29	42	100\\
29	43	100\\
29	44	100\\
29	45	100\\
29	46	100\\
29	47	100\\
29	48	100\\
29	49	100\\
29	50	100\\
29	51	100\\
30	1	0\\
30	2	0\\
30	3	0\\
30	4	0\\
30	5	0\\
30	6	0\\
30	7	0\\
30	8	0\\
30	9	0\\
30	10	0\\
30	11	0\\
30	12	0\\
30	13	0\\
30	14	0\\
30	15	0\\
30	16	0\\
30	17	0\\
30	18	0\\
30	19	0\\
30	20	0\\
30	21	0\\
30	22	0\\
30	23	0\\
30	24	10\\
30	25	10\\
30	26	10\\
30	27	10\\
30	28	10\\
30	29	10\\
30	30	10\\
30	31	10\\
30	32	10\\
30	33	10\\
30	34	10\\
30	35	10\\
30	36	100\\
30	37	100\\
30	38	100\\
30	39	100\\
30	40	100\\
30	41	100\\
30	42	100\\
30	43	100\\
30	44	100\\
30	45	100\\
30	46	100\\
30	47	100\\
30	48	100\\
30	49	100\\
30	50	100\\
30	51	100\\
31	1	0\\
31	2	0\\
31	3	0\\
31	4	0\\
31	5	0\\
31	6	0\\
31	7	0\\
31	8	0\\
31	9	0\\
31	10	0\\
31	11	0\\
31	12	0\\
31	13	0\\
31	14	0\\
31	15	0\\
31	16	0\\
31	17	0\\
31	18	0\\
31	19	0\\
31	20	0\\
31	21	0\\
31	22	0\\
31	23	0\\
31	24	10\\
31	25	10\\
31	26	10\\
31	27	10\\
31	28	10\\
31	29	10\\
31	30	10\\
31	31	10\\
31	32	10\\
31	33	10\\
31	34	10\\
31	35	10\\
31	36	100\\
31	37	100\\
31	38	100\\
31	39	100\\
31	40	100\\
31	41	100\\
31	42	100\\
31	43	100\\
31	44	100\\
31	45	100\\
31	46	100\\
31	47	100\\
31	48	100\\
31	49	100\\
31	50	100\\
31	51	100\\
32	1	0\\
32	2	0\\
32	3	0\\
32	4	0\\
32	5	0\\
32	6	0\\
32	7	0\\
32	8	0\\
32	9	0\\
32	10	0\\
32	11	0\\
32	12	0\\
32	13	0\\
32	14	0\\
32	15	0\\
32	16	0\\
32	17	0\\
32	18	0\\
32	19	0\\
32	20	0\\
32	21	0\\
32	22	0\\
32	23	0\\
32	24	10\\
32	25	10\\
32	26	10\\
32	27	10\\
32	28	10\\
32	29	10\\
32	30	10\\
32	31	10\\
32	32	10\\
32	33	10\\
32	34	10\\
32	35	10\\
32	36	100\\
32	37	100\\
32	38	100\\
32	39	100\\
32	40	100\\
32	41	100\\
32	42	100\\
32	43	100\\
32	44	100\\
32	45	100\\
32	46	100\\
32	47	100\\
32	48	100\\
32	49	100\\
32	50	100\\
32	51	100\\
33	1	0\\
33	2	0\\
33	3	0\\
33	4	0\\
33	5	0\\
33	6	0\\
33	7	0\\
33	8	0\\
33	9	0\\
33	10	0\\
33	11	0\\
33	12	0\\
33	13	0\\
33	14	0\\
33	15	0\\
33	16	0\\
33	17	0\\
33	18	0\\
33	19	0\\
33	20	0\\
33	21	0\\
33	22	0\\
33	23	0\\
33	24	10\\
33	25	10\\
33	26	10\\
33	27	10\\
33	28	10\\
33	29	10\\
33	30	10\\
33	31	10\\
33	32	10\\
33	33	10\\
33	34	10\\
33	35	10\\
33	36	100\\
33	37	100\\
33	38	100\\
33	39	100\\
33	40	100\\
33	41	100\\
33	42	100\\
33	43	100\\
33	44	100\\
33	45	100\\
33	46	100\\
33	47	100\\
33	48	100\\
33	49	100\\
33	50	100\\
33	51	100\\
34	1	0\\
34	2	0\\
34	3	0\\
34	4	0\\
34	5	0\\
34	6	0\\
34	7	0\\
34	8	0\\
34	9	0\\
34	10	0\\
34	11	0\\
34	12	0\\
34	13	0\\
34	14	0\\
34	15	0\\
34	16	0\\
34	17	0\\
34	18	0\\
34	19	0\\
34	20	0\\
34	21	0\\
34	22	0\\
34	23	0\\
34	24	10\\
34	25	10\\
34	26	10\\
34	27	10\\
34	28	10\\
34	29	10\\
34	30	10\\
34	31	10\\
34	32	10\\
34	33	10\\
34	34	10\\
34	35	10\\
34	36	10\\
34	37	100\\
34	38	100\\
34	39	100\\
34	40	100\\
34	41	100\\
34	42	100\\
34	43	100\\
34	44	100\\
34	45	100\\
34	46	100\\
34	47	100\\
34	48	100\\
34	49	100\\
34	50	100\\
34	51	100\\
35	1	0\\
35	2	0\\
35	3	0\\
35	4	0\\
35	5	0\\
35	6	0\\
35	7	0\\
35	8	0\\
35	9	0\\
35	10	0\\
35	11	0\\
35	12	0\\
35	13	0\\
35	14	0\\
35	15	0\\
35	16	0\\
35	17	0\\
35	18	0\\
35	19	0\\
35	20	0\\
35	21	0\\
35	22	0\\
35	23	0\\
35	24	10\\
35	25	10\\
35	26	10\\
35	27	10\\
35	28	10\\
35	29	10\\
35	30	10\\
35	31	10\\
35	32	10\\
35	33	10\\
35	34	10\\
35	35	10\\
35	36	10\\
35	37	100\\
35	38	100\\
35	39	100\\
35	40	100\\
35	41	100\\
35	42	100\\
35	43	100\\
35	44	100\\
35	45	100\\
35	46	100\\
35	47	100\\
35	48	100\\
35	49	100\\
35	50	100\\
35	51	100\\
36	1	0\\
36	2	0\\
36	3	0\\
36	4	0\\
36	5	0\\
36	6	0\\
36	7	0\\
36	8	0\\
36	9	0\\
36	10	0\\
36	11	0\\
36	12	0\\
36	13	0\\
36	14	0\\
36	15	0\\
36	16	0\\
36	17	0\\
36	18	0\\
36	19	0\\
36	20	0\\
36	21	0\\
36	22	0\\
36	23	0\\
36	24	10\\
36	25	10\\
36	26	10\\
36	27	10\\
36	28	10\\
36	29	10\\
36	30	10\\
36	31	10\\
36	32	10\\
36	33	10\\
36	34	10\\
36	35	10\\
36	36	10\\
36	37	100\\
36	38	100\\
36	39	100\\
36	40	100\\
36	41	100\\
36	42	100\\
36	43	100\\
36	44	100\\
36	45	100\\
36	46	100\\
36	47	100\\
36	48	100\\
36	49	100\\
36	50	100\\
36	51	100\\
37	1	0\\
37	2	0\\
37	3	0\\
37	4	0\\
37	5	0\\
37	6	0\\
37	7	0\\
37	8	0\\
37	9	0\\
37	10	0\\
37	11	0\\
37	12	0\\
37	13	0\\
37	14	0\\
37	15	0\\
37	16	0\\
37	17	0\\
37	18	0\\
37	19	0\\
37	20	0\\
37	21	0\\
37	22	0\\
37	23	0\\
37	24	10\\
37	25	10\\
37	26	10\\
37	27	10\\
37	28	10\\
37	29	10\\
37	30	10\\
37	31	10\\
37	32	10\\
37	33	10\\
37	34	10\\
37	35	10\\
37	36	10\\
37	37	100\\
37	38	100\\
37	39	100\\
37	40	100\\
37	41	100\\
37	42	100\\
37	43	100\\
37	44	100\\
37	45	100\\
37	46	100\\
37	47	100\\
37	48	100\\
37	49	100\\
37	50	100\\
37	51	100\\
38	1	0\\
38	2	0\\
38	3	0\\
38	4	0\\
38	5	0\\
38	6	0\\
38	7	0\\
38	8	0\\
38	9	0\\
38	10	0\\
38	11	0\\
38	12	0\\
38	13	0\\
38	14	0\\
38	15	0\\
38	16	0\\
38	17	0\\
38	18	0\\
38	19	0\\
38	20	0\\
38	21	0\\
38	22	0\\
38	23	0\\
38	24	10\\
38	25	10\\
38	26	10\\
38	27	10\\
38	28	10\\
38	29	10\\
38	30	10\\
38	31	10\\
38	32	10\\
38	33	10\\
38	34	10\\
38	35	10\\
38	36	10\\
38	37	100\\
38	38	100\\
38	39	100\\
38	40	100\\
38	41	100\\
38	42	100\\
38	43	100\\
38	44	100\\
38	45	100\\
38	46	100\\
38	47	100\\
38	48	100\\
38	49	100\\
38	50	100\\
38	51	100\\
39	1	0\\
39	2	0\\
39	3	0\\
39	4	0\\
39	5	0\\
39	6	0\\
39	7	0\\
39	8	0\\
39	9	0\\
39	10	0\\
39	11	0\\
39	12	0\\
39	13	0\\
39	14	0\\
39	15	0\\
39	16	0\\
39	17	0\\
39	18	0\\
39	19	0\\
39	20	0\\
39	21	0\\
39	22	0\\
39	23	0\\
39	24	10\\
39	25	10\\
39	26	10\\
39	27	10\\
39	28	10\\
39	29	10\\
39	30	10\\
39	31	10\\
39	32	10\\
39	33	10\\
39	34	10\\
39	35	10\\
39	36	10\\
39	37	100\\
39	38	100\\
39	39	100\\
39	40	100\\
39	41	100\\
39	42	100\\
39	43	100\\
39	44	100\\
39	45	100\\
39	46	100\\
39	47	100\\
39	48	100\\
39	49	100\\
39	50	100\\
39	51	100\\
40	1	0\\
40	2	0\\
40	3	0\\
40	4	0\\
40	5	0\\
40	6	0\\
40	7	0\\
40	8	0\\
40	9	0\\
40	10	0\\
40	11	0\\
40	12	0\\
40	13	0\\
40	14	0\\
40	15	0\\
40	16	0\\
40	17	0\\
40	18	0\\
40	19	0\\
40	20	0\\
40	21	0\\
40	22	0\\
40	23	0\\
40	24	10\\
40	25	10\\
40	26	10\\
40	27	10\\
40	28	10\\
40	29	10\\
40	30	10\\
40	31	10\\
40	32	10\\
40	33	10\\
40	34	10\\
40	35	10\\
40	36	10\\
40	37	100\\
40	38	100\\
40	39	100\\
40	40	100\\
40	41	100\\
40	42	100\\
40	43	100\\
40	44	100\\
40	45	100\\
40	46	100\\
40	47	100\\
40	48	100\\
40	49	100\\
40	50	100\\
40	51	100\\
41	1	0\\
41	2	0\\
41	3	0\\
41	4	0\\
41	5	0\\
41	6	0\\
41	7	0\\
41	8	0\\
41	9	0\\
41	10	0\\
41	11	0\\
41	12	0\\
41	13	0\\
41	14	0\\
41	15	0\\
41	16	0\\
41	17	0\\
41	18	0\\
41	19	0\\
41	20	0\\
41	21	0\\
41	22	0\\
41	23	0\\
41	24	10\\
41	25	10\\
41	26	10\\
41	27	10\\
41	28	10\\
41	29	10\\
41	30	10\\
41	31	10\\
41	32	10\\
41	33	10\\
41	34	10\\
41	35	10\\
41	36	10\\
41	37	100\\
41	38	100\\
41	39	100\\
41	40	100\\
41	41	100\\
41	42	100\\
41	43	100\\
41	44	100\\
41	45	100\\
41	46	100\\
41	47	100\\
41	48	100\\
41	49	100\\
41	50	100\\
41	51	100\\
42	1	0\\
42	2	0\\
42	3	0\\
42	4	0\\
42	5	0\\
42	6	0\\
42	7	0\\
42	8	0\\
42	9	0\\
42	10	0\\
42	11	0\\
42	12	0\\
42	13	0\\
42	14	0\\
42	15	0\\
42	16	0\\
42	17	0\\
42	18	0\\
42	19	0\\
42	20	0\\
42	21	0\\
42	22	0\\
42	23	0\\
42	24	10\\
42	25	10\\
42	26	10\\
42	27	10\\
42	28	10\\
42	29	10\\
42	30	10\\
42	31	10\\
42	32	10\\
42	33	10\\
42	34	10\\
42	35	10\\
42	36	10\\
42	37	100\\
42	38	100\\
42	39	100\\
42	40	100\\
42	41	100\\
42	42	100\\
42	43	100\\
42	44	100\\
42	45	100\\
42	46	100\\
42	47	100\\
42	48	100\\
42	49	100\\
42	50	100\\
42	51	100\\
43	1	0\\
43	2	0\\
43	3	0\\
43	4	0\\
43	5	0\\
43	6	0\\
43	7	0\\
43	8	0\\
43	9	0\\
43	10	0\\
43	11	0\\
43	12	0\\
43	13	0\\
43	14	0\\
43	15	0\\
43	16	0\\
43	17	0\\
43	18	0\\
43	19	0\\
43	20	0\\
43	21	0\\
43	22	0\\
43	23	0\\
43	24	10\\
43	25	10\\
43	26	10\\
43	27	10\\
43	28	10\\
43	29	10\\
43	30	10\\
43	31	10\\
43	32	10\\
43	33	10\\
43	34	10\\
43	35	10\\
43	36	10\\
43	37	100\\
43	38	100\\
43	39	100\\
43	40	100\\
43	41	100\\
43	42	100\\
43	43	100\\
43	44	100\\
43	45	100\\
43	46	100\\
43	47	100\\
43	48	100\\
43	49	100\\
43	50	100\\
43	51	100\\
44	1	0\\
44	2	0\\
44	3	0\\
44	4	0\\
44	5	0\\
44	6	0\\
44	7	0\\
44	8	0\\
44	9	0\\
44	10	0\\
44	11	0\\
44	12	0\\
44	13	0\\
44	14	0\\
44	15	0\\
44	16	0\\
44	17	0\\
44	18	0\\
44	19	0\\
44	20	0\\
44	21	0\\
44	22	0\\
44	23	0\\
44	24	10\\
44	25	10\\
44	26	10\\
44	27	10\\
44	28	10\\
44	29	10\\
44	30	10\\
44	31	10\\
44	32	10\\
44	33	10\\
44	34	10\\
44	35	10\\
44	36	10\\
44	37	10\\
44	38	100\\
44	39	100\\
44	40	100\\
44	41	100\\
44	42	100\\
44	43	100\\
44	44	100\\
44	45	100\\
44	46	100\\
44	47	100\\
44	48	100\\
44	49	100\\
44	50	100\\
44	51	100\\
45	1	0\\
45	2	0\\
45	3	0\\
45	4	0\\
45	5	0\\
45	6	0\\
45	7	0\\
45	8	0\\
45	9	0\\
45	10	0\\
45	11	0\\
45	12	0\\
45	13	0\\
45	14	0\\
45	15	0\\
45	16	0\\
45	17	0\\
45	18	0\\
45	19	0\\
45	20	0\\
45	21	0\\
45	22	0\\
45	23	0\\
45	24	10\\
45	25	10\\
45	26	10\\
45	27	10\\
45	28	10\\
45	29	10\\
45	30	10\\
45	31	10\\
45	32	10\\
45	33	10\\
45	34	10\\
45	35	10\\
45	36	10\\
45	37	10\\
45	38	100\\
45	39	100\\
45	40	100\\
45	41	100\\
45	42	100\\
45	43	100\\
45	44	100\\
45	45	100\\
45	46	100\\
45	47	100\\
45	48	100\\
45	49	100\\
45	50	100\\
45	51	100\\
46	1	0\\
46	2	0\\
46	3	0\\
46	4	0\\
46	5	0\\
46	6	0\\
46	7	0\\
46	8	0\\
46	9	0\\
46	10	0\\
46	11	0\\
46	12	0\\
46	13	0\\
46	14	0\\
46	15	0\\
46	16	0\\
46	17	0\\
46	18	0\\
46	19	0\\
46	20	0\\
46	21	0\\
46	22	0\\
46	23	0\\
46	24	10\\
46	25	10\\
46	26	10\\
46	27	10\\
46	28	10\\
46	29	10\\
46	30	10\\
46	31	10\\
46	32	10\\
46	33	10\\
46	34	10\\
46	35	10\\
46	36	10\\
46	37	10\\
46	38	100\\
46	39	100\\
46	40	100\\
46	41	100\\
46	42	100\\
46	43	100\\
46	44	100\\
46	45	100\\
46	46	100\\
46	47	100\\
46	48	100\\
46	49	100\\
46	50	100\\
46	51	100\\
47	1	0\\
47	2	0\\
47	3	0\\
47	4	0\\
47	5	0\\
47	6	0\\
47	7	0\\
47	8	0\\
47	9	0\\
47	10	0\\
47	11	0\\
47	12	0\\
47	13	0\\
47	14	0\\
47	15	0\\
47	16	0\\
47	17	0\\
47	18	0\\
47	19	0\\
47	20	0\\
47	21	0\\
47	22	0\\
47	23	0\\
47	24	10\\
47	25	10\\
47	26	10\\
47	27	10\\
47	28	10\\
47	29	10\\
47	30	10\\
47	31	10\\
47	32	10\\
47	33	10\\
47	34	10\\
47	35	10\\
47	36	10\\
47	37	10\\
47	38	100\\
47	39	100\\
47	40	100\\
47	41	100\\
47	42	100\\
47	43	100\\
47	44	100\\
47	45	100\\
47	46	100\\
47	47	100\\
47	48	100\\
47	49	100\\
47	50	100\\
47	51	100\\
48	1	0\\
48	2	0\\
48	3	0\\
48	4	0\\
48	5	0\\
48	6	0\\
48	7	0\\
48	8	0\\
48	9	0\\
48	10	0\\
48	11	0\\
48	12	0\\
48	13	0\\
48	14	0\\
48	15	0\\
48	16	0\\
48	17	0\\
48	18	0\\
48	19	0\\
48	20	0\\
48	21	0\\
48	22	0\\
48	23	0\\
48	24	10\\
48	25	10\\
48	26	10\\
48	27	10\\
48	28	10\\
48	29	10\\
48	30	10\\
48	31	10\\
48	32	10\\
48	33	10\\
48	34	10\\
48	35	10\\
48	36	10\\
48	37	10\\
48	38	100\\
48	39	100\\
48	40	100\\
48	41	100\\
48	42	100\\
48	43	100\\
48	44	100\\
48	45	100\\
48	46	100\\
48	47	100\\
48	48	100\\
48	49	100\\
48	50	100\\
48	51	100\\
49	1	0\\
49	2	0\\
49	3	0\\
49	4	0\\
49	5	0\\
49	6	0\\
49	7	0\\
49	8	0\\
49	9	0\\
49	10	0\\
49	11	0\\
49	12	0\\
49	13	0\\
49	14	0\\
49	15	0\\
49	16	0\\
49	17	0\\
49	18	0\\
49	19	0\\
49	20	0\\
49	21	0\\
49	22	0\\
49	23	0\\
49	24	10\\
49	25	10\\
49	26	10\\
49	27	10\\
49	28	10\\
49	29	10\\
49	30	10\\
49	31	10\\
49	32	10\\
49	33	10\\
49	34	10\\
49	35	10\\
49	36	10\\
49	37	10\\
49	38	100\\
49	39	100\\
49	40	100\\
49	41	100\\
49	42	100\\
49	43	100\\
49	44	100\\
49	45	100\\
49	46	100\\
49	47	100\\
49	48	100\\
49	49	100\\
49	50	100\\
49	51	100\\
50	1	0\\
50	2	0\\
50	3	0\\
50	4	0\\
50	5	0\\
50	6	0\\
50	7	0\\
50	8	0\\
50	9	0\\
50	10	0\\
50	11	0\\
50	12	0\\
50	13	0\\
50	14	0\\
50	15	0\\
50	16	0\\
50	17	0\\
50	18	0\\
50	19	0\\
50	20	0\\
50	21	0\\
50	22	0\\
50	23	0\\
50	24	10\\
50	25	10\\
50	26	10\\
50	27	10\\
50	28	10\\
50	29	10\\
50	30	10\\
50	31	10\\
50	32	10\\
50	33	10\\
50	34	10\\
50	35	10\\
50	36	10\\
50	37	10\\
50	38	100\\
50	39	100\\
50	40	100\\
50	41	100\\
50	42	100\\
50	43	100\\
50	44	100\\
50	45	100\\
50	46	100\\
50	47	100\\
50	48	100\\
50	49	100\\
50	50	100\\
50	51	100\\
51	1	0\\
51	2	0\\
51	3	0\\
51	4	0\\
51	5	0\\
51	6	0\\
51	7	0\\
51	8	0\\
51	9	0\\
51	10	0\\
51	11	0\\
51	12	0\\
51	13	0\\
51	14	0\\
51	15	0\\
51	16	0\\
51	17	0\\
51	18	0\\
51	19	0\\
51	20	0\\
51	21	0\\
51	22	0\\
51	23	0\\
51	24	10\\
51	25	10\\
51	26	10\\
51	27	10\\
51	28	10\\
51	29	10\\
51	30	10\\
51	31	10\\
51	32	10\\
51	33	10\\
51	34	10\\
51	35	10\\
51	36	10\\
51	37	10\\
51	38	100\\
51	39	100\\
51	40	100\\
51	41	100\\
51	42	100\\
51	43	100\\
51	44	100\\
51	45	100\\
51	46	100\\
51	47	100\\
51	48	100\\
51	49	100\\
51	50	100\\
51	51	100\\
};
\end{axis}
\end{tikzpicture}%